\documentclass[11pt]{aims}
\usepackage{amsmath, amssymb, mathrsfs}
  \usepackage{paralist}
  \usepackage{graphics} 
  \usepackage{epsfig} 
\usepackage{graphicx}  \usepackage{epstopdf}
 \usepackage[colorlinks=true]{hyperref}
 \hypersetup{urlcolor=blue, citecolor=red}
 \usepackage[toc,page]{appendix}
\usepackage{chngcntr}

\usepackage{titlesec}
\titleformat{\section}{\Large\bfseries}{\thesection.}{4pt}{}
\titleformat{\subsection}{\large\bfseries}{\thesection.\arabic{subsection}.}{4pt}{}
\titleformat{\subsubsection}{\bfseries}{\thesection.\arabic{subsection}.\arabic{subsubsection}.}{4pt}{}
\titleformat*{\paragraph}{\bfseries}
\titleformat*{\subparagraph}{\bfseries}
\usepackage[margin=1.2in]{geometry}

  \def\currentyear{}
   \def\currentmonth{}
    \def\ppages{}
     \def\DOI{}

\newtheorem{theorem}{Theorem}[section]

\newtheorem{lemma}[theorem]{Lemma}
\newtheorem{proposition}[theorem]{Proposition}
\theoremstyle{definition}
\newtheorem{definition}[theorem]{Definition}
\newtheorem{remark}[theorem]{Remark}

\newcommand{\RN}{\mathbb{R}^N}
\newcommand{\Rb}{\mathbb{R}}
\newcommand{\Lc}{\mathscr{L}}
\newcommand{\Hc}{\mathscr{H}}
\newcommand{\Vc}{\mathcal{V}}
\newcommand{\Sc}{\mathcal{S}}

\newcommand{\Oc}{\mathcal{O}}

\newcommand{\Cc}{\mathcal{C}}
\newcommand{\Mc}{\mathcal{M}}

\newcommand{\Dc}{\mathcal{D}}
\newcommand{\Gc}{\mathcal{G}}

\newcommand{\Pc}{\mathbb{P}}
\numberwithin{equation}{section}

\title[Blowup solutions for an exponential reaction-diffusion system] 
      {Blowup solutions for a reaction-diffusion system with exponential nonlinearities}

\author[T. Ghoul, V. T. Nguyen, H. Zaag]{}

 \email[T. Ghoul]{teg6@nyu.edu}
 \email[V. T. Nguyen]{Tien.Nguyen@nyu.edu}
 \email[H. Zaag]{Hatem.Zaag@univ-paris13.fr}

\thanks{H. Zaag is supported by the ANR project ANA\'E ref. ANR-13-BS01-0010-03. \\ --------------------\\ 
\today}
\begin{document}
\maketitle

\centerline{\scshape Tej-Eddine Ghoul$^\dagger$, Van Tien Nguyen$^\dagger$ and Hatem Zaag$^\ast$}
\medskip
{\footnotesize
 \centerline{$^\dagger$New York University in Abu Dhabi, P.O. Box 129188, Abu Dhabi, United Arab Emirates.}
  \centerline{$^\ast$Universit\'e Paris 13, Sorbonne Paris Cit\'e, LAGA, CNRS (UMR 7539), F-93430, Villetaneuse, France.}
}

\bigskip

\begin{abstract} We consider the following parabolic system whose nonlinearity has no gradient structure:
$$\left\{\begin{array}{ll}
\partial_t u = \Delta u + e^{pv}, \quad & \partial_t v = \mu \Delta v + e^{qu},\\
u(\cdot, 0) = u_0, \quad & v(\cdot, 0) = v_0,
\end{array}\right.  \quad p, q, \mu > 0,
$$
in the whole space $\RN$.  We show the existence of a stable blowup solution and obtain a complete description of its singularity formation. The construction relies on the reduction of the problem to a finite dimensional one and a topological argument based on the index theory to conclude. In particular, our analysis  uses neither the maximum principle nor the classical methods based on energy-type estimates which are not supported in this system. The stability is a consequence of the existence proof through a geometrical interpretation of the quantities of blowup parameters whose dimension is equal to the dimension of the finite dimensional problem.
\end{abstract}

\section{Introduction.}
In this paper we study the reaction-diffusion system
\begin{equation}\label{PS}
\left\{\begin{array}{ll}
\partial_t u = \Delta u + F(v), \quad & \partial_t v = \mu \Delta v + G(u),\\
u(\cdot, 0) = u_0, \quad & v(\cdot, 0) = v_0,
\end{array}\right. \quad \mu > 0,
\end{equation}
in the whole space $\RN$, where the nonlinearity is of exponential type 
\begin{equation}\label{def:fg}
F(v) = e^{pv}, \quad G(u) = e^{q u}, \quad p, q > 0.
\end{equation}
Our aim is to construct a blowup solution for this system and to precisely describe its blowup profile. We also intend to show the stability of our solution with respect to initial data. 

\bigskip

For the expert reader, we would like to immediately present our motivations in considering such a system. For other readers, we will take the time to present the history of the models, hoping that our motivations will gradually become more accessible to him.

\subsection{Our motivations for the expert reader.}
In fact, our motivation in this work is double:\\
\noindent $\bullet$ {\it Physical motivation}: this is an ignition model for thermal explosions of two mixed solid fuels of finite extent, as one may see from Bebernes, Bressan and Eberly \cite{BBEiumj87} as well as Bebernes and Kassoy \cite{BKsjam81}, cited by Zheng, Zhao and Chen \cite{ZZCna02}. In this model, $u$ and $v$ stand for the temperatures of the two fuels.\\ 
\noindent $\bullet$ {\it Mathematical motivation}: We acknowledge that our argument uses the method introduced by Bressan \cite{Breiumj90}, Bricmont and Kupiainen \cite{BKnon94} and Merle and Zaag \cite{MZdm97} for the scalar semilinear heat equation with exponential or power nonlinearity. That method is based on 3 steps:
\begin{itemize}
\item[-] the linearization of the equation around the intended profile;
\item[-] the reduction of the problem to a finite-dimensional one, corresponding to the positive eigenvalues, thanks to the control of the negative directions of the spectrum with the properties of the linear operator;
\item[-] the solution of the finite-dimensional problem thanks to Brouwer's lemma and the degree theory.
\end{itemize}
Nevertheless, the case of our system \eqref{PS} is much tougher than the mentioned scalar cases, at least for two reasons, which means that our analysis in this paper is far from being a simple adaptation of the arguments introduced in the scalar case, making our interest in \eqref{PS} completely meaningful. These are the two reasons:
\begin{itemize}
\item[-] first, we have here a system and not just a scalar equation, with different diffusivities between the components ($\mu$ may or may not be equal to $1$), which makes the above-mentioned linear operator neither self-adjoint nor diagonal. Some additional spectral arguments are therefore needed;
\item[-] second, the intended profile for the solution is unbounded in the space variable, as one may see from the statement of Theorem \ref{theo1} below, where we see that 
\[
(u,v) \sim (-\log(T-t) + \log\Phi^*, - \log(T-t) + \log\Psi^*),
\]  
with $\Phi^*$ and $\Psi^*$ introduced in \eqref{def:PhiPsistarpro}. This makes it difficult to control the solution in the intermediate zone, between the blow-up and the regular zones. Thanks to the introduction of  $U = e^{qu}$ and $V = e^{pv}$, we make the profile bounded, at the expense of adding two terms unknown in the scalar case, namely $|\nabla U|^2/U$ and $|\nabla V|^2/V$ (see \eqref{sys:PhiPsi_i} below). These terms are delicate, since both upper bound and lower bound are needed; especially when $U$ and $V$ become smaller in the intermediate and regular zones. In order to treat them, we introduce estimates of the solution in a 3-fold shrinking set (see Definition \ref{def:St} below), where the control in the blow-up zone is inspired by the scalar case, hence not new, whereas the control in the intermediate region is one of the novelties of our paper. 
\end{itemize}
More details are given below in the introduction and in the following sections.

\subsection{Previous literature and statement of the results.}

The local Cauchy problem for \eqref{PS} can be solved in several functional spaces $\mathcal{F}$, for example $\mathcal{F} = L^\infty(\RN) \times L^\infty(\RN)$ or in a special affine space $\mathcal{F} = \mathcal{H}_a$ for some positive constant $a$ with
\begin{equation}\label{def:Ha}
\mathcal{H}_a = \{(u,v) \in (\bar \phi, \bar \psi) + L^\infty(\RN) \times L^\infty(\RN)\;\; \text{with} \;\; q\bar \phi = p\bar \psi = -\ln(1 + a|x|^2)\}.
\end{equation}
We denote by $T = T(u_0, v_0) \in (0, +\infty]$ the maximal existence time of the classical solution $(u,v)$ of problem \eqref{PS}. If $T < +\infty$, then the solution blows up in finite time $T$ in the sense that 
$$\lim_{t \to T}(\|u(t)\|_{L^\infty(\RN)} + \|v(t)\|_{L^\infty(\RN)})= +\infty.$$
In that case, $T$ is called the blowup time of the solution. A point $a \in \RN$ is said to be a blowup point of $(u,v)$ if $(u,v)$ is not locally bounded near $(a,T)$ in the sense that $|u(x_n, t_n)| + |v(x_n, t_n)| \to +\infty$ for some sequence $(x_n,t_n) \to (a, T)$ as $n \to +\infty$. We say that the blowup is \textit{simultaneous} if 
\begin{equation}\label{def:simultaneous}
\limsup_{t \to T}\|u(t)\|_{L^\infty(\RN)} = \limsup_{t \to T}\|v(t)\|_{L^\infty(\RN)} = +\infty,
\end{equation}
and that it is \textit{non-simultaneous} if \eqref{def:simultaneous} does not hold, i.e. if one of the two components remains bounded on $\RN \times [0,T)$. For the system \eqref{PS}, it is easy to see that the blowup is always \textit{simultaneous}. Indeed, if $u$ is uniformly bounded on $\RN \times [0, T)$, then the second equation would yield a uniform bound on $v$. More specifically, we say that $u$ and $v$ blow up simultaneously at the same point $a \in \RN$ if $a$ is a blowup point both for $u$ and $v$. \\

\medskip

When system \eqref{PS} is coupled with power nonlinearities of the type
\begin{equation}\label{def:fg2}
F(v) = |v|^{p-1}v, \quad G(u) = |u|^{q-1}u,
\end{equation}
and the diffusion coefficient $\mu = 1$, Escobedo and Herrero \cite{EHjde91} (see also \cite{EHpams91}, \cite{EHampa93}) showed that any nontrivial positive solution which is defined for all $x \in \RN$ must simultaneously blow up in finite time, provided that
$$pq > 1, \quad \text{and} \quad \frac{\max\{p,q\} + 1}{pq - 1} \geq \frac{N}{2}.$$
The authors in \cite{AHVihp97} proved that if
\begin{equation} \label{eq:condAHV}
pq > 1, \quad \text{and} \quad q(pN - 2)_+ < N + 2 \quad \text{or} \quad p(qN-2)_+< N+2,
\end{equation}
then every positive solution $(u,v)$ of system \eqref{PS} exhibits the \textit{Type I} blowup, i.e. there exists some constant $C > 0$ such that
\begin{equation}\label{eq:blrate}
\|u(t)\|_{L^\infty(\RN)} \leq C\bar{u}(t), \quad \|v(t)\|_{L^\infty(\RN)} \leq C\bar{v}(t),
\end{equation}
where $(\bar u, \bar v)$ is the blowup positive solution of the associated ordinary differential system
$$\bar{u}(t) = \Gamma (T-t)^{-\frac{p+1}{pq-1}}, \quad \bar{v}(t) = \gamma(T-t)^{-\frac{q+1}{pq - 1}}$$
and 
\begin{equation}\label{def:Gamgam}
\gamma^p = \Gamma\left(\frac{p+1}{pq - 1}\right),\quad \Gamma^q = \gamma\left(\frac{q + 1}{pq - 1}\right).
\end{equation}
The estimate \eqref{eq:blrate} has been proved by Caristi and Mitidieri \cite{CMjde94} in a ball under assumptions on $p$ and $q$ different from \eqref{eq:condAHV}. See also Fila and Souplet \cite{FSnodea01}, Deng \cite{Dzamp96} for other results relative to estimate \eqref{eq:blrate}.

Through the introduction of the following similarity variables for all $a \in \RN$ ($a$ may or may not be a blowup point):
\begin{equation}\label{def:simiVars}
\begin{array}{c}
\Phi_{T,a}(y,s) = (T-t)^{\frac{p+1}{pq -1}}u(x,t), \quad \Psi_{T,a}(y,s) = (T-t)^{\frac{q + 1}{pq-1}}v(x,t),\\
\\
\text{where}\quad y = \dfrac{x-a}{\sqrt{T - t}}, \quad s = -\ln(T-t),
\end{array}
\end{equation}
Andreucci, Herrero and Vel\'azquez \cite{AHVihp97} (recall that we are considering the case when $\mu = 1$) showed that if the solution $(u,v)$ exhibits \textit{Type I} blowup, then one of the following cases occurs (up to replacing $(u,v)$ by $(-u,-v)$ if necessary):\\
\noindent $\bullet$ either $(\Phi_{T,a}, \Psi_{T,a})$ goes to $(\Gamma, \gamma)$ exponentially fast, \\
\noindent $\bullet$ or there exists $k \in \{1, \cdots, N\}$ such that after an orthogonal change of space coordinates, 
\begin{equation}\label{eq:beh}
\begin{array}{ll}
\Phi_{T,a}(y,s) &= \Gamma - \dfrac{c_1}{s}(p+1)\Gamma\sum \limits_{i=1}^k(y_i^2 - 2) + o\left(\dfrac{1}{s}\right),\\
&\\
\Psi_{T,a}(y,s) &= \gamma - \dfrac{c_1}{s}(q+1)\gamma\sum\limits_{i=1}^k(y_i^2 - 2) + o\left(\dfrac{1}{s}\right),
\end{array}
\end{equation}
where $(\Gamma, \gamma)$ is defined by \eqref{def:Gamgam}, $c_1 = c_1(p,q) > 0$, and the convergence holds in $\mathcal{C}^\ell_{loc}(\RN)$ for any $\ell \geq 0$. \\
It is worth mentioning the work of \cite{Zcpam01} where the author obtained a Liouville theorem for system \eqref{PS} with the nonlinearity \eqref{def:fg2} and $\mu = 1$ that improves the result in \cite{AHVihp97}. Based on this theorem, he was able to derive sharp estimates of asymptotic behaviors as well as a localization property for blowup solutions.\\

When system \eqref{PS} is considered with the nonlinearity \eqref{def:fg2} and the diffusion coefficient $\mu > 0$ (not necessarily equal to 1), Mahmoudi, Souplet and Tayachi \cite{MSTjde15} (see also Souplet \cite{Sjems09}) prove the single point blowup for any radially decreasing, positive and classical solution in a ball. This result improves a result by Friedman and Giga \cite{FGjut87} where the method requires a very restrictive conditions $p = q$ and $\mu = 1$ in order to apply the maximum principle to suitable linear combination of the components $u$ and $v$. The authors of \cite{MSTjde15} also derive the lower pointwise estimates on the final blowup profiles:
\begin{equation}\label{est:MST}
\text{for all}\; 0< |x|\leq \epsilon_1, \quad |x|^{\frac{2(p+1)}{pq-1}}u(T,x) \geq \epsilon_0 \quad \text{and} \quad |x|^{\frac{2(q+1)}{pq-1}}v(T,x) \geq \epsilon_0,
\end{equation}
for some $\epsilon_0, \epsilon_1 > 0$. Recently, we establish in \cite{GNZpre16c} the existence of finite time blowup solutions verifying the asymptotic behavior \eqref{eq:beh}. In particular, we exhibit stable finite time blowup solutions according to the dynamics:
\begin{equation*}\begin{array}{l}
u(x,t) \sim \Gamma \left[(T-t)\left(1 + \frac{b|x|^2}{(T-t)|\ln(T-t)|}\right) \right]^{-\frac{p+1}{pq-1}},\\
v(x,t) \sim \gamma \left[(T-t)\left(1 + \frac{b|x|^2}{(T-t)|\ln(T-t)|}\right) \right]^{-\frac{q+1}{pq-1}},
\end{array} \quad \text{as} \quad t \to T,
\end{equation*}
where $\Gamma, \gamma$ and $b$ are positive constants depending on $p, q, \mu$. Moreover, we derive the following sharp description of the final blowup profiles:
\begin{equation}
u(T, x) \sim \Gamma \left(\frac{b|x|^2}{2|\ln|x||} \right)^{-\frac{p+1}{pq - 1}} \quad \text{and} \quad  v(T, x) \sim \gamma \left(\frac{b|x|^2}{2|\ln|x||} \right)^{-\frac{q+1}{pq - 1}} \quad \text{as} \quad |x| \to 0. 
\end{equation}
The method we used in \cite{GNZpre16c} is an extension of the technique developed by Merle and Zaag \cite{MZdm97} treated for the standard semilinear heat equation 
\begin{equation} \label{eq:ScalarE}
\partial_t u = \Delta u + |u|^{p-1}u.
\end{equation}
The analysis in \cite{MZdm97} is mainly based on the spectral property of the linearized operator of the form 
$$\Lc = \Delta - \frac{1}{2}y \cdot \nabla + \textup{Id},$$
whose spectrum has two positive eigenvalues, a null and then infinity many negative eigenvalues. In particular, the method relies on a two step procedure:\\
- The reduction of the problem to a finite dimensional one. This means that controlling the solution in the similarity variables \eqref{def:simiVars} around the profile reduces to the control of the components corresponding to the two positive eigenvalues.\\
- Solving the finite dimensional problem thanks to a topological argument based on index theory.\\

\medskip

As for system \eqref{PS} with the nonlinearity \eqref{def:fg}, much less result is known, in particular in the study of the asymptotic behavior of the solution near singularities. To our knowledge, there are no results concerning the blowup behavior even when $\mu = 1$. The only known result is due to Souplet and Tayachi \cite{STna16} who follow the strategy of \cite{MSTjde15} to establish the single point blowup for large classes of radially decreasing solutions. A similar single point blowup result was obtained in \cite{FGjut87} under the restrictive condition $\mu = 1$. In this paper we exhibit finite time blowup solutions for system \eqref{PS} coupled with \eqref{def:fg} and obtain the first complete description of its blowup behavior. More precisely, we prove the following result.

\begin{theorem}[Existence of blowup solutions for system \eqref{PS} with the complete description of its profile]\label{theo1} Consider $a \in \RN$. There exists $T > 0$ such that system \eqref{PS} has a solution $(u,v)$ defined on $\RN \times [0, T)$ such that:\\
\noindent $(i)$ $e^{qu}$ and $e^{pv}$ blow up in finite time $T$ simultaneously at only one blowup point $a$.\\
\noindent $(ii)$  
\begin{equation}\label{eq:asyTh1}
\left\|(T-t)e^{qu(x,t)} - \Phi^*(z)\right\|_{L^\infty(\RN)} + \left\|(T-t)e^{pv(x,t)} - \Psi^*(z)\right\|_{L^\infty(\RN)} \leq \dfrac{C}{\sqrt{|\ln (T-t)|}},
\end{equation}
where $z = \dfrac{x - a}{\sqrt{(T-t)|\ln (T-t)|}}$ and the profiles are given by
\begin{equation}\label{def:PhiPsistarpro}
p\Phi^*(z) = q\Psi^*(z) =  \left(1 + b|z|^2\right)^{-1} \quad \text{with} \quad b = \frac{1}{2(\mu + 1)}.
\end{equation}

\noindent $(iii)\;$ for all $x \neq a$, $(u(x,t), v(x,t)) \to (u^*(x), v^*(x)) \in \Cc^2(\RN \backslash \{0\}) \times \Cc^2(\RN \backslash \{0\})$ with
$$u^*(x)  \sim \frac{1}{q}\ln \left(\frac{2b}{p}\frac{|\ln|x - a||}{|x-a|^2}\right) \quad\text{and} \quad v^*(x)  \sim \frac{1}{p}\ln \left(\frac{2b}{q}\frac{|\ln|x - a||}{|x-a|^2}\right) \quad \text{as} \quad |x - a| \to 0.$$
\end{theorem}

\bigskip

\begin{remark} The blowup profile \eqref{def:PhiPsistarpro} is formally derived through a matching asymptotic expansion in Section \ref{sec:formal} below. We would like to emphasis that the derivation of the blowup profile \eqref{def:PhiPsistarpro} is not obvious due to numerous parameters of the problem, in particular in precising the value $b = \frac{1}{2(\mu + 1)}$ which is crucial in various algebraic identities in our analysis.
\end{remark}
\begin{remark} When $p = q = \mu = 1$ and $v = u$, system \eqref{PS} reduces to the single equation 
\begin{equation}\label{eq:eu}
\partial_t u = \Delta u + e^u.
\end{equation}
Theorem \ref{theo1} obviously yields the existence of finite time blowup solution to equation \eqref{eq:eu} according the dynamic  
\begin{equation}\label{eq:usim}
u(x,t) \sim -\ln(T-t) - \ln\left(1 + \frac{|x|^2}{4(T-t)|\ln(T-t)|}\right) \quad \text{as} \quad t \to T,\end{equation} 
which covers the results obtained by Bressan \cite{Breiumj90, Brejde92} and the authors \cite{GNZjde17}. It is worth remarking that the asymptotic behavior \eqref{eq:usim} is different from the one obtained by Pulkkinen \cite{PULmmas11} (see also Fila-Pulkkinen \cite{FPtmj08}) where the authors concern non constant self-similar ones for a class of radially symmetric $L^1$-solutions.
\end{remark}

The proof of Theorem \ref{theo1} follows the strategy developed in \cite{MZdm97} and \cite{BKnon94} for the standard semilinear heat equation \eqref{eq:ScalarE}. This method has been successfully implemented in \cite{GNZpre16c} for constructing blowup solutions for system \eqref{PS} coupled with the nonlinearity \eqref{def:fg2}. One may think that the implementation in \cite{GNZpre16c} should work the same for system \eqref{PS} coupled with \eqref{def:fg}, perhaps with some technical complications. This is not the case, because the method we present here is not based on a simple perturbation of system \eqref{PS}-\eqref{def:fg2} treated in \cite{GNZpre16c} as explained shortly. 

It is worth mentioning that the method of \cite{MZdm97} has been also proved to be successful for constructing a solution to some partial differential equation with a prescribed behavior.  It was the case of the complex Ginzburg-Landau equation with no gradient structure by Masmoudi and Zaag \cite{MZjfa08} (see also the earlier work by Zaag \cite{ZAAihn98}) and Nouaili and Zaag \cite{NoZpre17}; by Nguyen and Zaag \cite{NZsns16}, \cite{NZens16} for a logarithmically perturbed nonlinear heat equation and for a refined blowup profile for equation \eqref{eq:ScalarE}, or by Nouaili and Zaag \cite{NZcpde15} for a non-variational complex-valued semilinear heat equation. It was also the case of a non-scaling invariant semilinear heat equation with a general nonlinearity treated in \cite{DNZpre17}, and the energy supercritical harmonic heat flow and wave maps by Ghoul, Ibrahim and Nguyen \cite{IGN16, GINpre17}.  Surprisingly enough, this kind of method is also applicable for the construction of multi-solitons for the semilinear wave equation in one space dimension by C\^ote and Zaag \cite{CZcpam13}, where the authors first show that controlling the similarity variables version around some expected behavior reduces to the control of a finite number of unstable directions, then use the same topological argument to solve the finite dimensional problem.

\bigskip

As in \cite{MZdm97} and \cite{GNZpre16c} (see also \cite{TZpre15}, \cite{MZjfa08}), it is possible to make the interpretation of the finite-dimensional variable in terms of the blowup time and the blowup point. This allows us to derive the stability of the profile $(\Phi^*, \Psi^*)$ in Theorem \ref{theo1} with respect to perturbations of the initial data. More precisely, we have the following:
\begin{theorem}[Stability of the blowup profile \eqref{eq:asyTh1}] \label{theo:2}
Let us denote by $(\hat{u}, \hat{v})(x,t)$ the solution constructed in Theorem \ref{theo1} and by $\hat{T}$ its blowup time. Then, there exists a neighborhood $\mathcal{V}_0$ of $(\hat{u}, \hat v)(x,0)$ in $\mathcal{H}_{a}$ defined in \eqref{def:Ha} such that for any $(u_0, v_0) \in \mathcal{V}_0$, system \eqref{PS} has a unique solution $(u,v)(x, t)$ with initial data $(u_0, v_0)$, and $(u,v)(x,t)$ blows up in finite time $T(u_0,v_0)$ at point $a(u_0, v_0)$. Moreover, estimates given in Theorem \ref{theo1} are satisfied by $(u,v)(x-a,t)$ and
$$T(u_0, v_0) \to \hat{T}, \quad a(u_0, v_0) \to 0 \quad \text{as $(u_0,v_0) \to (\hat{u}_0, \hat v_0)$ in $\mathcal{H}_{a}$}.$$
\end{theorem}
\begin{remark} The basic idea behind the stability proof is roughly understood as follows: The linearized operator $\Hc + \Mc$ (see \eqref{def:HLM} for its definition) has two positive eigenvalues $\lambda_0 = 1, \lambda_1 = \frac{1}{2}$, a zero eigenvalue $\lambda_2 = 0$, then an infinitely discrete negative spectrum (see Lemma \ref{lemm:diagonal} below). As usual in the analysis of stability of blowup problems, the component corresponding to $\lambda_0 = 1$ has the exponential growth $e^{s}$, which can be eliminated by means of change of the blowup time; and the component corresponding to $\lambda_1 = \frac{1}{2}$ has the growth $e^{s/2}$ can be eliminated by means of a shifting of the blowup point. As for the neutral, non exponential growth corresponding to $\lambda_2 = 0$, it can be also eliminated as well after a suitable use of the scaling dilation invariance associated to the problem. Hence, the contribution associated to these three modes of the linearized problem can 
 be assumed to be zero. Since the remaining modes  of the linearized problem corresponding to negative spectrum decay exponentially, one derives the stable asymptotic behavior of the corresponding blowup  mechanism. 

We will not give the proof of Theorem \ref{theo:2} because the stability result follows from the reduction to a finite dimensional case as in \cite{MZdm97} (see Theorem 2 and its proof in Section 4)  with the same argument. Here, we only prove the existence result (Theorem \ref{theo1}) and kindly refer the reader to \cite{MZdm97} and \cite{GNZpre16c} for a similar proof of the stability. 
\end{remark}

\subsection{Strategy of the proof of Theorem \ref{theo1}.}
Let us explain in the following the main steps of the proof of Theorem \ref{theo1}. For clearness, we divide our explanation in 3 paragraphs below:  \\
- The linearized problem;\\
- The properties of the linearized operator;\\
- The decomposition of the solution and the control of the nonlinear gradient terms.\\

\noindent \textit{(i) The linearized problem.}  Let us start with the change of variables
\begin{equation}\label{def:simvariables_i}
\arraycolsep=1.6pt\def\arraystretch{2}
\left\{\begin{array}{l} 
\Phi(y,s) = (T-t)e^{q u(x,t)}, \quad \Psi(y,s) = (T-t)e^{pv(x,t)},\\
\text{where} \quad y = \dfrac{x}{\sqrt {T-t}}, \quad s = -\ln(T-t),
\end{array}\right.
\end{equation}
which transforms system \eqref{PS} to 
\begin{equation} \label{sys:PhiPsi_i}
\arraycolsep=1.6pt\def\arraystretch{2}
\left\{\begin{array}{l}
\partial_s \Phi = \Delta \Phi - \dfrac{1}{2}y \cdot \nabla \Phi - \Phi  + q \Phi \Psi- \dfrac{|\nabla \Phi|^2}{\Phi},\\
\partial_s \Psi = \mu\Delta \Psi - \dfrac{1}{2}y \cdot \nabla \Psi - \Psi  + p \Phi \Psi- \mu\dfrac{|\nabla \Psi|^2}{\Psi},\end{array}\right.
\end{equation}
(in comparison with the work of \cite{GNZpre16c} treated for the case where system \eqref{PS} is considered with the nonlinearity \eqref{def:fg2}, we have extra nonlinear gradient terms in \eqref{sys:PhiPsi_i}, which come from the nonlinear transformation \eqref{def:simvariables_i}; the nonlinear gradient terms  
are the main sources causing serious difficulties in the analysis). The problem then reduces to construct for \eqref{sys:PhiPsi_i} a solution $(\Phi, \Psi)$ defined for all $(y,s) \in \Rb^N \times [s_0, +\infty)$ such that 
\begin{equation}\label{eq:asymPro_i}
\left\|\Phi(y,s) - \Phi^*\left(\frac{y}{\sqrt{s}}\right) \right\|_{L^\infty(\Rb^N)} + \left\|\Psi(y,s) - \Psi^*\left(\frac{y}{\sqrt{s}}\right) \right\|_{L^\infty(\Rb^N)} \longrightarrow 0,  
\end{equation}
as $s \to +\infty$. One may think that it is natural to linearize system \eqref{sys:PhiPsi_i} around $(\Phi^*, \Psi^*)$, however, the error generated by this approximate profile is too large to allow us to close estimates in our analysis. Following the formal approach given in Section \ref{sec:formal} below, the good approximate profile is given by
\begin{equation}\label{def:phipsi_i}
\phi(y,s) = \Phi^*\left(\frac{y}{\sqrt{s}}\right) + \frac{\mu}{p(1 + \mu)s} \quad \text{and} \quad \psi(y,s) = \Psi^*\left(\frac{y}{\sqrt{s}}\right) + \frac{1}{q(1 + \mu)s},
\end{equation}
where the term of order $\frac{1}{s}$ appears as a corrective term to minimize the generated error. We then introduce 
\begin{equation}
\Lambda = \Phi - \phi \quad \text{and} \quad \Upsilon = \Psi - \psi,
\end{equation}
leading to the system 
\begin{equation}\label{sys:LUp_i}
\partial_s\binom{\Lambda}{\Upsilon} = \Big(\Hc + \Mc + V(y,s)\Big) \binom{\Lambda}{\Upsilon} + \binom{q}{p}\Lambda \Upsilon + \binom{R_1}{R_2} + \binom{G_1}{G_2}, 
\end{equation}
where 
\begin{equation}\label{def:HLM}
\Hc = \left(\begin{matrix}
\Lc_1 & 0\\ 0 &\Lc_\mu
\end{matrix} \right), \quad \Mc = \left(\begin{matrix}
0 & \frac{q}{p}\\ \frac pq &0
\end{matrix} \right), \quad \Lc_\eta = \eta\Delta - \frac{1}{2}y \cdot \nabla, 
\end{equation}
\begin{equation*}
V(y,s) = \begin{pmatrix} 
q\psi - 1 & \quad q\big(\phi - 1/p\big)\\ p\big(\psi - 1/q\big) & \quad p\phi - 1
\end{pmatrix},
\end{equation*}
the term $\binom{R_1}{R_2}$ is the generated error which is uniformly bounded by $\frac{C}{s}$ (see definition \eqref{def:Rys} and Lemma \ref{lemm:expandR1R2} below), the nonlinear gradient term $\binom{G_1}{G_2}$ is built to be quadratic (see definition \eqref{def:Gys} and Lemma \ref{lemm:expG1G2} below). \\

\noindent \textit{(ii) The properties of the linearized operator.} As we will see in Section \ref{sec:formu} below, the key step towards Theorem \ref{theo1} is the construction of a solution $(\Lambda, \Upsilon)$ for system \eqref{sys:LUp_i} defined for all $(y,s) \in \Rb^N \times [s_0, +\infty)$ such that 
$$\|\Lambda(s)\|_{L^\infty(\Rb^N)} + \|\Upsilon(s)\|_{L^\infty(\Rb^N)} \to 0 \quad \text{as} \quad s \to +\infty. $$
In view of system \eqref{sys:LUp_i}, we see that the nonlinear terms and the generated error are small and can be negligible in comparison with the linear term. Therefore, the linear part will play an important role in the dynamic of the solution. As we show in Lemma \ref{lemm:diagonal} below, the linearized operator $\Hc + \Mc$ can be diagonalizable and its spectrum is explicitly given by 
$$\text{spec}\Big( \Hc + \Mc \Big) = \left\{\pm 1 - \frac{n}{2}, n \in \mathbb{N}\right\}.$$ 
Depending on the asymptotic behavior of the potential term $V$, the full linear part has two fundamental properties:\\
- For $|y| \leq K_0 \sqrt{s}$ for some $K_0$ large, the potential term is considered as a perturbation of the effect of $\Hc + \Mc$.\\ 
- For $|y| \geq K_0\sqrt s$, the linear operator behaves as an operator with fully negative spectrum, which gives the decay of the solution in this region.\\

\textit{(iii) The decomposition of the solution and the control of the nonlinear gradient terms.} While the control of the flow in the region $|y| \geq K_0\sqrt{s}$ is easy, it is not the case in the inner region, i.e. when $|y| \leq K_0 \sqrt{s}$. Moreover, the nonlinear gradient terms appearing in \eqref{sys:LUp_i} cause serious difficulties in the analysis, and crucial modifications are needed in comparison with the proof in \cite{MZdm97} and \cite{GNZpre16c}. The essential idea in our approach is that we introduce estimates in three regions in different variable scales, inspired by  the works of \cite{MZnon97} and \cite{GNZjde17}, as follows:\\

- In the \textit{blowup region} $\Dc_1 = \{|x| \leq K_0 \sqrt{(T-t)|\ln(T-t)|}\}$, we carry on  our analysis in the similarity variables setting. In particular, the solution $(\Lambda, \Upsilon)$ is decomposed according to the eigenfunctions of $\Hc + \Mc$,
$$\binom{\Lambda}{\Upsilon} = \sum_{n = 0}^2 \theta_n \binom{f_n}{g_n} + \binom{\Lambda_-}{\Upsilon_-},$$
where $\binom{f_n}{g_n}$ is the eigenfunction of $\Hc + \Mc$ corresponding to the eigenvalue $\lambda_n = 1  -\frac{n}{2}$; and $\binom{\Lambda_-}{\Upsilon_-}$ is the projection of $\binom{\Lambda}{\Upsilon}$ on the subspace of $\Hc + \Mc$ where the spectrum of $\Hc + \Mc$ is strictly negative. Since the spectrum of the linear part of system satisfied by $(\Lambda_-, \Upsilon_-)$ (see \eqref{eq:Lamneg} below) is negative, it is controllable to zero.\\
The control of $\theta_2$ is delicate. In fact, we need to refine the asymptotic behavior of the potential term $V(y,s)$ and the nonlinear gradient term $\binom{G_1}{G_2}$ in \eqref{sys:LUp_i} to find that 
$$\theta_2' = -\frac{2}{s}\theta_2 + \Oc\left(\frac{1}{s^3}\right),$$
which shows a negative spectrum (in the slow variable $\tau = \ln s$), hence, it is controllable to zero as well. Here, we want to remark that the factor $-\frac{2}{s}$  and the error $\frac{1}{s^3}$ are derived thanks to the linearization of system \eqref{sys:PhiPsi_i} around the approximate profile $(\phi, \psi)$ defined in \eqref{def:phipsi_i} with the precise value of the constant $b$ introduced in Theorem \ref{theo1}.\\
As for the control of the positive modes $\theta_0$ and $\theta_1$ (\textit{reduction to a finite dimensional problem}), we use a basic topological argument to show the existence of initial data $(u_0, v_0)$ depending on $(N+1)$ parameters (see definition \eqref{def:uvt0} below) such that the corresponding modes $\theta_0$ and $\theta_1$ are controlled to zero. \\

- In the \textit{intermediate region} $\Dc_2 = \{x \vert \; K_0/4\sqrt{(T-t)|\ln(T-t)|} \leq |x| \leq \epsilon_0 \}$, we use classical parabolic regularity estimates on $(\tilde{u}, \tilde{v})$, a rescaled version of $(u,v)$ (see definition \ref{def:uvtilde} below). Roughly speaking, we show that in this region the solution behaves like the solution of the associated ordinary differential system to \eqref{PS}. The analysis in this region also gives the final blowup profile as described in part $(iii)$ of Theorem \ref{theo1}.

- In the \textit{regular region} $\Dc_3 = \{|x| \geq \epsilon_0/4\}$, we  directly control the solution thanks to the local in time well-posedness of the Cauchy problem for system \eqref{PS}.\\

\noindent We would like to remark that in \cite{MZdm97} and \cite{GNZpre16c}, the authors introduce the estimates in the region $|y| \leq K_0\sqrt s$ and the regular region $|y| \geq K_0\sqrt s$. However, the estimates in the region $|y| \geq K_0\sqrt s$ imply the smallness of  $(\Lambda, \Upsilon)$ only, and do not allow any control of the nonlinear gradient terms in this region. In other words, the analysis based on the method of \cite{MZdm97} and \cite{GNZpre16c}, that is to estimate the solution in the $z = \frac{y}{\sqrt s}$ variable is not sufficient and must be improved. By introducing additional estimates in the regions $\Dc_2$ and $\Dc_3$, we are able to achieve the full control of the nonlinear gradient term, then, complete the proof of Theorem \ref{theo1}. \\

\medskip

\noindent The rest of the paper is organized as follows:\\
- In section \ref{sec:forap}, we first derive the basic properties of the linearized operator $\Hc + \Mc$, then, we give a formal explanation on the derivation of the blowup profile $(\Phi^*, \Psi^*)$ by means of the spectral analysis. This formal approach also gives an approximate profile to be linearized around. \\
- In Section \ref{sec:Existence}, we give the main arguments of the proof of Theorem \ref{theo1} and postpone most of technicalities to next sections. Interested readers can find in Subsection \ref{sec:St} a particular definition of a shrinking set to trap the solution of \eqref{PS} according to the blowup regime described in Theorem \ref{theo1}. They also find a basic topological argument for the finite dimensional problem at page \pageref{step2:topoarg}. \\
- In Section \ref{sec:reduction}, we give the proof of Proposition \ref{prop:redu}, which gives the reduction of the problem to a finite dimensional one. This is the central part in the proof of Theorem \ref{theo1}.

\section{A formal approach through a spectral analysis of the linearized operator.}\label{sec:forap}
In this section we follow the idea of Bricmont and Kupiaien \cite{BKnon94} treated for the semilinear heat equation in order to formally derive the blowup profile described in \eqref{eq:asyTh1}. The argument is mainly based on a spectral analysis of the linearized operator and a matching asymptotic expansion.

\subsection{Spectral properties of the linearized operator.} \label{sec:SpecHM}

In this part we recall some well-known properties of the linear operator $\Lc_\eta$ from which we derive spectral properties of the linear operator $\Hc + \Mc$ introduced in \eqref{def:HLM}.  

\noindent $\bullet$ \textbf{Spectral properties of $\Lc_\eta$:} Let $\eta > 0$, we consider the weighted space $L^2_{\rho_\eta}(\RN, \Rb)$ the set of all $f \in L^2_{loc}(\RN, \Rb)$ satisfying 
$$ \|f\|^2_{\rho_\eta} = \big<f,f\big>_{\rho_\eta} < +\infty,$$
where the inner product is defined by 
\begin{equation}\label{def:normL2rho}
\big<f,g\big>_{\rho_\eta} =  \int_{\RN} f(y)g(y)\rho_\eta(y)dy \quad \text{with} \quad \rho_\eta(y) = \frac{1}{(4\pi \eta)^{N/2}}e^{-\frac{|y|^2}{4\eta^2}}.
\end{equation}
Note that the linear operator $\Lc_\eta$ can be written in the divergence form
$$\Lc_\eta v =  \frac{\eta}{\rho_\eta}\;\text{div}\,\Big(\rho_\eta\nabla v\Big),$$
which shows that $\Lc_\eta$ is self-adjoint with respect to the weight $\rho_\eta$, i.e.
\begin{equation}\label{eq:Ladjoint}
\forall v, w \in L^2_{\rho_\eta}, \quad \int_{\RN}v\Lc_\eta w \rho_\eta dy = \int_{\RN}w \Lc_\eta v \rho_\eta dy.
\end{equation}
For each $\alpha = (\alpha_1, \cdots, \alpha_N)\in \mathbb{N}^N$, we set
\begin{equation*}
\tilde{h}_\alpha(y) = c_\alpha\prod_{i = 1}^N H_{\alpha_i}\left(\frac{y_i}{2\sqrt{\eta}}\right),
\end{equation*}
where $H_n$ is the one dimensional Hermite polynomial defined by
\begin{equation}\label{def:Hermite}
H_n(x) = (-1)^ne^{x^2}\frac{d^n}{dx^n}(e^{-x^2}),
\end{equation}
and $c_\alpha \in \Rb$ is the normalization constant chosen so that the term of highest degree in $\tilde{h}_\alpha$ is $\prod_{i = 1}^Ny_i^{\alpha_i}$. In the one dimensional case, we have
\begin{equation}\label{eq:hntildeN1}
\tilde{h}_n(y) = \sum_{j = 0}^{\left[\frac{n}{2}\right]}c_{n,j}\eta^j y^{n - 2j} \quad \text{with} \quad c_{n,j} = (-1)^j\frac{n!}{(n - 2j)! j!}.
\end{equation}
The first four terms are explicitly given by
$$\tilde h_0 = 1, \quad \tilde h_1 = y, \quad \tilde{h}_2 = y^2 - 2\eta,$$
$$\tilde{h}_3 = y^3 - 6\eta y, \quad \tilde{h}_4 = y^4 - 12\eta y^2 + 12\eta^2.$$

The family of eigenfunctions of $\Lc_\eta$ generates an orthogonal basis in $L^2_{\rho_\eta}(\RN, \Rb)$, i.e. for any different $\alpha$ and $\beta$ in $\mathbb{N}^N$, 
\begin{equation*}
\Lc_\eta \tilde{h}_\alpha = -\frac{|\alpha|}{2}\tilde{h}_\alpha, \quad |\alpha| = \alpha_1+\cdots + \alpha_N,
\end{equation*}
\begin{equation}\label{eq:orthohnhm}
\int_{\RN} \tilde{h}_\alpha(y) \tilde{h}_\beta(y)\rho_\eta(y)dy = 0,
\end{equation}
and that for any $f$ in $L^2_{\rho_\eta}(\RN, \Rb)$, one can decompose
$$f = \sum_{\alpha \in \mathbb{N}^N}\big<f, \tilde{h}_\alpha\big>_{\rho_\eta}\tilde{h}_\alpha = \sum_{\alpha \in \mathbb{N}^N} f_\alpha \tilde{h}_\alpha.$$

\begin{remark} \label{rema:orth} For any polynomial $P_n(y)$ of degree $n$, we have by \eqref{eq:orthohnhm},
$$\int_{\RN}P_n(y) \tilde{h}_\alpha(y) \rho_\eta(y)dy = 0 \quad \text{if}\;\; |\alpha| \geq  n + 1.$$
\end{remark}

\medskip

\noindent $\bullet$ \textbf{Spectral properties of $\Hc$}: Let us consider the functional space $L^2_{\rho_1}(\RN, \Rb) \times L^2_{\rho_\mu}(\RN, \Rb)$, which is the set of all $\binom{f}{g} \in L^2_{loc}(\RN, \Rb) \times L^2_{loc}(\RN,\Rb)$ such that 
$$\left<\binom{f}{g}, \binom{f}{g}\right> < +\infty,$$
where 
$$\left<\binom{f_1}{g_1}, \binom{f_2}{g_2}\right>:= \big<f_1, f_2 \big>_{\rho_1} + \big<g_1, g_2 \big>_{\rho_\mu}.$$
If we introduce for each $\alpha \in \mathbb{N}^N$,
\begin{equation}\label{def:hkalpha}
h_\alpha(y) = a_\alpha\prod_{i=1}^NH_{\alpha_i}\left(\frac{y_i}{\sqrt 2}\right) \quad \text{and} \quad \hat{h}_\alpha(y) = \hat{a}_\alpha \prod_{i=1}^NH_{\alpha_i}\left(\frac{y_i}{2\sqrt{\mu}}\right),
\end{equation}
where $H_n$ is defined by \eqref{def:Hermite}, and $a_\alpha$ and $\hat a_\alpha$ are constants chosen so that the terms of highest degree in $h_\alpha$ and $\hat{h}_\alpha$ is $\prod_{i = 1}^N y^{\alpha_i}$, then 
\begin{equation}\label{eq:spectrumH}
\Hc \binom{h_\alpha}{0} = -\frac{|\alpha|}{2}\binom{h_\alpha}{0} \quad \text{and} \quad \Hc \binom{0}{\hat h_\alpha} = -\frac{|\alpha|}{2}\binom{0}{\hat h_\alpha}.
\end{equation}
Moreover, for each $\binom{f}{g}$ in $L^2_{\rho_1}(\RN, \Rb) \times L^2_{\rho_\mu}(\RN, \Rb)$, we have the decomposition
\begin{align*}
\binom{f}{g}= \sum_{\alpha \in \mathbb{N}^N}\big<f, h_\alpha\big>_{\rho_1}\binom{h_\alpha}{0} + \big<g,\hat{h}_\alpha \big>_{\rho_\mu}\binom{0}{\hat{h}_\alpha}.
\end{align*}

\medskip

\noindent $\bullet$ \textbf{Spectral properties of $\Hc + \Mc$}: In this part we derive a basis where $\Hc + \Mc$ is diagonal. More precisely, we have the following lemma whose proof follows from an explicit computation.

\begin{lemma}[Diagonalization of $\Hc + \Mc$ in the one dimensional case] \label{lemm:diagonal}  For all $n \in \mathbb{N}$, there exist polynomials $f_n, g_n, \tilde{f}_n$ and $\tilde{g}_n$ of degree $n$ such that

\begin{equation}\label{eq:HMspec1}
\Big(\Hc+ \Mc\Big)\binom{f_n}{g_n} = \left(1 - \frac{n}{2}\right)\binom{f_n}{g_n}, 
\end{equation}
and 
\begin{equation}\label{eq:HMspectilde}
 \Big(\Hc+ \Mc\Big)\binom{\tilde{f}_n}{\tilde{g}_n} = -\left(1 + \frac{n}{2}\right)\binom{\tilde{f}_n}{\tilde{g}_n},
\end{equation}
where
\begin{equation}\label{def:fngn}
\binom{f_n}{g_n} =  \sum_{j = 0}^{\left[\frac{n}{2}\right]}d_{n, n- 2j} \binom{h_{n - 2j}}{0} + e_{n, n - 2j}\binom{0}{\hat{h}_{n - 2j}},
\end{equation}
\begin{equation}\label{def:fngntilde}
\binom{\tilde{f}_n}{\tilde{g}_n} = \sum_{j = 0}^{\left[\frac{n}{2}\right]}\tilde d_{n, n - 2j} \binom{h_{n - 2j}}{0} + \tilde e_{n, n - 2j}\binom{0}{\hat{h}_{n - 2j}},
\end{equation}
and the coefficients $d_{n, n-2j}, e_{n, n-2j}$, $\tilde{d}_{n, n-2j}$, $\tilde{e}_{n, n-2j}$ depend on the parameters $p, q$ and $\mu$. In particular, we have

\begin{equation}\label{eq:dnen2}
\binom{d_{n,n}}{e_{n,n}} = \binom{q}{p}, \quad \binom{d_{n,n - 2}}{e_{n,n - 2}} = n(n-1)(\mu-1)\binom{-q}{p}.
\end{equation}
and 
\begin{equation}\label{eq:dnen2tilde}
\binom{\tilde d_{n,n}}{\tilde e_{n,n}} = \binom{q}{-p}, \quad \binom{\tilde d_{n,n-2}}{\tilde e_{n,n-2}} = \frac 13 n(n - 1)(1 - \mu)\binom{q}{p}.
\end{equation}

\end{lemma}
\begin{remark}\label{rema:012} Lemma \eqref{lemm:diagonal} also holds in higher dimensions with some complication in the computation. Here, we remark that the spectrum of $\Hc + \Mc$ has only two positive eigenvalues $\lambda_0 = 1$ and $\lambda_1 = \frac{1}{2}$ corresponding to the eigenvectors $\binom{f_0}{g_0}$ and $\binom{f_1}{g_1}$; a zero eigenvalue $\lambda_2 = 0$ corresponding to the eigenvector $\binom{f_2}{g_2}$. In the two dimensional case, we have 
$$\binom{f_0}{g_0} = \binom{q}{p}, \quad \binom{f_1}{g_1} = \binom{q y_i}{p y_i}_{1 \leq i \leq N
},$$
and 
$$\binom{f_2}{g_2} = \binom{f_{2,ij}}{g_{2,ij}}_{1 \leq i,j \leq N},$$
where
$$\binom{f_{2,ij}}{g_{2,ij}} = \binom{f_{2,ji}}{g_{2,ji}} = \binom{q y_i y_j}{p y_iy_j} \quad \text{for}\; 1 \leq i \neq j \leq N,$$
and 
\begin{equation}\label{def:f2g2Ndim}
\binom{f_{2,ii}}{g_{2,ii}} = \binom{q (y_i^2 - 2\mu)}{p (y_i^2 - 2)} \quad \text{for}\; 1 \leq i \leq N.
\end{equation}
\end{remark}

The following lemma gives the definition of the projection on the modes $\binom{f_n}{g_n}$ and $\binom{\tilde f_n}{\tilde g_n}$.

\begin{lemma}[Definition of the projection on the directions $\binom{f_n}{g_n}$ and $\binom{\tilde f_n}{\tilde g_n}$]\label{lemm:DefProjection} Let $M \gg 1$ be an even integer and let $\binom{\Lambda}{\Upsilon}$ be of the form 
\begin{equation}
\binom{\Lambda}{\Upsilon} = \sum_{n \leq M} \omega_n\binom{h_n}{0} + \hat \omega_n \binom{0}{\hat h_n}.
\end{equation}
Then we can expand  $\binom{\Lambda}{\Upsilon}$ with respect to the basis $\left\{\binom{f_n}{g_n}, \binom{\tilde f_n}{\tilde g_n}\right\}_{n \leq M}$ as follows:
\begin{equation}
\binom{\Lambda}{\Upsilon} = \sum_{n \leq M} \theta_n\binom{f_n}{g_n} + \tilde \theta_n \binom{\tilde f_n} {\tilde g_n},
\end{equation}
where
\begin{equation}\label{def:Pn}
\theta_n =  \sum_{j = 0}^{\left[\frac{M - n}{2}\right]} A_{n + 2j, n}\, \omega_{n + 2j} + B_{n + 2j, n}\, \hat \omega_{n + 2j}:= \Pc_{n,M}\binom{\Lambda}{\Upsilon},
\end{equation}
and 
\begin{equation}\label{def:Pntilde}
\tilde{\theta}_n =  \sum_{j = 0}^{\left[\frac{M - n}{2}\right]} \tilde A_{n + 2j, n}\, \omega_{n + 2j} + \tilde B_{n + 2j,n}\, \hat \omega_{n + 2j} := \tilde \Pc_{n,M}\binom{\Lambda}{\Upsilon},
\end{equation}
with the coefficients $A_{n + 2j, n}$, $B_{n + 2j,n}$, $\tilde A_{n + 2j, n}$ and $\tilde B_{n + 2j,n}$ for $j = 0, 1, 2, \cdots$ depending on $p, q$ and $\mu$. In particular, we have 
\begin{equation}\label{eq:AnBn}
A_{n,n} = \frac 1{2q}, \quad B_{n,n} = \frac 1{2p},
\end{equation}
and
\begin{equation}\label{eq:Ann2Bnn2}
A_{n+2,n} = \frac{1}{6q}(n+2)(n+1)(\mu - 1), \quad B_{n+2,n} = \frac{1}{6p}(n+2)(n+1)(1 - \mu).
\end{equation}
\end{lemma}
\begin{proof} Since the proof is exactly the same lines as the one written in \cite{GNZpre16c} and since it is purely computational, we kindly refer interested readers to Lemma 3.4 of \cite{GNZpre16c} for an analogous  proof.
\end{proof}
\begin{remark} \label{rema:pro} From Lemma \ref{lemm:DefProjection}, we obviously see that when a function is of the form $ \sum_{n = 0}^M \theta_n \binom{f_n}{g_n} + \tilde{\theta}_n \binom{\tilde f_n}{\tilde g_n}$, its projections on $\binom{f_n}{g_n}$ and $\binom{\tilde f_n}{\tilde g_n}$ are respectively $\theta_n$ and $\tilde \theta_n$.
\end{remark}

\subsection{A formal approach.} \label{sec:formal}
In this part we make use the spectral properties of the linear operator $\Hc+\Mc$ given in the previous subsection to formally derive the profile described in Theorem \ref{theo1}. For simplicity, we assume that $(u,v)$ is a positive, radially symmetric solution of system \eqref{PS} in the one dimensional case. By the translation invariance in space, we assume that $(u,v)$ blows up in finite time $T > 0$ at the origin. Let us start with the nonlinear transformation
\begin{equation}\label{def:baruv}
\bar u = e^{qu} \quad \text{and} \quad \bar v = e^{pv},
\end{equation}
which leads to the new system
\begin{equation}\label{sys:baruv}
\partial_t \bar u = \Delta \bar u - \frac{|\nabla \bar u|^2}{\bar u} + q \bar u \bar v, \qquad \partial_t \bar v = \mu\Delta \bar v - \mu\frac{|\nabla \bar v|^2}{\bar v} + p \bar u \bar v.
\end{equation}
We then introduce the similarity variables
\begin{equation}\label{def:simvariables}
\arraycolsep=1.6pt\def\arraystretch{2}
\left\{\begin{array}{l} 
\Phi(y,s) = (T-t)\bar u(x,t), \quad \Psi(y,s) = (T-t)\bar v(x,t),\\
\text{where} \quad y = \dfrac{x}{\sqrt {T-t}}, \quad s = -\log(T-t),
\end{array}\right.
\end{equation}
which shows that $(\Phi, \Psi)$ solves
\begin{equation} \label{sys:PhiPsi}
\arraycolsep=1.6pt\def\arraystretch{2}
\left\{\begin{array}{l}
\partial_s \Phi = \Delta \Phi - \dfrac{1}{2}y \cdot \nabla \Phi - \Phi - \dfrac{|\nabla \Phi|^2}{\Phi} + q \Phi \Psi,\\
\partial_s \Psi = \mu\Delta \Psi - \dfrac{1}{2}y \cdot \nabla \Psi - \Psi - \mu\dfrac{|\nabla \Psi|^2}{\Psi} + p \Phi \Psi.\end{array}\right.
\end{equation}
In the similarity variables \eqref{def:simvariables}, justifying \eqref{eq:asyTh1} is equivalent to show that 
\begin{equation}
\Phi(y,s) \sim \Phi^*\left(\frac{y}{\sqrt s}\right) \quad \text{and} \quad \Psi(y,s) \sim \Psi^*\left(\frac{y}{\sqrt s}\right) \quad \text{as} \quad s \to +\infty. 
\end{equation}
Note that the nonzero constant solution to system \eqref{sys:PhiPsi} is $\left(\frac{1}{p}, \frac{1}{q} \right)$. This suggests the linearization
\begin{equation}\label{def:barPhiPsi}
\bar \Phi = \Phi - \frac{1}{p} \quad \text{and} \quad \bar \Psi = \Psi - \frac{1}{q},
\end{equation}
and $(\bar \Phi, \bar \Psi)$ solves the system 
\begin{equation}\label{sys:barPhiPsi}
\partial_s \binom{\bar \Phi}{\bar \Psi} = \left(\Hc + \Mc\right)\binom{\bar \Phi}{\bar \Psi}  + \binom{q}{p}\bar\Phi \bar \Psi - \binom{|\nabla \bar \Phi|^2\left(\bar \Phi + \frac 1p\right)^{-1}}{\mu|\nabla \bar \Psi|^2\left(\bar \Psi + \frac 1q\right)^{-1}},
\end{equation}
where $\Hc$ and $\Mc$ are defined by \eqref{def:HLM}.

From Lemma \ref{lemm:diagonal}, we know that $\binom{f_n}{g_n}_{n \geq 3}$ and $\binom{\tilde f_n}{\tilde g_n}_{n \in \mathbb{N}}$ correspond to negative eigenvalues of $\Hc + \Mc$, therefore, we may consider that 
\begin{equation}\label{eq:expapr}
\binom{\bar \Phi}{\bar \Psi} = \theta_0(s)\binom{f_0}{g_0} + \theta_2(s)\binom{f_2}{g_2},
\end{equation}
where $|\theta_0(s)| + |\theta_2(s)| \to 0$ as $s\to +\infty$ (note that $\theta_1(s) \equiv 0$ by the radially symmetric assumption). Plugging this ansatz in system \eqref{sys:barPhiPsi} yields 
\begin{align*}
\theta_0'\binom{f_0}{g_0} + \theta_2'\binom{f_2}{g_2} = \theta_0\binom{f_0}{g_0} &+ \binom{q}{p}\left[(\theta_0 f_0 + \theta_2 f_2)(\theta_0 g_0 + \theta_2g_2)\right] \\
& - \theta_2^2\binom{|\nabla f_2|^2 \left(\theta_0 f_0 + \theta_2 f_2 + \frac{1}{p}\right)^{-1}}{\mu |\nabla g_2|^2\left(\theta_0g_0 + \theta_2 g_2 + \frac{1}{q}\right)^{-1}}.
\end{align*}
Assume that $|\theta_0(s)| \ll |\theta_2(s)|$ as $s \to +\infty$, we then use Lemma \ref{lemm:DefProjection} to find the ordinary differential system 
\begin{equation}\label{sys:Theta02}
\arraycolsep=1.4pt\def\arraystretch{2}
\left\{\begin{array}{ll}
\theta_0' &= \theta_0 + \Oc(|\theta_2|^2),\\
\theta_2' &= c_2\theta_2^2 + \Oc(|\theta_2|^3 + |\theta_0 \theta_2| + |\theta_0|^3),
\end{array}\right.
\end{equation} 
where the constant $c_2$ is computed as follows: 
\begin{align*}
c_2 &= \Pc_{2,M}\binom{qf_2g_2 - p|\nabla f_2|^2}{pf_2g_2 - q\mu|\nabla g_2|^2} = \Pc_{2,M}\binom{pq^2\big[h_4 + (6 - 2\mu)h_2 + 12\big]}{p^2q \big[\hat h_4 + (6\mu - 2)\hat h_2 \big]}\\
&= A_{2,2}(6-2\mu)pq^2 + B_{2,2}(6\mu - 2)p^2q + A_{4,2}pq^2 + B_{4,2}pq^2\\
& = 2pq(\mu + 1).
\end{align*}
Solving system \eqref{sys:Theta02} yields
$$\theta_2(s) = - \frac{1}{2pq(\mu + 1) s} + \Oc\left(\frac{\ln s}{s^2}\right) \quad \text{and} \quad |\theta_0(s)| = \Oc\left(\frac{1}{s^2}\right) \quad \text{as} \quad s\to +\infty.$$
From \eqref{eq:expapr} and \eqref{def:barPhiPsi}, we have just derived the following asymptotic expansion:
\begin{equation}\label{eq:solPhiPsiyb}
\arraycolsep=1.4pt\def\arraystretch{2}
\left\{\begin{array}{ll}
\Phi(y,s) &= \frac{1}{p} \left[ 1 - \frac{y^2}{2(\mu + 1)s} + \frac{\mu}{(\mu + 1)s}\right] + \Oc\left(\frac{\ln s}{s^2}\right),\\
\Psi(y,s) &= \frac{1}{q}\left[1 - \frac{y^2}{2(\mu + 1)s} + \frac{1}{(\mu + 1)s} \right] + \Oc\left(\frac{\ln s}{s^2}\right),\end{array}\right.
\end{equation}
where the convergence takes place in $L^2_{\rho_1} \times L^2_{\rho_\mu}$ as well as uniformly on compact sets by standard parabolic regularity. 

These expansions provide a relevant variable for blowup, namely $z = \frac{y}{\sqrt s}$, therefore, we  try to search formally solutions of \eqref{sys:PhiPsi} of the form 
\begin{equation}\label{eq:formPhiPsiz}
\arraycolsep=1.4pt\def\arraystretch{2}
\left\{\begin{array}{ll}
\Phi(y,s) &= \Phi_0(z) + \frac{\mu}{p(\mu + 1)s} + \Oc\left(\frac{1}{s^{1 + \nu}}\right),\\
\Psi(y,s) &= \Psi_0(z) + \frac{1}{q(\mu + 1)s} + \Oc\left(\frac{1}{s^{1 + \nu}}\right),
\end{array}\right.
\end{equation}
for some $\nu > 0$, subject to the condition
\begin{equation}\label{eq:boucond}
\Phi_0(0) = \frac{1}{p}, \quad \Psi_0(0) = \frac{1}{q}.
\end{equation}
Plugging this ansatz in system \eqref{sys:PhiPsi}, keeping only the main order, we end up with the following system satisfied by $(\Phi_0, \Psi_0)$:
\begin{align}
-\frac{z}{2}\Phi'_0 - \Phi_0 + q\Phi_0\Psi_0 = 0, \quad -\frac{z}{2}\Psi'_0 - \Psi_0 + p\Phi_0\Psi_0 = 0. \label{eq:Phi0Psi0}
\end{align}
Solving this system with the condition \eqref{eq:boucond} yields
$$\Phi_0(z) = \frac{1}{p}(1 + c_0|z|^2)^{-1}, \quad \Psi_0(z) = \frac{1}{q}(1 + c_0|z|^2)^{-1},$$
for some constant $c_0 > 0$. 
By matching asymptotic this expansion with \eqref{eq:solPhiPsiyb}, we find that 
$$c_0 = \frac{1}{2(\mu + 1)}.$$
In conclusion, we have formally obtained from \eqref{eq:formPhiPsiz} the following candidate for the profile:
\begin{equation}\label{def:phipsi}
\arraycolsep=1.4pt\def\arraystretch{2}
\left\{\begin{array}{ll}
\Phi(y,s) & \sim \phi(y,s):= \frac{1}{p}\left(1 + \frac{y^2}{2(\mu + 1)s}\right)^{-1}  + \frac{\mu}{p(\mu + 1)s},\\
\Psi(y,s) & \sim  \psi(y,s):= \frac{1}{q}\left(1 + \frac{y^2}{2(\mu + 1)s}\right)^{-1} + \frac{1}{q(\mu + 1)s}.\end{array}\right.
\end{equation}

\section{Proof of Theorem \ref{theo1} without technical details.}\label{sec:Existence}
In this section we give the proof of Theorem \ref{theo1}. To avoid winding up with details, we will only give the main arguments of the proof and postpone most of technicalities to next sections. For simplicity, we consider the one dimensional case ($N = 1$), however, the proof remains the same for higher dimensions $N \geq 2$.

Hereafter we denote by $C$ a generic positive constant depending only on the parameters of the problem such as $N, p, q, \mu$ and $K$ introduced in \eqref{def:chi}. 

\subsection{Linearization of the problem.}\label{sec:formu}
In this part we give the formulation of the problem to justify the formal result obtained in previous section, i.e. the proof of Theorem \ref{theo1}. We want to prove the existence of suitable initial data $(u_0, v_0)$ so that the corresponding solution $(u,v)$ of system \eqref{PS} blows up in finite time $T$ only at one point $a \in \Rb$ and verifies \eqref{eq:asyTh1}. From translation invariance of equation \eqref{PS}, we may assume that $a = 0$. Through the transformations \eqref{def:baruv} and \eqref{def:simvariables}, we want to find $s_0 > 0$ and $(\Phi(y,s_0), \Psi(y,s_0))$ such that the solution $(\Phi, \Psi)$ of system \eqref{sys:PhiPsi} with initial data $(\Phi(y,s_0), \Psi(y,s_0))$ satisfies
\begin{equation}\label{eq:limPhiPsis3}
\lim_{s\to+\infty} \left\| \Phi(y,s) - \Phi^*\left(\frac{y}{\sqrt{s}}\right)\right\|_{L^\infty(\RN)} = \lim_{s\to+\infty} \left\| \Psi(y,s) - \Psi^*\left(\frac{y}{\sqrt{s}}\right)\right\|_{L^\infty(\RN)} = 0,
\end{equation}
where $\Phi^*$ and $\Psi^*$  are defined  in \eqref{def:PhiPsistarpro}.

According to the formal analysis in the previous section, let us introduce $\Lambda(y,s)$ and $\Upsilon(y,s)$ such that 
\begin{equation}\label{def:LamUps}
\Phi(y,s) = \Lambda(y,s) + \phi(y,s), \quad \Psi(y,s) = \Upsilon(y,s) + \psi(y,s),
\end{equation}
where $\phi$ and $\psi$ are defined by \eqref{def:phipsi}.

With the introduction of $(\Lambda,\Upsilon)$ in \eqref{def:LamUps}, the problem is then reduced to construct functions $(\Lambda,\Upsilon)$ such that 
\begin{equation}\label{eq:goalconstruct}
\lim_{s \to+\infty}\|\Lambda(s)\|_{L^\infty(\RN)} = \lim_{s \to+\infty}\|\Upsilon(s)\|_{L^\infty(\RN)} = 0.
\end{equation}
and from \eqref{sys:PhiPsi}, $(\Lambda,\Upsilon)$ solves the system
\begin{equation}\label{eq:LamUp}
\partial_s \binom{\Lambda}{\Upsilon} = \Big(\Hc + \Mc + V(y,s)\Big)\binom{\Lambda}{\Upsilon} + \binom{q}{p}\Lambda \Upsilon + \binom{R_1}{R_2} + \binom{G_1}{G_2},
\end{equation}
where $\Hc$ and $\Mc$ are defined by \eqref{def:HLM}, 
\begin{equation}\label{def:Vys}
V(y,s) = \begin{pmatrix} 
q\psi - 1 & \quad q\big(\phi - 1/p\big)\\ p\big(\psi - 1/q\big) & \quad p\phi - 1
\end{pmatrix} = \begin{pmatrix} V_1 &V_2 \\V_3 &V_4
\end{pmatrix},
\end{equation}
\begin{equation}\label{def:Gys}
\binom{G_1}{G_2} = \binom{-|\nabla (\Lambda + \phi)|^2 (\Lambda + \phi)^{-1} + |\nabla \phi|^2 /\phi}{-\mu|\nabla(\Upsilon + \psi)|^2(\Upsilon + \psi)^{-1} + \mu |\nabla \psi|^2/ \psi},
\end{equation}
and 
\begin{equation}\label{def:Rys}
\binom{R_1}{R_2} = \binom{-\partial_s \phi + \Delta \phi - \frac{1}{2}y\cdot \nabla \phi - \phi + q\phi\psi - |\nabla \phi|^2/\phi}{-\partial_s \psi + \mu\Delta \psi - \frac{1}{2}y\cdot \nabla \psi - \psi + p\phi\psi - \mu |\nabla \psi|^2/\psi}.
\end{equation}

Since we would like to make $(\Lambda, \Upsilon)$ go to zero as $s \to +\infty$ in $L^\infty(\RN)\times L^\infty(\RN)$, then the nonlinear terms $\binom{q}{p}\Lambda \Upsilon$ and $\binom{G_1}{G_2}$, which are built to be quadratic, can be neglected. The error term $\binom{R_1}{R_2}$ is of the size $\frac{1}{s}$ uniformly in $\Rb^N$.  Thus, the dynamics of \eqref{eq:LamUp} are strongly influenced by the linear part
$$\Big(\Hc + \Mc + V(y,s)\Big)\binom{\Lambda}{\Upsilon} \quad \text{as}\quad s \to +\infty.$$
The spectrum of $\Hc + \Mc$ is well studied in the previous section. The potential $V(y,s)$ has two fundamental properties that will strongly influence our analysis:

\noindent - The effect of $V$ inside the \textit{blowup region} $|y| \leq K\sqrt{s}$ will be considered as a perturbation of the effect of $\Hc + \Mc$. 

\noindent - Outside the \textit{blowup region}, i.e. when $|y|\geq K \sqrt s$, we have the following property: for all $\epsilon > 0$, there exist $K_{\epsilon} > 0$ and $s_\epsilon > 0$ such that 
\begin{equation*}
\sup_{s \geq s_\epsilon, |y| \geq K_\epsilon \sqrt{s}} |V(y,s)|\leq \epsilon.
\end{equation*}
In other words, outside the \textit{blowup region}, the linear operator $\Hc + \Mc + V$ behaves as 
$$\Hc + \left(\begin{array}{cc} \pm \epsilon - 1 & \pm \epsilon\\ \pm \epsilon & \pm \epsilon - 1
\end{array}\right).$$
Given that the spectrum of $\Hc$ is non positive (see \eqref{eq:spectrumH} above) and that the matrix has negative eigenvalues for $\epsilon$ small, we see that $\Hc + \Mc + V$ behaves like one with a fully negative spectrum, which greatly simplifies the analysis in that region. \\

Since the behavior of the potential $V$ inside and outside the \textit{blowup region} is different, we will consider the dynamics for $|y| \geq K\sqrt{s}$ and $|y| \leq 2K\sqrt s$ separately for some $K$ to be fixed large. Let us consider a non-increasing cut-off function $\chi_0 \in \mathcal{C}^\infty_0([0, +\infty))$, with $\text{supp}(\chi_0) \subset [0,2]$ and $\chi_0 \equiv 1$ on $[0,1]$, and introduce
\begin{equation}\label{def:chi}
\chi(y,s) = \chi_0\left(\frac{|y|}{K_0\sqrt{s}}\right),
\end{equation}
where $K_0$ is chosen large enough so that various technical estimates hold. We define
\begin{equation}\label{def:LeUe}
\binom{\Lambda_e}{\Upsilon_e} = (1 - \chi)\binom{\Lambda}{\Upsilon},
\end{equation}
$\binom{\Lambda_e}{\Upsilon_e}$ coincides with $\binom{\Lambda}{\Upsilon}$ for $|y| \geq 2K_0\sqrt{s}$. As announced a few lines above and as we will see in  Section \ref{sec:outerpart}, the spectrum of the linear operator of the equation
satisfied by $\binom{\Lambda_e}{\Upsilon_e}$ is negative, which makes the control of $\|\Lambda_e(s)\|_{L^\infty(\Rb)}$ and $\|\Upsilon_e(s)\|_{L^\infty(\Rb)}$ easy.

While the control of the outer part is simple, it is not the case for the inner part of $\binom{\Lambda}{\Upsilon}$, i.e. for $|y| \leq 2K_0\sqrt s$. In fact, inside the \textit{blowup region} $|y| \leq 2K_0\sqrt s$, the potential $V$ can be seen as a perturbation of the effect of $\Hc + \Mc$ whose spectrum has two positive eigenvalues, a zero eigenvalue in addition to infinitely negative ones (see Lemma \ref{lemm:diagonal} above). For the sake of controlling $\binom{\Lambda}{\Upsilon}$ in the region $|y| \leq 2K_0\sqrt s$, we will expand  $\binom{\Lambda}{\Upsilon}$ with respect to the family $\left\{\binom{h_n}{0}, \binom{0}{\hat{h}_n}\right\}_{n \geq 0}$ and then with respect to the family $\left\{\binom{f_n}{g_n}, \binom{\tilde{f}_n}{\tilde{g}_n}\right\}_{n \geq 0}$ as follows:
\begin{align}
\binom{\Lambda(y,s)}{\Upsilon(y,s)} &= \sum_{n \leq M} Q_n(s) \binom{h_n(y)}{0} + \hat{Q}_n(s) \binom{0}{\hat{h}_n(y)} + \binom{\Lambda_-(y,s)}{\Upsilon_-(y,s)},\label{eq:expandLAmbUpb}\\
& = \sum_{n \leq M}\theta_n(s)\binom{f_n(y)}{g_n(y)} + \tilde{\theta}_n(s)\binom{\tilde{f}_n(y)}{\tilde{g}_n(y)} + \binom{\Lambda_-(y,s)}{\Upsilon_-(y,s)}. \label{decomoposeLamUp}
\end{align}
where $M$ is a fixed even integer satisfying 
\begin{equation}\label{eq:Mfixed}
M \geq 4\left[ \frac{p}{q} + \frac{q}{p} + \sum_{i = 1}^4\|V_i\|_{L^\infty_{y,s}}
 \right],
\end{equation}
with $\|V_i\|_{L^\infty_{y,s}} = \max\limits_{y\in \Rb^N, s \geq 1}|V_i(y,s)|$ .

\noindent $\bullet\;$ $Q_n(s)$ and $\hat Q_n(s)$ are  respectively the projections of $\binom{\Lambda}{\Upsilon}$ on $\binom{h_n}{0}$ and $\binom{0}{\hat h_n}$ defined by
\begin{equation}\label{def:Qn}
Q_n(s) = \frac{\left< \binom{\Lambda}{\Upsilon}, \binom{h_n}{0} \right>}{\left< \binom{h_n}{0}, \binom{h_n}{0} \right>} = \frac{\big<\Lambda,h_n \big>_{\rho_1}}{\big< h_n, h_n\big>^2_{\rho_1}}\equiv \Pi_n\binom{\Lambda}{\Upsilon},
\end{equation}
\begin{equation}\label{def:Qnhat}
\hat Q_n(s) = \frac{\left< \binom{\Lambda}{\Upsilon}, \binom{0}{\hat h_n} \right>}{\left< \binom{0}{\hat h_n}, \binom{0}{\hat h_n} \right>} = \frac{\big<\Upsilon,\hat{h}_n \big>_{\rho_\mu}}{\big< \hat{h}_n, \hat{h}_n\big>^2_{\rho_\mu}} \equiv \hat \Pi_n\binom{\Lambda}{\Upsilon},
\end{equation}

\noindent $\bullet\;$ $\binom{\Lambda_-(y,s)}{\Upsilon_-(y,s)} = \Pi_{-,M}\binom{\Lambda}{\Upsilon}$ denotes the infinite-dimensional part of $\binom{\Lambda}{\Upsilon}$, where $\Pi_{-,M}$ is the projector on the subspace of $L_{\rho_1} \times L_{\rho_\mu}$ where the spectrum of $\Hc$ is lower than $\frac{1-M}{2}$. We have the orthogonality: for all $n \leq M$,
\begin{equation}\label{eq:partnegVW}
\left< \binom{\Lambda_-}{\Upsilon_-}, \binom{h_n}{0} \right>= \big<\Lambda_-, h_n\big>_{\rho_1} = 0 \;\; \text{and}\;\; \left< \binom{\Lambda_-}{\Upsilon_-}, \binom{0}{\hat h_n} \right> = \big<\Upsilon_-,\hat h_n\big>_{\rho_\mu} = 0.
\end{equation}

\noindent $\bullet\;$ We set $\Pi_{+,M} = \textbf{Id} - \Pi_{-,M}$, and the complementary part 
$$\binom{\Lambda_+}{\Upsilon_+} = \Pi_{+,M}\binom{\Lambda}{\Upsilon} = \binom{\Lambda}{\Upsilon} - \binom{\Lambda_-}{\Upsilon_-}$$
which satisfies for all $s$,
\begin{equation}
\left<\binom{\Lambda_+(y,s)}{\Upsilon_+(y,s)}, \binom{\Lambda_-(y,s)}{\Upsilon_-(y,s)}\right> = 0.
\end{equation}

\noindent $\bullet\;$ $\theta_n(s)= \Pc_{n,M}\binom{\Lambda}{\Upsilon}$ and $\tilde{\theta}_n(s) = \tilde \Pc_{n,M}\binom{\Lambda}{\Upsilon}$ are respectively projections of $\binom{\Lambda}{\Upsilon}$ on $\binom{f_n}{g_n}$ and $\binom{\tilde f_n}{\tilde g_n}$. From Lemma \ref{lemm:DefProjection}, we can express $\theta_n(s)$ and $\tilde{\theta}_n(s)$ in terms of $Q_n(s)$ and $\hat Q_n(s)$.

\subsection{Definition of the shrinking set and its properties.}\label{sec:St}

In this part we will give the definition of a shrinking set to trap the solution according to the blowup regime described in Theorem \ref{theo1}. In particular, we aim at defining a set whose elements will satisfy \eqref{eq:goalconstruct}. To do so, we follow ideas of \cite{GNZpre16c} and \cite{GNZjde17} where the authors suggested a modification of the argument of \cite{MZdm97} for the standard semilinear heat equation \eqref{eq:ScalarE}. In particular, we shall control the solution in three different zones covering $\Rb^N$, defined as follows: For $K_0 > 0$, $\epsilon_0 > 0$ and $t \in [0, T)$, we set 
\begin{align*}
\Dc_1(t) &= \left\{x\;  \Big\vert \; |x| \leq K_0 \sqrt{|\ln(T-t)|(T-t)} \right\}\\
&\quad \equiv \left\{x \;\big\vert \; |y| \leq K_0 \sqrt{s}\right\} \equiv \left\{ x \; \Big\vert \; |z| \leq K_0\right\},\\
\Dc_2(t) &= \left\{x \; \Big\vert \; \frac{K_0}{4} \sqrt{|\ln(T-t)|(T-t)} \leq |x| \leq \epsilon_0 \right\}\\
&\quad  \equiv \left\{x \; \Big\vert\; \frac{K_0}{4} \sqrt{s} \leq |y| \leq \epsilon_0e^{\frac{s}{2}}\right\} \equiv \left\{ x \; \Big \vert\; \frac{K_0}{4} \leq |z| \leq  \frac{\epsilon_0}{\sqrt s }e^{\frac s2} \right\},\\
\Dc_3(t) &= \left\{x\;  \Big\vert \; |x| \geq \frac{\epsilon_0}{4} \right\} \equiv \left\{x \;\big\vert \; |y| \geq \frac{\epsilon_0}{4}e^{\frac{s}{2}}\right\} \equiv \left\{ x \; \Big\vert \; |z| \geq \frac{\epsilon_0}{4 \sqrt s}e^{\frac{s}{2}} \right\}.
\end{align*}
- In the \textit{blowup region} $\Dc_1$, we work with the self similar system \eqref{eq:LamUp} and do an analysis according to the decomposition \eqref{decomoposeLamUp} and  the definition \eqref{def:LeUe}.\\
- In the intermediate region $\Dc_2$, we control the solution by using classical parabolic estimates on $(\tilde u, \tilde v)$, a rescaled version of $(u,v)$ defined for $x \ne 0$ by 
\begin{equation}\label{def:uvtilde}
\arraycolsep=1.6pt\def\arraystretch{2}
\left\{\begin{array}{l} 
\tilde{u}(x, \xi, \tau) = \frac{1}{q}\ln \sigma(x) + u\Big(x + \xi\sqrt{\sigma(x)}, t(x) + \tau \sigma(x)\Big),\\
\tilde{v}(x, \xi, \tau) = \frac{1}{p}\ln \sigma(x) + v\Big(x + \xi\sqrt{\sigma(x)}, t(x) + \tau \sigma(x)\Big),
\end{array} \right.
\end{equation}
where  $t(x)$ is uniquely defined for $|x|$ sufficiently small by 
\begin{equation}\label{def:tx}
|x| = \frac{K_0}{4}\sqrt{\sigma(x)|\ln \sigma(x)|} \quad \text{with} \quad \sigma(x) = T - t(x).
\end{equation}
From \eqref{PS}, we see that $(\tilde{u}, \tilde{v})$ satisfies  the same system for $(u,v)$. That is for all $\xi \in \RN$ and $\tau \in \left[-\frac{t(x)}{\sigma(x)},1 \right)$,
\begin{equation}\label{eq:Ucxt}
\partial_\tau \tilde u = \Delta_\xi \tilde u + e^{p \tilde v}, \quad \partial_\tau \tilde v = \mu \Delta_\xi \tilde v + e^{q \tilde u}.
\end{equation}
We will in fact prove that $(\tilde u,\tilde v)$ behaves for 
$$|\xi| \leq \alpha_0\sqrt{|\ln \sigma(x)|}\quad \text{and} \quad \tau \in \left[\frac{t_0  - t(x)}{\sigma(x)},1\right)$$
for some $t_0 < T$ and $\alpha_0 > 0$, like the solution of the ordinary differential system 
\begin{equation}
\partial_\tau \hat u = e^{p\hat v}, \quad \partial_\tau \hat v = e^{q\hat u},
\end{equation}
subject to the initial data 
$$\hat u(0) = -\frac{1}{q}\ln \left[p\left(1 + \frac{K_0^2/16}{2(\mu + 1)}\right)\right], \quad \hat v(0) = -\frac{1}{p}\ln \left[q\left(1 + \frac{K_0^2/16}{2(\mu + 1)}\right)\right].$$
The solution is explicitly given by 
\begin{equation}\label{def:solUc}
\hat u(\tau) = -\frac{1}{q}\ln \left[p\left(1 - \tau + \frac{K_0^2/16}{2(\mu + 1)}\right)\right], \quad \hat v(\tau) =-\frac{1}{p}\ln \left[q\left(1 - \tau + \frac{K_0^2/16}{2(\mu + 1)}\right)\right]. 
\end{equation}
As we will see that the analysis in $\Dc_2$ will imply the conclusion of item $(iii)$ of Theorem \ref{theo1}. \\
- In $\Dc_3$, we directly estimate  $(u,v)$ by using the local in time well-posedness of the Cauchy problem for system \eqref{PS}.

We give the definition of the shrinking set to trap the solution according to the blowup regime described in Theorem \ref{theo1}. This set is precisely defined as follows:
\begin{definition}[Definition of a shrinking set] \label{def:St} For all $t_0 < T$, $K_0 > 0$, $\epsilon_0 > 0$, $\alpha_0 > 0$, $A > 0$, $\delta_0 > 0$, $\eta_0 > 0$, $C_0 > 0$, for all $t \in [t_0,T)$, we define  $\;\Sc(t_0, K_0, \epsilon_0, \alpha_0, A, \delta_0, \eta_0, C_0,  t)$ (or $\Sc(t)$ for short) being the set of all functions $(u,v)$ such that

\noindent $(i)\;$ (\textit{Control in the blowup region $\Dc_1$})  $\binom{\Lambda(s)}{\Upsilon(s)} \in \Vc_A(s)$ where $\binom{\Lambda}{\Upsilon}$ is defined as in \eqref{def:LamUps}, $s = -\ln(T-t)$ and $\Vc_A(s)$ is the set of all functions $\binom{\Lambda}{\Upsilon}$ verifying
$$\|\Lambda_e(s)\|_{L^\infty(\Rb)}, \|\Upsilon_e(s)\|_{L^\infty(\Rb)} \leq \frac{A^{M+2}}{\sqrt{s}},$$
$$ \left\|\Lambda_-(y,s) \right\|_{L^\infty(\mathbb{R})}, \left\|\Upsilon_-(y,s)\right\|_{L^\infty(\mathbb{R})} \leq A^{M+1}s^{-{\frac{M+2}{2}}} \big(|y|^{M+1} + 1\big),$$
$$ \left\|\nabla \Lambda_-(y,s) \right\|_{L^\infty(\mathbb{R})}, \left\|\nabla \Upsilon_-(y,s)\right\|_{L^\infty(\mathbb{R})} \leq A^{M+2}s^{-{\frac{M+2}{2}}} \big(|y|^{M+1} + 1\big),$$
$$|\tilde \theta_i(s)| \leq \frac{A^{2}}{s^2}\;\; \text{for}\;\; i = 0, 1,2, \quad |\theta_j(s)|, |\tilde{\theta}_j(s)| \leq A^js^{-\frac{j+1}{2}} \;\; \text{for}\;\; 3\leq j\leq M,$$
$$|\theta_0(s)|, |\theta_1(s)| \leq \frac{A}{s^2}, \quad |\theta_2(s)| \leq \frac{A^4 \ln s}{s^2},$$
where $\Lambda_e, \Upsilon_e$ are defined by \eqref{def:LeUe},  $\Lambda_-, \Upsilon_-$, $\theta_n$, $\tilde \theta_n$ are defined as in \eqref{decomoposeLamUp}.\\

\noindent $(ii)\,$ (\textit{Control in the intermediate region $\Dc_2$}) For all $|x| \in \left[\frac{K_0}{4}\sqrt{|\ln(T-t)|(T-t)}, \epsilon_0\right]$, $\tau = \tau(x,t) = \frac{t - t(x)}{\sigma(x)}$ and $|\xi| \leq \alpha_0 \sqrt{\ln \sigma(x)}$, 
\begin{align*}
\left|\tilde u(x,\xi, \tau) - \hat u(\tau)\right| \leq \delta_0, \quad |\nabla_\xi \tilde u(x, \xi, \tau)| \leq \frac{C_0}{\sqrt{|\ln \sigma(x)|}}, \\
\left|\tilde v(x,\xi, \tau) - \hat v(\tau)\right| \leq \delta_0, \quad |\nabla_\xi \tilde v(x, \xi, \tau)| \leq \frac{C_0}{\sqrt{|\ln \sigma(x)|}}, 
\end{align*}
where $\tilde u, \tilde{v}$, $\hat u$, $\hat v$, $t(x)$ and $\sigma(x)$ are defined in \eqref{def:uvtilde}, \eqref{def:solUc} and \eqref{def:tx} respectively.\\

\noindent $(iii)$ (\textit{Control in the regular region $\Dc_3$}) For all $|x| \geq \frac{\epsilon_0}{4}$, 
\begin{align*}
|\nabla_x ^i u(x,t) - \nabla_x ^i u(x, t_0)| \leq \eta_0 \quad \text{and} \quad |\nabla_x ^i v(x,t) - \nabla_x ^i v(x, t_0)| \leq \eta_0 \quad \text{for} \;\; i = 0,1.
\end{align*}
\end{definition}
\begin{remark} In comparison with the shrinking set defined in \cite{MZdm97}, our definition has additional estimates on $\nabla \Lambda_-$ and $\nabla \Upsilon_-$ in $\Dc_1$, $\nabla_\xi \tilde{u}$ and  $\nabla \tilde{v}$  in $\Dc_2$, $\nabla_x u$ and $\nabla_x v$ in $\Dc_3$. These estimates are needed to achieve the control of the nonlinear gradient term $\binom{G_1}{G_2}$ appearing in \eqref{eq:LamUp}. This idea was first used in \cite{MZnon97} for the finite time quenching for the vortex reconnection with the boundary problem, and then in \cite{GNZpre16c} for equation \eqref{eq:eu} coupled with a critical nonlinear gradient term.  
\end{remark}

\bigskip

As a mater of fact, if $\binom{\Lambda}{\Upsilon}(s) \in \Vc_A(s)$ for $s \geq s_0$, then
\begin{equation}\label{eq:LUinVAest}
\|\Lambda(s)\|_{L^\infty(\Rb)} + \|\Upsilon(s)\|_{L^\infty(\Rb)} \leq \frac{CA^{M+2}}{\sqrt{s}}, \quad \forall s \geq s_0,
\end{equation}
for some positive constant $C$. More precisely, we have the following proposition.

\begin{proposition}[Properties of elements belonging to $\Sc(t)$] \label{prop:proSt} For all $K_0 \geq 1$ and $\epsilon_0 > 0$, there exist $t_{0,2}(K_0, \epsilon_0)$ and $\eta_{0,2}(\epsilon_0) > 0$ such that for all $t_0 \in [t_{0,2}, T)$, $A \geq 1$, $\alpha_0 > 0$, $C_0 > 0$, $\delta_0 \leq \frac{1}{2} \min\{|\hat{u}(1)|, \hat{v}(1)\}$ and $\eta_0 \in (0, \eta_{0,2}]$, we have the following properties: Assume that the initial data $(u,v)(x,t_0)$ is given by \eqref{def:uvt0} and that for all $t \in [t_0,T)$, $(u,v)(t) \in \Sc(t)$, then there exists a positive constant $C = C(K_0, C_0)$ such that for all $y \in \RN$ and $s = -\log(T-t)$,\\
$(i)$ (Estimates on $(\Lambda, \Upsilon)$)
\begin{align*}
&|\Lambda(y,s)| + |\Upsilon(y,s)| \leq \frac{CA^{M+2}}{\sqrt s},\\
&|\Lambda(y,s)| + |\Upsilon(y,s)| \leq \frac{CA^4\ln s}{s^2}(|y|^2 + 1) + \sum_{j = 3}^{M + 1} \frac{CA^j}{s^{\frac{j + 1}{2}}}(|y|^j + 1).
\end{align*}
$(ii)$ (Estimates on $(\nabla \Lambda, \nabla \Upsilon)$)
\begin{align*}
&|\nabla \Lambda(y,s)| + |\nabla \Upsilon(y,s)| \leq \frac{CA^4\ln s}{s^2}(|y| + 1) + \sum_{j = 3}^M \frac{CA^j}{s^\frac{j + 1}{2}}(|y|^{j - 1} + 1) + \frac{CA^{M+2}}{s^\frac{M+2}{2}}(|y|^{M+1} + 1),\\
&|(1 - \chi(y,s))\nabla \Lambda(y,s)|  + |(1 - \chi(y,s))\nabla \Upsilon(y,s)|\leq \frac{C}{\sqrt s},\\
&|\nabla \Lambda(y,s)| + |\nabla \Upsilon(y,s)| \leq \frac{CA^{M+2}}{\sqrt s}.
\end{align*}
\end{proposition} 
\begin{proof} The proof of item $(i)$ and the first estimate in item $(ii)$ directly follows from the definition of the set $\Vc_A$ given in part $(i)$ of Definition \ref{def:St} and the decomposition \eqref{decomoposeLamUp}. The proof of the second estimate in item $(ii)$ follows from parts $(ii)$ and $(iii)$ of Definition \ref{def:St}. We kindly refer to Proposition A.1 in \cite{GNZjde17} where the reader can find an analogous proof for the case of single equation and have no difficulties to adapt to the system case. The last estimate in item $(ii)$ is a direct consequence of the first two ones. This concludes the proof of Proposition \ref{prop:proSt}.
\end{proof}

\subsection{Preparation of initial data.}

As for initial data at time $t = t_0$ for which the corresponding solution to system \eqref{PS} is trapped in the set $\Sc(t)$ for all $t \in [t_0,T)$, we consider the following functions depending on $(N+1)$ fine-tune parameters $(d_0, d_1) \in \Rb^{1 + N}$: 
\begin{align}
\binom{qu}{pv}_{d_0,d_1}(x,t_0) &= \binom{\hat u_*(x)}{\hat v_*(x)}\Big(1 - \chi_1(x,t_0)\Big) + \left\{\binom{1}{1}s_0 +\ln\left[\binom{\phi}{\psi}(y_0,s_0)\right]\right\}  \chi_1(x,t_0) \nonumber\\
& + \ln \left\{\left(d_0\binom{f_0(y_0)}{g_0(y_0)} + d_1. \binom{f_1(y_0)}{g_1(y_0)}\right)\frac{A^2}{s_0^2}\chi(16y_0, s_0)\right\}  \chi_1(x,t_0), \label{def:uvt0}
\end{align}
where $s_0 = -\ln(T-t_0)$, $y_0 = x e^{\frac{s_0}{2}}$, $\phi$ and $\psi$ are defined by \eqref{def:phipsi}, $\binom{f_0}{g_0}$ and $\binom{f_1}{g_1}$ are the eigenfunctions corresponding to the positive eigenvalues of the linear operator $\Hc + \Mc$ (see Lemma \ref{lemm:diagonal}), $\chi$ is introduced in \eqref{def:chi}, $\chi_1$ is defined by 
$$\chi_1(x,t_0) = \chi_0 \left(\frac{|x|}{|\ln(T-t_0)| \sqrt{T-t_0}}\right) = \chi_0\left(\frac{y_0}{s_0}\right),$$
and $(\hat u_*, \hat v_*) \in \Cc^\infty(\Rb^N \ \{0\}) \times \Cc^\infty(\Rb^N \setminus \{0\})$ is defined by 
\begin{equation}\label{def:ustar}
\hat u_*(x) = \left\{\begin{array}{ll}
\ln\left(\frac{4(\mu + 1)|\ln |x||}{p|x|^2}\right) &\quad \text{for}\quad |x| \leq C(a),\\
 -\ln\left(1 + a |x|^2\right) &\quad \text{for} \quad |x|\geq 1,
\end{array}
 \right.
\end{equation}
\begin{equation}\label{def:vstar}
\hat v_*(x) = \left\{\begin{array}{ll}
\ln\left(\frac{4(\mu + 1)|\ln |x||}{q|x|^2}\right) &\quad \text{for}\quad |x| \leq C(a),\\
 -\ln\left(1 + a |x|^2\right) &\quad \text{for} \quad |x|\geq 1.
\end{array}
 \right.
\end{equation}
 
By selecting suitable parameters, we make sure that the initial data \eqref{def:uvt0} starts in $\Sc(t_0)$. More precisely, we have the following.

\begin{proposition}[Properties of initial data \eqref{def:uvt0}] \label{prop:uvt0} There exists $K_{0,1} > 0$ such that for each $K_0 \geq K_{0,1}$ and $\delta_{0,1} > 0$, there exist $\alpha_{0,1}(K_0, \delta_{0,1}) > 0$, $C_{0,1}(K_0) > 0$ such that for all $\alpha_0 \in (0, \alpha_{0,1}]$, there exists $\epsilon_{0,1}(K_0, \delta_{0,1}, \alpha_0) > 0$ such that for all $ \epsilon_0 \in  (0,\epsilon_{0,1}]$ and $A \geq 1$, there exists $t_{0,1}(K_0, \delta_{0,1}, \epsilon_0, A, C_{0,1}) < T$ such that for all $t_0 \in [t_{0,1}, T)$, there exists a subset $\mathcal{D}_{t_0, A} \subset \Rb \times \RN$ with the following properties. If $(u,v)_{d_0,d_1}(x,t_0)$ is defined as in \eqref{def:uvt0}, then:\\

\noindent $(I)$ For all $(d_0,d_1) \in \Dc_{t_0,A}$,  $(u,v)_{d_0,d_1}(x,t_0)$ belongs in $\Sc(t_0, K_0, \epsilon_0, \alpha_0, A, \delta_{0,1}, 0, C_{0,1}, t_0)$.  More precisely, we have

$(i)$ \textit{(Estimates in $\Dc_1$)} $(\Lambda_0, \Upsilon_0)_{d_0,d_1} \in \Vc_{A}(s_0)$, where $(\Lambda_0, \Upsilon_0)_{d_0,d_1}$ is defined from $(u,v)_{d_0,d_1}(x,t_0)$ through the transformations \eqref{def:LamUps}, \eqref{def:simvariables} and \eqref{def:baruv} with $s_0 = - \ln(T-t_0)$ and $y = xe^{s_0/2}$,  with strict inequalities except for $(\theta_{0,0}, \theta_{0,1})(s_0)$ in the sense that
$$\|\Lambda_{0,e}\|_{L^\infty(\Rb)} = \|\Upsilon_{0,e}\|_{L^\infty(\Rb)} = 0,$$
$$ \left\|\Lambda_{0,-}(y) \right\|_{L^\infty(\mathbb{R})}, \left\|\Upsilon_{0,-}(y)\right\|_{L^\infty(\mathbb{R})}, \left\|\nabla \Lambda_{0,-}(y) \right\|_{L^\infty(\mathbb{R})}, \left\|\nabla \Upsilon_{0,-}(y)\right\|_{L^\infty(\mathbb{R})} \leq s_0^{-\frac{M+2}{2}} \big(|y|^{M+1} + 1\big),$$
$$|\tilde {\theta}_{0,i}| \leq \frac{1}{s_0^2}\;\; \text{for}\;\; i = 0, 1,2, \quad |\theta_{0,j}|, |\tilde{\theta}_{0,j}| \leq s_0^{-\frac{j+1}{2}} \;\; \text{for}\;\; 3\leq j\leq M, \quad |\theta_{0,2}| \leq \frac{\ln s_0}{s_0^2},$$
$$\left|\theta_{0,0} - \frac{Ad_0}{s_0^2}\right| + \left|\theta_{0,1} - \frac{Ad_1}{s_0^2}\right| \leq C(|d_0| + |d_1|)e^{-s_0},$$
where  $\Lambda_{0,e}, \Upsilon_{0,e}$, $\Lambda_{0,-}, \Upsilon_{0,-}$, $\theta_{0,n}$, $\tilde \theta_{0,n}$ are the components of $(\Lambda_0, \Upsilon_0)_{d_0,d_1}$ defined as in \eqref{def:LeUe} and \eqref{decomoposeLamUp}.\\

$(ii)$ \textit{(Estimates in $\Dc_2$)} For all $|x| \in \left[\frac{K_0}{4}\sqrt{|\ln\sigma(x)|\sigma(x)}, \epsilon_0\right]$, $\tau_0 = \tau_0(x,t_0) = \frac{t_0 - t(x)}{\sigma(x)}$ and $|\xi| \leq \alpha_0 \sqrt{\ln \sigma(x)}$ with $\sigma(x), t(x)$ being uniquely defined by \eqref{def:tx}, we have 
\begin{align*}
\left|\tilde u(x,\xi, \tau_0) - \hat u(\tau_0)\right| \leq \delta_{0,1}, \quad |\nabla_\xi \tilde u(x, \xi, \tau_0)| \leq \frac{C_{0,1}}{\sqrt{|\ln \sigma(x)|}},\\
\left|\tilde v(x,\xi, \tau_0) - \hat v(\tau_0)\right| \leq \delta_{0,1}, \quad |\nabla_\xi \tilde v(x, \xi, \tau_0)| \leq \frac{C_{0,1}}{\sqrt{|\ln \sigma(x)|}},
\end{align*}
where $\tilde u, \tilde{v}$, $\hat u$, $\hat v$ are defined in \eqref{def:uvtilde} and \eqref{def:solUc}.\\

\noindent $(II)$ Let $\hat \Vc_A(s_0) = \left[-\frac{A}{s_0^2}, \frac{A}{s_0^2} \right]^{1 + N}$, then
\begin{align*}
(d_0,d_1) \in \Dc_{t_0,A} &\Longleftrightarrow \big(\theta_{0,0}, \theta_{0,1}\big)(s_0) \in \hat \Vc_A(s_0),\\
(d_0,d_1) \in \partial \Dc_{t_0,A} &\Longleftrightarrow \big(\theta_{0,0}, \theta_{0,1}\big)(s_0) \in \partial \hat \Vc_A(s_0),
\end{align*}
\end{proposition} 
\begin{proof} Item $(II)$ directly follows from item $(i)$ of part $(I)$. The proof of item $(i)$ of part $(I)$ mainly relies on the projections of $(\Lambda_0, \Upsilon_0)_{d_0,d_1}$ defined as in Lemma \ref{lemm:DefProjection}. Since its proof is purely computational,  we refer the readers to Lemma 5.2 in \cite{GNZpre16c} for an analogous proof. As for the proof of item $(ii)$ of part $(I)$, see Lemma A.2 in \cite{GNZjde17} where the proof for the case of a single equation is treated in details and the same proof can be carried on for the system case without difficulties. This concludes the proof of Proposition \ref{prop:uvt0}.
\end{proof}

\subsection{Existence of solutions trapped in $\Sc(t)$.}
In this section we aim at proving the following proposition which implies Theorem \ref{theo1}.
\begin{proposition}[Existence of solutions of \eqref{eq:LamUp} trapped in $\Sc(t)$] \label{prop:goalVA}
We can choose parameters $t_0 < T$, $K_0, \epsilon_0, \alpha_0, A, \delta_0, \eta_0, C_0$ such that the following holds: there exists $(d_0,d_1) \in \Rb^{1 + N}$ such that if $(u,v)(x,t)$ is the solution to the system \eqref{PS} with initial data at $t = t_0$ given by \eqref{def:uvt0}, then $(u,v)(x,t)$ exists for all $(x,t) \in \Rb^N \times [t_0,T)$ and satisfies
$$(u,v)(t) \in \Sc(t), \quad \forall t \in [t_0,T).$$
\end{proposition}
\begin{proof} The proof of this proposition follows from the general idea developed in \cite{MZdm97}. We proceed in two steps:\\
- In the first step, we reduce the problem of controlling $(u, v)(t)$ in $\Sc(t)$ to the control of $(\theta_0, \theta_1)(s)$ in $\left[-\frac{A}{s^2}, \frac{A}{s^2} \right]^{1 + N}$, where $(\theta_0, \theta_1)$ are the positive modes of $(\Lambda, \Upsilon)$ defined as in \eqref{decomoposeLamUp}.\\
- In the second step, we use a classical topological argument based on index theory to solve the finite dimensional problem. \\

\noindent \textit{Step 1: Reduction to a finite dimensional problem.}

In this step, we show through \textit{a priori estimate} that the control of $(u,v)(t)$ in $\Sc(t)$ reduces to the control of $(\theta_0, \theta_1)(s)$ in $\hat \Vc_A(s) = \left[-\frac{A}{s^2}, \frac{A}{s^2} \right]^{1 + N}$. This result crucially follows from a good understanding of the properties of the linear operator $\Hc + \Mc + V$ of equation \eqref{eq:LamUp} in the \textit{blowup region} $\Dc_1$ together with classical parabolic techniques for the analysis in the intermediate and regular regions $\Dc_2$ and $\Dc_3$. In particular, we claim the following proposition, which is the heart of our contribution:
\begin{proposition}[Control of $(u,v)(t)$ in $\Sc(t)$ by $(\theta_0,\theta_1)(s)$ in $\hat \Vc_A(s)$]\label{prop:redu} We can choose parameters $t_0 < T$, $K_0$, $\epsilon_0$, $\alpha_0$, $A$, $\delta_0$, $\eta_0$, $C_0$  such that the following properties hold. Assume that $(u,v)(x,t_0)$ is given by \eqref{def:uvt0} with $(d_0,d_1) \in \Dc_{t_0,A}$. Assume in addition that for some $t^* \in [t_0,T)$, 
$$(u,v)(t) \in \Sc(t_0, K_0, \epsilon_0, \alpha_0, A, \delta_0, \eta_0, C_0, t), \quad \forall t \in [t_0,t^*],$$
and 
$$(u,v)(t^*) \in \partial\Sc(t_0, K_0, \epsilon_0, \alpha_0, A, \delta_0, \eta_0, C_0, t^*).$$
Then, we have\\
$(i)$ (Finite dimensional reduction) $(\theta_0, \theta_1)(s^*) \in \partial \hat{\mathcal{V}}_{A}(s^*)$, where $s^* = -\log(T - t^*)$ and $\theta_0, \theta_1$ are the components of $(\Lambda, \Upsilon)$ defined as in \eqref{decomoposeLamUp}.\\
$(ii)$ (Transversality) There exists $\mu_0 > 0$ such that for all $\mu \in (0, \mu_0)$, 
$$(\theta_0, \theta_1)(s^* + \mu) \not \in \hat{\mathcal{V}}_{A}(s^* + \mu),$$
hence,
$$ (u,v)(t^* + \mu') \not  \in \Sc(t_0, K_0, \epsilon_0, \alpha_0, A, \delta_0, \eta_0, C_0, t^* + \mu'), \quad \mu' = \mu'(t^*,\mu) > 0.$$
\end{proposition}

\begin{proof} The proof uses ideas of \cite{GNZpre16c, GNZjde17} where the authors adapted the technique of \textit{a priori} estimates developed in \cite{BKnon94} and \cite{MZdm97} for equation \eqref{eq:ScalarE}. Let us insist on the fact that the techniques introduced in \cite{BKnon94} and \cite{MZdm97} are not enough to handle the nonlinear gradient term appearing in equation \eqref{eq:LamUp}. The essential idea is to introduce additional estimates in the intermediate and regular zones to achieve the control of this term and this is one of the main novelties in this paper. The main feature of the proof is that the bounds appearing in Definition \ref{def:St} can be improved, except the bounds on $(\theta_0,\theta_1)$. More precisely, the improvement of the bounds in the \textit{blowup region} $\Dc_1$ (except for $\theta_0, \theta_1$) is done through projecting equation \eqref{eq:LamUp} on the different components of $(\Lambda, \Upsilon)$ introduced in \eqref{decomoposeLamUp}. One can see that the components $\big\{\theta_j\big\}_{2 \
 leq j \leq M}$, $\big\{\tilde{\theta}_j\big\}_{0 \leq j \leq M}$, $(\Lambda_-, \Upsilon_-)$, $(\nabla \Lambda_-, \nabla \Upsilon_-)$, $(\Lambda_e, \Upsilon_e)$ which correspond to decreasing directions of the flow, are already small at $s = s_0$ and they remain small up to $s = s^*$, hence, they can not touch their boundary. In $\Dc_2$ and $\Dc_3$, we directly use parabolic techniques applied to system \eqref{PS}  
to achieve the improvement. Therefore, only $\theta_0$ and $\theta_1$ may touch their boundary at $s = s^*$ and the conclusion follows. Since we would like to keep the proof of Proposition \ref{prop:goalVA} short, we leave the proof of Proposition \ref{prop:redu} to the next section.
\end{proof}

\noindent\textit{Step 2: A basic topological argument.} \label{step2:topoarg}

From Proposition \ref{prop:redu}, we claim that there exist $(d_0,d_1) \in \Dc_{t_0,A}$ such that equation \eqref{PS} with initial data \eqref{def:uvt0} has a solution 
$$(u,v)_{d_0,d_1}(t) \in \Sc(t_0, K_0, \epsilon_0, \alpha_0, A, \delta_0, \eta_0, C_0, t)\quad  \text{for all $t \in [t_0, T)$},$$
for a suitable choice of the parameters. Note that the argument of the proof is not new and it is completely analogous as in \cite{MZdm97}. Let us give its main ideas.

Let us consider $t_0, K_0, \epsilon_0, \alpha_0, A, \delta_0, \eta_0, C_0$ such that Propositions \ref{prop:redu} and \ref{prop:uvt0} hold. From Proposition \ref{prop:uvt0}, we have 
\begin{align*}
\forall (d_0,d_1) \in \Dc_{t_0,A}, \quad (u,v)_{d_0,d_1}(x,t_0) \in \Sc(t_0, K_0, \epsilon_0, \alpha_0, A, \delta_0, \eta_0, C_0, t_0),
\end{align*}
where $(u,v)_{d_0,d_1}(x,t_0)$ is defined by \eqref{def:uvt0}. Note that $(u,v)_{d_0,d_1}(x,t_0) \in \mathcal{H}_{a}$, where $\mathcal{H}_a$ is introduced in \eqref{def:Ha}. Therefore, from the local existence theory for the Cauchy problem of \eqref{PS} in $\mathcal{H}_{a}$, we can define for each $(d_0,d_1) \in \Dc_{t_0,A}$ a maximum time $t_*(d_0,d_1) \in [t_0,T)$ such that 
$$(u,v)_{d_0,d_1}(t) \in \Sc(t_0, K_0, \epsilon_0, \alpha_0, A, \delta_0, \eta_0, C_0, t) , \quad \forall t \in [t_0,t_*).$$ 
If $t_*(d_0,d_1) = T$ for some $(d_0, d_1) \in \mathcal{D}_{t_0,A}$, then the proof is complete. Otherwise, we argue by contradiction and assume that $t_*(d_0, d_1) < T$ for any $(d_0, d_1) \in \mathcal{D}_{t_0,A}$. By continuity and the definition of $t_*$, the solution $(u,v)_{d_0,d_1}(t)$ at time $t = t_*$ is on the boundary of $\Sc(t_*)$. From part $(i)$ of Proposition \ref{prop:redu}, we have
$$(\theta_0, \theta_1)(s_*) \in \partial\hat{\mathcal{V}}_A(s_*) \quad \text{with} \quad s_* = -\ln(T - t_*).$$
Hence, we may define the rescaled flow $\Gamma$ at $s = s_*$ for $\theta_0$ and $\theta_1$ as follows:
\begin{align*}
\Gamma:\quad \mathcal{D}_{t_0,A}\quad &\mapsto \quad \partial([-1,1] \times [-1,1]^N)\\
(d_0, d_1)\quad &\to \quad \left(\frac{s_*^2}{A}\theta_0(s_*), \frac{s_*^2}{A}\theta_1(s_*)\right).
\end{align*}
It follows from part $(ii)$ of Proposition \ref{prop:redu} that $\Gamma$ is continuous. If we manage to prove that the degree of $\Gamma$ on the boundary is different from zero, then we have a contradiction from the degree theory. Let us prove that. From part $(II)$ Proposition \ref{prop:uvt0}, we see that if $(d_0,d_1) \in \partial \mathcal{D}_{t_0,A}$, then 
$$(\theta_0, \theta_1)(s_0) \in \partial \hat{\mathcal{V}}_{A}(s_0).$$
Using part $(ii)$ of Proposition \ref{prop:redu}, we see that $(\Lambda, \Upsilon)(s)$ must leave $\mathcal{V}_{A}(s)$ at $s = s_0$, hence, $s_*(d_0,d_1) = s_0$. Using again part $(i)$ of Proposition \ref{prop:uvt0}, we see that the degree of $\Gamma$ on the boundary must be different from zero. This gives us a contradiction (by the index theory) and concludes the proof of Proposition \ref{prop:goalVA}, assuming that Proposition \ref{prop:redu} holds.
\end{proof}

\subsection{Conclusion of the proof of Theorem \ref{theo1}.}
In this part we use Proposition \ref{prop:goalVA} to conclude the proof of Theorem \ref{theo1}. We have already showed in  Proposition \ref{prop:goalVA} that there exist initial data of the form \eqref{def:uvt0} such that the corresponding solution  $(u,v)(t)$ of system \eqref{eq:LamUp} satisfies $(u,v)(t) \in \Sc(t)$ for all $t \in [t_0, T)$. From item $(i)$ of Definition \ref{def:St}, we have $(\Lambda, \Upsilon)(s) \in \Vc_A(s)$ for all $s \geq s_0$. This means that \eqref{eq:goalconstruct} holds for all $s \geq s_0$. From \eqref{def:LamUps}, \eqref{def:simiVars} and \eqref{def:baruv}, we concludes the proof of part $(ii)$ of Theorem \ref{theo1}. 

From \eqref{eq:asyTh1}, we see  that 
$$e^{qu(0,t)} \sim \frac{1}{p(T-t)} \quad \text{and} \quad e^{v(0,t)} \sim \frac{1}{q(T-t)} \quad \text{as}\; t \to T.$$
Hence, $e^{qu}$ and $e^{pv}$ blow up at time $T$ at the ogirin simultaneously.  It remains to show that if $x_0 \neq 0$, then $x_0$ is not a blowup point of $e^{qu}$ and $e^{pv}$. The following result allows us to conclude.

\begin{proposition}[No blowup under some threshold] \label{prop:Noblowup} For all $C_0 > 0$, there is $\eta_0 > 0$ such that if $\big(u(\xi,\tau), v(\xi,\tau)\big)$ solves
\begin{equation*}
\big |\partial_\tau u - \Delta u\big| \leq C_0\big(1+e^{qv}\big), \quad \big| \partial_\tau v - \mu \Delta v\big| \leq C_0 \big( 1 + e^{pu}\big)
\end{equation*}
and satisfies
$$(1 - \tau)e^{pv(\xi, \tau)} + (1 - \tau)e^{qu(\xi, \tau)} \leq \epsilon, \quad \forall |\xi| < 1, \; \tau \in [0,1),$$
then, $e^{qu}$ and $e^{pv}$ do not blow up at $\xi = 0$ and $\tau = 1$.
\end{proposition}
\begin{proof} The proof of this result uses ideas given by Giga and Kohn \cite{GKcpam89} for the single equation with the nonlinear source term $|u|^p$. Their proof uses a truncation technique together with the smoothness effect of the heat semigroup $e^{\tau \Delta}$ and some type of Gronwall's argument. Since their argument can be extended to our case without difficulties, we kindly refer the interested readers to Theorem 2.1 in \cite{GKcpam89} for an analogous proof.
\end{proof}

\medskip
From \eqref{eq:asyTh1}, we see that 
$$
\sup_{|x| < \frac{|x_0|}{2}} (T-t)e^{qu(x,t)} \leq \Phi^*\left(\frac{|x_0|/2}{\sqrt{(T-t)\ln(T-t)}} \right) + \frac{C}{\sqrt{\ln(T-t)}} \to 0,
$$
and 
$$
\sup_{|x| < \frac{|x_0|}{2}} (T-t)e^{pv(x,t)} \leq \Psi^*\left(\frac{|x_0|/2}{\sqrt{(T-t)\ln(T-t)}} \right) + \frac{C}{\sqrt{\ln(T-t)}} \to 0,
$$
as $t \to T$, hence, $x_0$ is not a blowup point of $e^{qu}$ and $e^{pv}$ from Proposition \ref{prop:Noblowup}. This concludes the proof of part $(i)$ of Theorem \ref{theo1}.\\

We now give the proof of part $(iii)$ of Thereom \ref{theo1}. Using the technique of Merle \cite{Mercpam92}, we derive the existence of a blowup profile $(u^*, v^*) \in \Cc^2(\Rb^*) \times \Cc^2(\Rb^*)$ such that 
$$(u,v)(x,t) \to (u^*, v^*)(x) \quad \text{as} \quad t \to T.$$
Here, we are interested in finding an equivalent of $(u^*, v^*)(x)$ for $|x|$ small.  To do so, let us consider the rescaled functions $\big(\tilde{u}, \tilde{v}\big)(x, \xi, \tau))$ defined as in \eqref{def:uvtilde}. From item $(ii)$ of Definition \ref{def:St} and \eqref{def:solUc}, we have 
\begin{align*}
u^*(x) = \lim_{t \to T}u(x,t) &= \lim_{\tau \to 1} \left[-\frac{1}{q}\ln(T-t(x)) + \tilde{u}(x,0, \tau)\right]\\
&\sim -\frac{1}{q}\ln(T-t(x)) - \frac{1}{q}\ln \left(p \frac{K_0^2/16}{2(\mu + 1)}\right).
\end{align*}
Using the definition \eqref{def:tx} of $t(x)$, we have 
$$ -\ln(T-t(x))\sim -2\ln|x|, \quad T-t(x)  = \frac{16}{K_0^2}\frac{|x|^2}{|\ln(T-t(x))|} \sim \frac{16}{K_0^2}\frac{|x|^2}{2|\ln|x||} \quad \text{for} \; |x| \to 0.$$
This yields
$$u^*(x) \sim -\frac{1}{q}\ln \left(\frac{p|x|^2}{4(\mu + 1)|\ln|x||}\right) \quad \text{as} \;\; |x| \to 0.$$
Similarly, we obtain 
$$v^*(x) \sim -\frac{1}{p}\ln \left(\frac{q|x|^2}{4(\mu + 1)|\ln|x||}\right) \quad \text{as} \;\; |x| \to 0.$$
This concludes the proof of Theorem \ref{theo1} assuming that Proposition \ref{prop:redu} holds.

\section{Reduction to a finite dimensional problem.} \label{sec:reduction}
In this section we give the proof of Proposition \ref{prop:redu}, which is the central part in our analysis. As mentioned in the beginning of Section \ref{sec:Existence}, we will consider the one dimensional case for simplicity, however, the same proof holds for higher dimensional cases.  We proceed in two subsections:\\
- In the first subsection, we derive an \textit{a priori estimates} on $(u,v)(t)$ in $\Sc(t)$. In the region $\Dc_1$, we project system \eqref{eq:LamUp} on the different components defined by \eqref{def:LeUe} and the decomposition \eqref{decomoposeLamUp}. In comparison with the work \cite{GNZpre16c}, we have an extra nonlinear gradient term $\binom{G_1}{G_2}$ which is delicate since we need both upper and lower bound of the solution. In the intermediate region $\Dc_2$, we work with the rescaled version \eqref{def:uvtilde} and control the solution by classical parabolic techniques. In the regular region $\Dc_3$, we directly estimate the solution by using the local well-posedness in time of the Cauchy problem for system \eqref{PS}. \\
- In the second subsection, we use the \textit{a priori estimates} obtained in the first part to show that the new bounds are better than the ones defined in $\Sc(t)$ except for the modes $\theta_0$ and $\theta_1$. This reduces the problem to a finite dimensional one which concludes item $(i)$ of Proposition \ref{prop:redu}. The outgoing transversality property is just a consequence of the dynamics of the modes $\theta_0$ and $\theta_1$.

\subsection{A priori estimates in $\Dc_1$.}
We claim the following:
\begin{proposition}[A priori estimates in $\Dc_1$]  \label{prop:dyn} There exist $K_{0,2} > 0$ and $A_{0,2} > 0$ such that for all $K_0 \geq K_{0,2}$, $\epsilon_0 > 0$, $A \geq A_{0,2}$, $\lambda^* > 0$, $C_{0,2} > 0$,  there exists $t_{0,2}(K_0, \epsilon_0, A, \lambda^*, C_{0,2})$ with the following property: For all $\delta_0 \leq \frac{1}{2}\min\{|\hat{u}(1)|, |\hat v(1)|\}$, $\alpha_0 > 0$, $C_0 > 0$ and $\eta_0 \leq \eta_{0,2}$ for some $\eta_{0,2}(\epsilon_0) > 0$, $\lambda \in [0, \lambda^*]$ and $t_0 \in [t_{0,2}, T)$, assume that 
\begin{itemize}
\item $(u,v)(x, t_0)$ is given by \eqref{def:uvt0} and $(d_0,d_1)$ is chosen such that $(\theta_{0,0}, \theta_{0,1}) \in \left[-\frac{A}{s_0^2}, \frac{A}{s_0^2}\right]^2$, where $s_0 = \ln (T-t_0)$ and $(\theta_{0,0}, \theta_{0,1})$ are the components of $(\Lambda, \Upsilon)(y,s_0)$ defined as in \eqref{decomoposeLamUp}.
\item for some $\sigma \geq s_0$, we have for all $t \in [T - e^{-\sigma}, T - e^{-(\sigma + \lambda)}]$, $$(u,v)(x,t) \in \Sc(t_0, K_0, \epsilon_0, \alpha_0, A, \delta_0, C_0, \eta_0, t).$$
\end{itemize}
Then, we have for all $s \in [\sigma, \sigma+\lambda]$,

\noindent $(i)\;$ (ODEs satisfied by the positive modes) For $n = 0, 1$, we have
$$\left|\theta_n'(s) - \left(1 - \frac{n}{2}\right)\theta_n(s)\right| \leq \frac{C}{s^2}.$$

\noindent $(ii)\;$ (ODE satisfied by the null mode) 
$$\left|\theta_2'(s) + \frac 2s\theta_2(s)\right| \leq \frac{CA^3}{s^3}.$$

\noindent $(iii)\;$ (Control of the finite dimensional part)
\begin{align*}
|\theta_j(s)| &\leq e^{-\left(\frac{j}{2} - 1 \right)(s -\tau)}|\theta_j(\tau)| + \frac{CA^{j-1}}{s^\frac{j+1}{2}},\quad 3\leq j\leq M,\\
|\tilde \theta_j(s)| &\leq e^{-\left(\frac{j}{2} +1 \right)(s -\tau)}|\tilde \theta_j(\tau)| + \frac{CA^{j-1}}{s^\frac{j+1}{2}}, \quad 3\leq j \leq M,\\
|\tilde \theta_j(s)| &\leq e^{-\left(\frac{j}{2} + 1 \right)(s -\tau)}|\tilde \theta_j(\tau)| + \frac{C}{s^2}, \quad j = 0,1,2.
\end{align*}

\noindent $(iv)\;$ (Control of the infinite dimensional part) 
\begin{align*}
&\left\|\frac{\Lambda_{-}(y,s)}{1 + |y|^{M+1}} \right\|_{L^\infty(\Rb)} + \left\|\frac{\Upsilon_{-}(y,s)}{1 + |y|^{M+1}} \right\|_{L^\infty(\Rb)}\\
& \leq Ce^{-\frac{(M+1)(s-\tau)}{4}} \left(\left\|\frac{\Lambda_{-}(y,\tau)}{1 + |y|^{M+1}} \right\|_{L^\infty(\Rb)} + \left\|\frac{\Upsilon_{-}(y,\tau)}{1 + |y|^{M+1}} \right\|_{L^\infty(\Rb)}
 \right) + \frac{CA^M}{s^\frac{M+2}{2}}.
\end{align*}

\noindent $(v)\;$ (Control of the gradient) 
\begin{align*}
\forall y \in \Rb^N, \quad |\nabla \Lambda_-(y,s)| + |\nabla \Upsilon_-(y,s)| \leq CA^{M+1}s^{-\frac{M+2}{2}} \big(|y|^{M+1} + 1\big).
\end{align*}

\noindent $(vi)\;$ (Control of the outer part)
\begin{align*}
&\|\Lambda_e(s)\|_{L^\infty(\Rb)} + \|\Upsilon_e(s)\|_{L^\infty(\Rb)}\\
& \leq Ce^{-\frac 12(s-\tau)}\left(\|\Lambda_e(\tau)\|_{L^\infty(\Rb)} + \|\Upsilon_e(\tau)\|_{L^\infty(\Rb)}\right) +  \frac{CA^{M+1}}{\sqrt{s}}(1 + s - \tau).
\end{align*}
\end{proposition}

\begin{remark} Note the the factor $\frac{2}{s}$ appearing the ODE satisfied by $\theta_2$ comes from the projection $\Pc_{2,M}$ of $V\binom{\Lambda}{\Upsilon}$ and $\binom{G_1}{G_2}$ thanks to the precise computation in Lemmas \ref{lemm:diagonal} and \ref{lemm:DefProjection}. In particular, we prove in Lemmas \ref{lemm:Pro3rdtermf2g2} and \ref{lemm:Pro6term} below that 
$$\Pc_{2,M}\left[V\binom{\Lambda}{\Upsilon}\right] \sim \frac{4}{s}\theta_2, \quad \Pc_{2,M}\binom{G_1}{G_2}\sim -\frac{2}{s}\theta_2.$$

\end{remark}

Because of the number of parameters in our problem ($p$, $q$ and $\mu$) and the coordinates in \eqref{decomoposeLamUp}, resulting in a very long proof, we will organize the rest of this subsection in three separate parts for the reader's convenience:\\

\noindent - \textit{Part 1}: We deal with system \eqref{eq:LamUp} to write ODEs satisfied by $\theta_n$ and $\tilde \theta_n$ for $n \leq M$. The definition of the projection of $\binom{\Lambda}{\Upsilon}$ on $\binom{f_n}{g_n}$ and $\binom{\tilde f_n}{g_n}$ given in Lemma \ref{lemm:DefProjection} will be the main tool to derive these ODEs. Then, we prove items $(i)$, $(ii)$ and $(iii)$ of Proposition \ref{prop:dyn}.

\noindent - \textit{Part 2}: We derive from system \eqref{eq:LamUp} a system satisfied by $(\Lambda_-,\Upsilon_-)$ and prove item $(iv)$ of Proposition \ref{prop:dyn}.  Unlike the estimate on $\theta_n$ and $\tilde \theta_n$ where we use the properties of the linear operator $\Hc + \Mc$, here we use the operator $\Hc$. The value of $M$, which is fixed large enough as in \eqref{eq:Mfixed}, is essential in the proof, in the sense that it allows us to successfully apply Gronwall's lemma. The item $(v)$ follows from a parabolic regularity argument applied to the system satisfied by  $(\Lambda_-,\Upsilon_-)$.

\noindent - \textit{Part 3}:  We derive the system satisfied by $(\Lambda_e,\Upsilon_e)$ and prove item $(vi)$ of Proposition \ref{prop:dyn}. As mentioned earlier, the linear operator of the equation satisfied by $\Lambda_e$ and $\Upsilon_e$ has a negative spectrum, which makes the control of $\|\Lambda_e(s)\|_{L^\infty(\Rb)}$ and $\|\Upsilon_e(s)\|_{L^\infty(\Rb)}$ easy.\\

Note that system \eqref{eq:LamUp} is analogous to the one in \cite{GNZpre16c}, except for the extra nonlinear gradient term $\binom{G_1}{G_2}$. One of them concerns the shrinking set introduced in Definition \ref{def:St} which involves an extra gradient estimate in $\Dc_1$ and additional estimates in $\Dc_2$ and $\Dc_3$. When taking into account this new definition, we shall use some estimates similar to those obtained in \cite{GNZpre16c} and only focus on the novelties. We would like to mention that our handling of the gradient term is inspired  by the technique developed by Tayachi and Zaag \cite{TZpre15} (see also \cite{TZnor15}) for the following nonlinear heat equation
\begin{equation*}
\partial_t u = \Delta u + |u|^{p-1}u + \mu |\nabla u|^\frac{2p}{p+1}, \quad \mu > 0, p > 3.
\end{equation*}
In \cite{GNZjde17}, we adapt the techinique of \cite{TZpre15} to handle the case when $p \to +\infty$, namely the equation
\begin{equation*}
\partial_t u = \Delta u + e^u + \mu |\nabla u|^2, \quad \mu > -1.
\end{equation*}

\subsubsection{Control of the finite dimensional part.}
In this subsection we give the proof of items $(i)-(iii)$ of Proposition \ref{prop:dyn}. In particular, we will estimate the main contribution to the projections $\Pc_{n, M}$ and $\tilde \Pc_{n,M}$ (see Lemma \ref{lemm:DefProjection} for the definition) of all terms appearing in \eqref{eq:LamUp}, then the conclusion simply follows by addition.

\subparagraph{$\bullet\;$ The derivative term $\partial_s\binom{\Lambda}{\Upsilon}$.} From the decomposition \eqref{decomoposeLamUp} and Lemma \ref{lemm:DefProjection}, we have
\begin{equation}
\Pc_{n,M}\left[\partial_s\binom{\Lambda}{\Upsilon}\right] = \theta_n' \quad \text{and} \quad \tilde \Pc_{n,M}\left[\partial_s\binom{\Lambda}{\Upsilon}\right] = \tilde \theta_n'.
\end{equation}

\subparagraph{$\bullet\;$ The linear term $(\Hc + \Mc)\binom{\Lambda}{\Upsilon}$.} We claim the following:
\begin{lemma}[Projections of $(\Hc + \Mc)\binom{\Lambda}{\Upsilon}$ on $\binom{f_n}{g_n}$ and $\binom{\tilde f_n}{\tilde g_n}$ for $n \leq M$] \label{lemm:Pro2ndterm} For all $n \leq M$,\\

\noindent $(i)\;$ It holds that
\begin{align}
&\left|\Pc_{n,M}\left[(\Hc + \Mc)\binom{\Lambda}{\Upsilon} \right] - \left(1 - \frac n2\right)\theta_n(s) \right| \nonumber \\
& \quad  +\left|\tilde \Pc_{n,M}\left[(\Hc + \Mc)\binom{\Lambda}{\Upsilon} \right] - \left(1 + \frac n2\right)\tilde \theta_n(s) \right| \nonumber \\
& \qquad \leq C\left\|\frac{\Lambda_-(y,s)}{1 + |y|^{M+1}} \right\|_{L^\infty(\Rb)} + C\left\|\frac{\Upsilon_-(y,s)}{1 + |y|^{M+1}} \right\|_{L^\infty(\Rb)}.\label{est:2ndterm}
\end{align}

\noindent $(ii)\;$ For all $A \geq 1$, there exists $s_4(A) \geq 1$ such that for all $s \geq s_4(A)$, if $\binom{\Lambda(s)}{\Upsilon(s)} \in \Vc_A(s)$, then:

\begin{align}
&\left|\Pc_{n,M}\left[(\Hc + \Mc)\binom{\Lambda}{\Upsilon} \right] - \left(1 - \frac n2\right)\theta_n(s) \right| \nonumber \\
& \quad  +\left|\tilde \Pc_{n,M}\left[(\Hc + \Mc)\binom{\Lambda}{\Upsilon} \right] - \left(1 + \frac n2\right)\tilde \theta_n(s) \right| \leq C\frac{A^{M+1}}{s^\frac{M+2}{2}}.\label{est:2ndtermVA}
\end{align}
\end{lemma}

\begin{proof} The proof follows exactly the same lines as in \cite{GNZpre16c}. The only difference is the eigenvalues of the matrix $\Mc$ which are given by $\pm 1$. We refer the readers to Lemma 5.4 in \cite{GNZpre16c} for all the details of the proof. 
\end{proof}

\subparagraph{$\bullet\;$ The potential term $V(y,s)\binom{\Lambda}{\Upsilon}$.} We claim the following:

\begin{lemma}[Expansion of the potential term $V(y,s)$] \label{lemm:estVys} Let $V(y,s)$ be defined by \eqref{def:Vys}, we have
\begin{equation}\label{est:Viorder1}
i = 1,2,3,4, \quad |V_i(y,s)| \leq \frac{C(1 + |y|^2)}{s}, \quad \forall y \in \Rb, \; s \geq 1,
\end{equation}
and for all $k \in \mathbb{N}^*$, 
\begin{equation}\label{est:Viorderk}
i = 1,2,3,4, \quad V_i(y,s) = \sum_{j = 1}^k \frac{1}{s^j}W_{i,j}(y) + \tilde{W}_{i,k}(y,s),
\end{equation}
where $W_{i,j}(y)$ is an even polynomial of degree $2j$, and $\tilde{W}_{i,k}(y,s)$ satisfies the estimate
$$|\tilde{W}_{i,k}(y,s)| \leq \frac{C(1 + |y|^{2k + 2})}{s^{k+1}}, \quad \forall |y| \leq \sqrt{s}, \; s \geq 1.$$
Moreover, we have for all $|y| \leq \sqrt{s}$ and $s \geq 1$,
\begin{equation}\label{exp:Vys}
\left|V(y,s) +  \frac{1}{2(\mu + 1)s} \begin{pmatrix} h_2 & \frac qp \hat h_2\\ \frac pq h_2 & \hat h_2
\end{pmatrix}\right|  \leq \frac{C(1 + |y|^4)}{s^2}.
\end{equation}
\end{lemma}
\begin{proof} The proof simply follows from Taylor expansions and we refer to Lemma 5.5 in \cite{GNZpre16c} for a similar proof. 
\end{proof}

We now use Lemma \ref{lemm:estVys} to derive the projections of $V\binom{\Lambda}{\Upsilon}$ on $\binom{f_n}{g_n}$ and $\binom{\tilde{f}_n}{\tilde{g}_n}$. More precisely, we have the following:

\begin{lemma}[Projections of $V\binom{ \Lambda}{\Upsilon}$ on $\binom{f_n}{g_n}$ and $\binom{\tilde{f}_n}{\tilde{g}_n}$] \label{lemm:Pro3rdterm} $\quad$\\

\noindent $(i)\;$ For all $s \geq 1$ and $n \leq M$, we have

\begin{align*}
&\left|\Pc_{n,M} \left[V\binom{\Lambda}{\Upsilon} \right]\right| + \left|\tilde \Pc_{n,M}\left[ V\binom{ \Lambda}{ \Upsilon}\right] \right|\nonumber\\
& \quad  \leq \frac{C}{s}\sum_{i = n-2}^M\big(|\theta_i(s)| + |\tilde \theta_i(s)|\big) + \sum_{i = 0}^{n-3}\frac{C}{s^{\frac{n-i}{2}}}\big(|\theta_i(s)| + |\tilde \theta_i(s)|\big)\\
& \qquad  + \frac{C}{s}\left(\left\|\frac{\Lambda_-(y,s)}{1 + |y|^{M+1}} \right\|_{L^\infty(\Rb)} + \left\|\frac{\Upsilon_-(y,s)}{1 + |y|^{M+1}} \right\|_{L^\infty(\Rb)}\right).
\end{align*}

\noindent $(ii)\;$  For all $A \geq 1$, there exists $s_5(A) \geq 1$ such that for all $s \geq s_5(A)$, if $\binom{\Lambda(s)}{\Upsilon(s)} \in \Vc_A(s)$, then:\\
- for $3 \leq n \leq M$,
\begin{equation*}
\left|\Pc_{n,M} \left[V\binom{\Lambda}{\Upsilon} \right]\right| + \left|\tilde \Pc_{n,M}\left[ V\binom{ \Lambda}{ \Upsilon}\right] \right|\leq \frac{CA^{n - 2}}{s^\frac{n + 1}{2}}.
\end{equation*}
- for $n = 0, 1, 2$,
\begin{equation*}
\left|\Pc_{n,M}\left[V\binom{\Lambda}{\Upsilon}\right] \right| + \left|\tilde \Pc_{n,M}\left[V\binom{\Lambda}{\Upsilon} \right] \right|\leq \frac{C}{s^2}.
\end{equation*}
\end{lemma}
\begin{proof}  The argument of the proof is the same as the one written in \cite{GNZpre16c} although we have a slightly different definition of the potential term $V$. However, since we have an analogous expansion of $V$ given in Lemma \ref{lemm:estVys}, the readers will have no difficulties to adapt those proof to this new situation. We then refer to Lemma 5.6 in \cite{GNZpre16c} for all the details of the proof.

\end{proof}

Using the precise expansion \eqref{exp:Vys}, we are able to derive a sharp estimate for the projection of $V\binom{ \Lambda}{\Upsilon}$ on $\binom{f_2}{g_2}$. In particular, we claim the following.

\begin{lemma}[Refined projection of $V\binom{ \Upsilon}{\Lambda}$ on $\binom{f_2}{g_2}$]\label{lemm:Pro3rdtermf2g2} $\quad $\\
$(i)\;$ It holds that 
\begin{align*}
\quad \left|\Pc_{2,M}\left[ V\binom{\Lambda}{\Upsilon}\right] + \frac{4}{s}\theta_2(s)\right|& \leq \frac{C}{s} \left(\sum_{j = 0, j \ne 2}^M|\theta_j(s)| + \sum_{j = 0}^M|\tilde \theta_j(s)|\right)\\
& + \frac Cs\left(\left\|\frac{\Lambda_-(y,s)}{1 + |y|^{M+1}} \right\|_{L^\infty(\Rb)}+\left\|\frac{\Upsilon_-(y,s)}{1 + |y|^{M+1}} \right\|_{L^\infty(\Rb)} \right).
\end{align*}
$(ii)\;$ For all $A \geq 1$, there exists $s_6(A) \geq 1$ such that for all $s \geq s_6(A)$, if $\binom{\Lambda(s)}{\Upsilon(s)} \in \Vc_A(s)$, then:
\begin{equation*}
\quad \left|\Pc_{2,M}\left[V\binom{\Lambda}{\Upsilon}\right] + \frac{4}{s}\theta_2(s)\right| \leq \frac{CA^3}{s^3}.
\end{equation*}
\end{lemma}
\begin{proof} See Lemma 5.7 in \cite{GNZpre16c} for a similar proof. The readers should notice that the only difference in comparison with the proof written in that paper is the expansion \eqref{exp:Vys} which results in
\begin{align*}
&\Pc_{2,M}\left[\frac{\theta_2}{2(\mu + 1)s}\begin{pmatrix} h_2 & \frac qp \hat h_2\\ \frac pq h_2 & \hat h_2
\end{pmatrix} \binom{f_2}{g_2}\right] \\
& = \frac{\theta_2}{2(\mu + 1)s}\Pc_{2,M}\binom{2q[h_4 + (10 - 2\mu)h_2 + 8]}{2p[\hat h_4 + (10 \mu - 2)\hat h_2 + 8\mu^2]}\\
& =\frac{\theta_2}{2(\mu + 1)s} \left[2q A_{4,2} + 2p B_{4,2} + 2q(10 - 2\mu)A_{2,2} + 2p(10 - 2\mu)B_{2,2}\right] = \frac{4}{s}\theta_2.
\end{align*}
This concludes the proof of Lemma \ref{lemm:Pro3rdtermf2g2}.
\end{proof}

\subparagraph{$\bullet\;$ The nonlinear term $\binom{q}{p}\Lambda\Upsilon$.} 
We claim the following:

\begin{lemma}[Projections of $\binom{q}{p}\Lambda \Upsilon$ on $\binom{f_n}{g_n}$ and $\binom{\tilde f_n}{\tilde{g}_n}$] \label{lemm:Pro4thterm} For all $A \geq 1$, there exists $s_7(A) \geq 1$ such that for all $s \geq s_7(A)$, if $\binom{\Lambda(s)}{\Upsilon(s)} \in \Vc_A(s)$, then:\\
- for $3 \leq m \leq M$, 
\begin{equation*}
\quad \left|\Pc_{m,M}\left[\binom{q}{p}\Lambda \Upsilon\right]\right| + \left|\tilde \Pc_{m,M}\left[\binom{q}{p}\Lambda \Upsilon\right]\right| \leq \frac{CA^n}{s^\frac{n + 2}{2}},
\end{equation*}
- for $m = 0, 1, 2$,
\begin{equation*}
\quad \left|\Pc_{m,M}\left[\binom{q}{p}\Lambda \Upsilon\right]\right| + \left|\tilde \Pc_{m,M}\left[\binom{q}{p}\Lambda \Upsilon\right]\right| \leq \frac{C}{s^3}.
\end{equation*}
\end{lemma}
\begin{proof} From Lemma \ref{lemm:DefProjection}, it is enough to estimate $\Pi_m(\Lambda \Upsilon)$ and $\hat \Pi_m (\Lambda \Upsilon)$ with $m \leq M$, since it implies the same estimate for  $\Pc_{m,M}$ and $\tilde{\Pc}_{m,M}$. Since the estimates for $\Pi_m$ and $\hat \Pi_m$ are the same, we only deal with $\Pi_m(\Lambda \Upsilon)$ which is defined as follows:
$$\Pi_m(\Lambda \Upsilon) = \|h_m\|^{-2}_{\rho_1}\int_{\Rb}\Lambda \Upsilon h_m \rho_1 dy.$$
By the decomposition \eqref{decomoposeLamUp} and part $(i)$ of Definition \ref{def:St}, we write for $0 \leq m \leq M$,
\begin{align*}
\Lambda \Upsilon &= \left(\sum_{i = 0}^M \theta_i f_i + \tilde{\theta}_i \tilde{f}_i + \Lambda_- \right)\left(\sum_{j = 0}^M \theta_j g_j + \tilde{\theta}_j \tilde{g}_j + \Upsilon_- \right)\\
& = \left(\sum_{i = 0}^M \alpha_i y^i + \Lambda_-\right) \left(\sum_{j = 0}^M \beta_j y^j + \Upsilon_-\right) \\
& = \sum_{i + j = 0}^{2M} \alpha_i \beta_j y^{i + j} + \Oc\left(\frac{A^{2(M+1)}\ln s}{s^{\frac{M+2}{2} + 2}}  \big(|y|^{2M + 1}  +1\big)\right),
\end{align*}
where $|\alpha_i|, |\beta_i| \leq \frac{CA^4\ln s}{s^2}$ for $i = 0, 1, 2$ and $|\alpha_i|, |\beta_i| \leq \frac{CA^i}{s^{\frac{i + 1}{2}}}$ for $3 \leq i \leq M$. From Remark \ref{rema:orth}, we deduce that 
$$\left|\Pi_m(\Lambda \Upsilon)\right| \leq C\sum_{i + j = m}^{2M}|\alpha_i \beta_j| + C\frac{A^{2(M+1)}\ln s}{s^{\frac{M+2}{2} + 2}} \leq \left\{ \begin{array}{ll} \frac{CA^m}{s^{\frac{m + 2}{2}}} &\quad \text{for} \; 3 \leq m \leq M\\
\frac{CA^8 \ln^2 s}{s^4} &\quad \text{for}\; m = 0, 1, 2.
\end{array}\right.$$
This concludes the proof of Lemma \ref{lemm:Pro4thterm}.

\end{proof}

\subparagraph{$\bullet\;$ The error term $\binom{R_1}{R_2}$.}
We first expand $R_1(y,s)$ and $R_2(y,s)$ as a power series of $\frac{1}{s}$ as $s \to +\infty$, uniformly for $|y| < \sqrt s$. More precisely, we claim the following:
\begin{lemma}[Expansion of $R_1$ and $R_2$ as $s \to +\infty$] \label{lemm:expandR1R2}
For all $m \in \mathbb{N}$, the functions $R_1(y,s)$ and $R_2(y,s)$ defined in \eqref{def:Rys} can be decomposed as follows: for all $|y| < \sqrt s$ and $s \geq 1$,
\begin{equation}\label{eq:Riexpand}
\quad \left|R_{i}(y,s) - \sum_{k = 1}^{m-1} \frac{1}{s^{k+1}}R_{i,k}(y)\right| \leq \frac{C(1 + |y|^{2m})}{s^{m + 1}},
\end{equation}
where $R_{i,k}$ is a polynomial of degree $2k$. More precisely, we have
\begin{align}
R_{1,1} = \frac{\mu(2 + \mu)}{p(1 + \mu)^2} + \frac{1 - \mu^2}{p(1 + \mu)^3}y^2,&\label{eq:R1exps3}\\
R_{2,1} = \frac{1 + 2\mu}{q(1 + \mu)^2} + \frac{\mu^2 - 1}{q(1 + \mu)^3}y^2.\label{eq:R2exps3}
\end{align}
\end{lemma}
\begin{proof} Let $z = \frac{y}{\sqrt s}$, $D = \frac{\mu}{p(\mu + 1)}$, $E = \frac{1}{q(\mu + 1)}$, we then write from \eqref{def:phipsi}, 
$$\phi(y,s) = \Phi^*(z) + \frac{D}{s}, \quad \psi(y,s) = \Psi^*(z) + \frac{E}{s},$$
where $\Phi^*$ and $\Psi^*$ are defined by \eqref{def:PhiPsistarpro}. Using the fact that $(\Phi^*, \Psi^*)$ satisfies \eqref{eq:Phi0Psi0}, we rewrite 
\begin{align*}
R_1(y,s) &= \frac{z}{2s} \cdot\nabla_z\Phi^* + \frac{D}{s^2} + \frac{1}{s}\Delta_z \Phi^* - \frac{D}{s} + \frac{qDE}{s^2} + \frac{qD}{s}\Psi^* + \frac{qE}{s}\Phi^* - \frac{|\nabla_z \Phi^*|^2}{s (\Phi^* + \frac{D}{s})},\\
R_2(y,s) &= \frac{z}{2s} \cdot \nabla_z\Psi^* + \frac{E}{s^2} + \frac{\mu}{s}\Delta_z \Psi^* - \frac{E}{s} + \frac{pDE}{s^2} + \frac{pD}{s}\Psi^* + \frac{pE}{s}\Phi^* - \mu\frac{|\nabla_z \Psi^*|^2}{s (\Psi^* + \frac{E}{s})}.
\end{align*}
The proof then follows from Taylor expansion of $R_i, i = 1,2$ near $z = 0$. Note that the term of order $\frac{1}{s}$ is identically zero.  This concludes the proof of Lemma \ref{lemm:expandR1R2}.
\end{proof}

From Lemma \ref{lemm:expandR1R2}, we directly derive the following estimate of the projections of $\binom{R_1}{R_2}$ on $\binom{f_n}{g_n}$ and $\binom{\tilde f_n}{\tilde g_n}$:
\begin{lemma}[Projections of $\binom{R_1}{R_2}$ on $\binom{f_n}{g_n}$ and $\binom{\tilde f_n}{\tilde g_n}$] \label{lemm:Pro5thterm} For all $s \geq 1$ and $n \leq M$, we have\\
- if $n$ is odd, then
\begin{equation}\label{est:PnMR1R1nodd}
\Pc_{n,M}\binom{R_1(y,s)}{R_2(y,s)} = \tilde{\Pc}_{n,M}\binom{R_1(y,s)}{R_2(y,s)} = 0,
\end{equation}
- if $n \geq 4$ is even, then
\begin{equation}\label{est:PnMR1R1n4}
\left|\Pc_{n,M}\binom{R_1(y,s)}{R_2(y,s)} \right| + \left|\tilde \Pc_{n,M}\binom{R_1(y,s)}{R_2(y,s)} \right| \leq \frac{C}{s^{\frac{n + 2}{2}}}.
\end{equation}

- if $n = 0$ and $n = 2$, then
\begin{equation}\label{est:PnMR1R1n02}
\left|\Pc_{0,M}\binom{R_1(y,s)}{R_2(y,s)} \right| + \left|\tilde \Pc_{0,M}\binom{R_1(y,s)}{R_2(y,s)} \right| + \left|\tilde \Pc_{2,M}\binom{R_1(y,s)}{R_2(y,s)} \right| \leq \frac{C}{s^2},
\end{equation}
and
\begin{equation}\label{est:PnMR1R1n2}
\left|\Pc_{2,M}\binom{R_1(y,s)}{R_2(y,s)} \right| \leq \frac{C}{s^3}.
\end{equation}

\end{lemma}
\begin{proof} The proof simply follows from the expansion \eqref{eq:Riexpand} and Lemma \ref{lemm:DefProjection}. For the sharp estimate \eqref{est:PnMR1R1n2}, we need to use the precise expressions \eqref{eq:R1exps3} and \eqref{eq:R2exps3} which gives 
$$\Pc_{2,M}\binom{R_{1,1}}{R_{1,2}} = 0.$$
This concludes the proof of Lemma \ref{lemm:Pro5thterm}.  
\end{proof}

\subparagraph{$\bullet\;$ The nonlinear gradient term $\binom{G_1}{G_2}$.} In comparison with the work \cite{GNZpre16c}, this part is new. We shall give all details of the proof.
\begin{lemma}[Expansion of $\binom{G_1}{G_2}$] \label{lemm:expG1G2} For all $K_0 \geq 1$, $A \geq 1$ and $\epsilon_0 > 0$, there exists $t_{0,3}(K_0,A, \epsilon_0) < T$ and $\eta_{0,3}(\epsilon_0)$ such that for each $t_0 \in [t_{0,3}, T)$, $\alpha_0 > 0$, $C_0 > 0$, $C_0' > 0$, $\delta_0 \leq \min\{|\hat u(1)|,|\hat v(1)|\}$ and $\eta_0 \in (0, \eta_{0,3}]$: if $(u,v)(x,t_0)$ is given by \eqref{def:uvt0} and $(u,v)(t) \in \Sc(t)$ for $t \in [t_0, T)$, then we have \\
\begin{align}
&|\chi(y,s)G_1(\Lambda,y,s)|\leq C(K_0,A)\chi(y,s) \left(\frac{|\Lambda|}{s} + \frac{|\nabla \Lambda|}{\sqrt s}\right),\label{est:G1chi}\\
&|\chi(y,s)G_2(\Upsilon,y,s)|\leq C(K_0,A)\chi(y,s) \left(\frac{|\Upsilon|}{s} + \frac{|\nabla \Upsilon|}{\sqrt s}\right),\label{est:G2chi}\\
&\left|\big(1 - \chi(y,s)\big)\binom{G_1(\Lambda,y,s)}{G_2(\Upsilon, y,s)}\right| \leq \frac{C(K_0,C_0')}{s},\label{est:G12chi1}
\end{align}
and  for $k \in \mathbb{N}^*$,
\begin{align}
&\left|\chi(y,s) \left\{G_1(\Lambda,y,s) - \sum_{j = 1}^k \frac{1}{j!} \left[D_j \frac{|\nabla \phi|^2}{\phi^{j + 1}} \Lambda^j + D_{j - 1} \frac{2\nabla \Lambda \cdot \nabla \phi}{\phi^j}\Lambda^{j - 1} + D_{j - 2} \frac{2|\nabla \Lambda|^2}{\phi^{j - 2}}\Lambda^{j - 2} \right] \right\} \right|\nonumber\\
& \quad \leq C(K_0,A)\chi(y,s)\left( \frac{1}{s}|\Lambda|^{k+1} + \frac{|y|^2}{s^2}|\Lambda|^k + |\Lambda|^{k -1}|\nabla \Lambda|^2\right),\label{eq:expG1}\\ 
&\left|\chi(y,s) \left\{G_2(\Upsilon,y,s) - \mu \sum_{j = 1}^k \frac{1}{j!} \left[D_j \frac{|\nabla \psi|^2}{\psi^{j + 1}} \Upsilon^j + D_{j - 1} \frac{2\nabla \Upsilon \cdot \nabla \psi}{\psi^j}\Upsilon^{j - 1} + D_{j - 2} \frac{2|\nabla \Upsilon|^2}{\psi^{j - 2}}\Upsilon^{j - 2} \right] \right\} \right|\nonumber\\
& \quad \leq C(K_0,A)\chi(y,s)\left( \frac{1}{s}|\Upsilon|^{k+1} + \frac{|y|^2}{s^2}|\Upsilon|^k +  |\Upsilon|^{k -1}|\nabla \Upsilon|^2\right),\label{eq:expG2}
\end{align}
where $D_j = (-1)^{j+1}j!$ and $D_{-1} = 0$.
\end{lemma}
\begin{proof} We only deal with the estimates for $G_1$, the estimates for $G_2$ follows similarly. Let $\nu \in [0,1]$ and
$$\Gc_1(\nu) = -\frac{|\nu \nabla \Lambda + \nabla \phi|^2}{\nu \Lambda + \phi} + \frac{|\nabla \phi|^2}{\phi}.$$
We have by \eqref{def:Gys},
$$G_1(y,s) \equiv \Gc_1(1) = \sum_{j = 0}^k \frac{1}{j!} \Gc_1^{(j)}(0) + \frac{1}{(k+1)!}\int_0^1(1 - \nu)\Gc_1^{(k+1)}(\nu)d\nu, \quad \forall k \in \mathbb{N},$$
where $\Gc_1(0) = 0$ and for $j \geq 1$, 
\begin{align*}
\Gc_1^{(j)}(\nu) &= D_j\Lambda^j \frac{|\nu \nabla \Lambda + \nabla \phi|^2}{(\nu \Lambda + \phi)^{j + 1}} + D_{j - 1}\Lambda^{j - 1}\frac{2\nabla \Lambda \cdot (\nu \nabla \Lambda + \nabla \phi)}{(\nu \Lambda + \phi)^{j}} + D_{j - 2}\Lambda^{j - 2} \frac{2|\nabla \Lambda|^2}{(\nu \Lambda + \phi)^{j - 1}},
\end{align*}
with $D_j = (-1)^{j+1}j!$ and $D_{-1} = 0$ by convention. 
The estimate \eqref{est:G1chi} and the expansion \eqref{eq:expG1} then follow from the fact that 
$$ |\nabla \phi| \leq \frac{C}{\sqrt{s}}, \quad \frac{|\nabla \phi|^2}{\phi} \leq \frac{C}{s}, \quad \frac{|\nabla \phi|^2}{\phi^2} \leq \frac{C|y|^2}{s^2}, \quad \forall y \in \Rb, \; s \geq s_{0,3}(K_0).$$
In order to prove \eqref{est:G12chi1}, it remains to show that for $|y| \geq K_0\sqrt s$, $\frac{|\nabla \Lambda + \nabla \phi|^2}{\Lambda + \phi} \leq \frac{C}{s}$. From \eqref{def:LamUps}, \eqref{def:simvariables} and \eqref{def:baruv}, it is equivalent to show that for $|x| \geq r(t) = K_0 \sqrt{(T-t)|\ln(T-t)|}$ and $t \geq t_0$,
\begin{equation}\label{est:nabueu}
|\nabla_x u|^2e^{qu(x,t)} \leq \frac{C}{(T-t)^2|\ln(T-t)|}.
\end{equation}
Arguing as in \cite{GNZjde17}, we consider two cases:\\
- \textit{Case 1: $|x| \in [r(t), \epsilon_0)]$}. In this case we use the bounds given in part $(ii)$ of Definition \ref{def:St} to prove \eqref{est:nabueu}. By \eqref{def:uvtilde}, we have 
$$|\nabla_x u(x,t)|^2e^{qu(x,t)} = \sigma(x)^{-2}|\nabla_\xi \tilde{u}(x,0, \tau(x,t))|^2e^{q\tilde{u}(x, 0, \tau(x,t))},$$
where  $\tau(x,t) = \frac{t - t(x)}{\sigma(x)}$, $\sigma(x) = T - t(x)$ and $t(x)$ is uniquely defined by \eqref{def:tx}. From part $(ii)$ of Definition \ref{def:St}, we have for $|x| \in [r(t), \epsilon_0]$, 
$$|\tilde{u}(x, 0, \tau(x,t)) - \hat u(\tau(x,t))| \leq \delta_0, \quad |\nabla_\xi \tilde{u}(x,0, \tau(x,t))| \leq \frac{C_0}{\sqrt{|\ln \sigma(x)|}},$$
from which we derive 
$$|\nabla_x u(x,t)|^2e^{qu(x,t)} \leq \frac{C(C_0)}{\sigma(x)^2|\ln \sigma(x)|} \leq \frac{C(C_0)}{\sigma(r(t))^2 |\ln \sigma(r(t))|}.$$
Since $r(t) \to 0$ as $t \to T$, we deduce from \eqref{def:tx}, 
$$\sigma(r(t)) \sim \frac{2}{K_0^2}\frac{r^2(t)}{|\ln r(t)|} \quad \text{and} \quad \ln \sigma(r(t)) \sim \ln r(t) \quad \text{as} \;\; t\to T. $$
Recalling that $r(t) = K_0\sqrt{(T-t)|\ln(T-t)|}$, we derive
$$\frac{C(C_0)}{\sigma(r(t))^2 |\ln \sigma(r(t))|} \sim \frac{C(C_0, K_0)}{(T-t)^2|\ln(T-t)|},$$
which concludes the proof of \eqref{est:nabueu} for $|x| \in [r(t), \epsilon_0]$.\\
- \textit{Case 2: $|x| \geq \epsilon_0$}. From part $(iii)$ of Definition \ref{def:St}, we have 
$$i = 0, 1, \quad |\nabla^i_x u(x,t) - \nabla_x^iu(x,t_0)| \leq \eta_0, \quad \forall |x| \geq \epsilon_0.$$
Let 
$$\eta_{0,3}(\epsilon_0) = \frac{1}{2}\min\{\min_{|x| \geq \epsilon_0}|u(x,t_0)|, \min_{|x| \geq \epsilon_0}|\nabla u(x,t_0)|\},$$
then from \eqref{def:uvt0}, we have for $\eta_0 \in (0, \eta_{0,3}]$ and $|x| \geq \epsilon_0$, 
$$|\nabla_x u(x,t)|^2e^{qu(x,t)} \leq C|\nabla_x u(x,t_0)|^2e^{qu(x,t_0)} \leq C|\nabla_x \hat u_*(x)|^2e^{q\hat u_*(x)} \leq C(\epsilon_0),$$
where $\hat u_*$ is defined by \eqref{def:ustar}. Therefore, if $t_0 \in [t_{0,3}, T)$, where $t_{0,3} = t_{0,3}(\epsilon_0)< T$ such that $C(\epsilon_0) \leq \frac{C}{(T-t_{0,3})^2|\ln(T - t_{0,3})|}$, we have proved \eqref{est:nabueu} for $t = t_0$ and $|x| \geq \epsilon_0$. Since $\frac{C}{(T-t_{0})^2|\ln(T - t_{0})|} \leq \frac{C}{(T-t)^2|\ln(T - t)|}$ for all $t \in [t_0, T)$, we concludes that estimate \eqref{est:nabueu} holds true for $t \geq t_0$ and $|x| \geq \epsilon_0$. This concludes the proof of Lemma \ref{lemm:expG1G2}.
\end{proof}

From Lemma \ref{lemm:expG1G2}, we are ready to estimate the projection of $\binom{G_1}{G_2}$ on $\binom{f_n}{g_n}$ and $\binom{\tilde f_n}{\tilde g_n}$. In particular, we have the following.
\begin{lemma}[Projection of $\binom{G_1}{G_2}$ on $\binom{f_n}{g_n}$ and $\binom{\tilde f_n}{\tilde g_n}$] \label{lemm:Pro6term} Under the assumption of Lemma \ref{lemm:expG1G2}, we have \\
- For $n = 0, 1, 2$,
\begin{equation}\label{eq:P01G12}
\quad \left|\Pc_{n,M}\binom{G_1}{G_2}\right| + \left|\tilde{\Pc}_{n,M}\binom{G_1}{G_2}\right| \leq \frac{C}{s^2}.
\end{equation}
- For $3 \leq n \leq M$,
\begin{equation}\label{eq:P3MG12}
\quad \left|\Pc_{n,M}\binom{G_1}{G_2}\right| + \left|\tilde{\Pc}_{n,M}\binom{G_1}{G_2}\right| \leq \frac{CA^n}{s^{\frac{n + 2}{2}}}.
\end{equation}
Moreover, we have the refined estimate 
\begin{equation}\label{eq:P2G12}
\quad \left|\Pc_{2,M}\binom{G_1}{G_2} - \frac{2}{s}\theta_2\right| \leq \frac{CA^2}{s^3}.
\end{equation}

\end{lemma}
\begin{proof} From \eqref{est:G1chi}, \eqref{est:G2chi}, \eqref{est:G12chi1}, part $(ii)$ of Proposition \ref{prop:proSt} and Lemma \ref{lemm:DefProjection}, we derive for $n = 0, 1, 2$,
$$|\Pc_{n,M}\binom{G_1}{G_2}| + |\tilde{\Pc}_{n,M}\binom{G_1}{G_2}| \leq \frac{C(A) \ln s}{s^2\sqrt s} + C(A)e^{-cs} \leq \frac{1}{s^2}.$$
We can refine the estimate for $\Pc_{2,M}$ by using the expansions \eqref{eq:expG1} and \eqref{eq:expG2} with $k = 1$ which reads as follows: 
\begin{align*}
&\left|\chi(y,s)\left[G_1(\Lambda, y,s) + \frac{2\nabla \Lambda \cdot \nabla \phi}{\phi}\right]\right| \leq C(A, K_0) \left(\frac{|\Lambda|^2}{s} + \frac{|y|^2}{s^2}|\Lambda| + |\nabla \Lambda|^2\right),\\
&\left|\chi(y,s)\left[G_2(\Upsilon, y,s) + \frac{2\mu\nabla \Upsilon \cdot \nabla \psi}{\psi}\right]\right| \leq C(A, K_0) \left(\frac{|\Upsilon|^2}{s} + \frac{|y|^2}{s^2}|\Upsilon| + |\nabla \Upsilon|^2\right).
\end{align*} 
From these expansions, part $(i)$ of Definition \ref{def:St}, decomposition \ref{decomoposeLamUp} and the fact that $\frac{|\nabla \phi|}{\phi} + \frac{|\nabla \psi|}{\psi} \leq \frac{C|y|}{s}$ for all $y \in \Rb$, we derive 
\begin{align*}
\Pc_{2,M}\binom{G_1}{G_2} &= -2\theta_2\Pc_{2,M}\binom{(\nabla f_2 \cdot \nabla \phi) / \phi}{\mu(\nabla g_2 \cdot\nabla \psi) / \psi} + \Oc\left(\frac{CA^2}{s^3} \right),\\
& = \frac{4}{(1 + \mu)s}\theta_2\Pc_{2,M}\binom{qy^2}{\mu py^2} + \Oc\left(\frac{CA^2}{s^3} \right) = \frac{2}{s}\theta_2 + \Oc\left(\frac{CA^2}{s^3} \right),
\end{align*}
which is the desired conclusion in \eqref{eq:P2G12}.

For $3 \leq n \leq M$, we note from Lemma \ref{lemm:DefProjection} that it is enough to estimate $\Pi_{n}(G_1)$ and $\tilde{\Pi}_{n}(G_2)$ which directly implies the estimates  for $\Pc_{n,M}\binom{G_1}{G_2}$ and $\tilde{\Pc}_{n,M}\binom{G_1}{G_2}$. Since the estimates for $\Pi_{n}(G_1)$ and $\tilde{\Pi}_n(G_2)$ are similar, we only deal with $\Pi_n(G_1)$. From \eqref{est:G12chi1}, we have $\int_{|y| \geq K_0 \sqrt s} G_1 h_n \rho_1 dy \leq Ce^{-cs}$. We now use the expansion \eqref{eq:expG1} for the estimates in the region $|y| \leq 2K_0 \sqrt{s}$. To do so, let us expand $G_1(y,s)$ for $|y| \leq 2K_0\sqrt{s}$ in  power series of $y$ for $|y| \leq 2K\sqrt s$. We start with the term $\frac{|\nabla \phi|^2}{\phi^{j + 1}} \Lambda^j$ for $j \geq 1$. By the definition \eqref{eq:Phi0Psi0}, we write
\begin{align*}
\frac{|\nabla \phi|^2}{\phi^{j + 1}} &= \sum_{k = 0}^{M/2} \frac{c_k}{s^{k+2}} y^2(\Phi^*)^{k + 4 - (j+1)} +  \Oc\left(\frac{|y|^2}{s^{M/2 + 3}}\right) = \sum_{m = 1}^{M/2}\frac{\tilde{c}_m}{s^{m + 1}}y^{2m} + \Oc\left(\frac{|y|^2}{s^{M/2 + 2}}\right),
\end{align*}
and from part $(i)$ of Definition \ref{def:St} and the decomposition \eqref{decomoposeLamUp}, 
\begin{align*}
\Lambda^j = \left[\sum_{i = 0}^M \alpha_i y^i + \Oc\left(\frac{|y|^{M+1} + 1}{s^{M/2 + 1}}\right)\right]^j = \sum_{i = 0}^{M}\tilde{\alpha}_i y^i + \Oc\left(C(A)\frac{|y|^{M+1} + 1}{s^{M/2 + 1}}\right),
\end{align*}
with $|\tilde{\alpha}_i| \leq \frac{C(A)}{s^{(i+1)/2}}$.  Hence, we have 
$$\frac{|\nabla \phi|^2}{\phi^{j + 1}} \Lambda^j = \sum_{m = 2}^{M} d_m y^m + \Oc\left(C(A)\frac{|y|^{M+1} + 1}{s^{M/2 + 2}}\right) \quad \text{with} \quad |d_m| \leq \frac{C(A)}{s^{\frac{m + 3}{2}}},$$
from which we directly obtain the estimate
$$\left| \Pi_{m}\left(\frac{|\nabla \phi|^2}{\phi^{j + 1}} \Lambda^j\right) \right| \leq \frac{C}{s^{\frac{m + 2}{2}}} \quad \text{for}\quad 3 \leq m \leq M.$$
A similar computation yields the same bound on the projection $\Pi_m, 3 \leq m \leq M$ of the terms $\frac{\nabla \Lambda \cdot \nabla \phi}{\phi^j}\Lambda^{j - 1}$ (for $j \geq 1$) and $\frac{|\nabla \Lambda|^2}{\phi^{j - 2}}\Lambda^{j - 2}$ (for $j \geq 2$). This concludes the proof of \eqref{eq:P3MG12} as well as Lemma \ref{lemm:Pro6term}.
\end{proof}
\bigskip

\paragraph{Proof of items $(i)-(iii)$ of Proposition \ref{prop:dyn}.}
We have estimated the projections $\Pc_{n,M}$ and $\tilde \Pc_{n,M}$ of the all terms appearing in system \eqref{eq:LamUp}. In particular, taking the projection of \eqref{eq:LamUp} on $\binom{f_n}{g_n}$ and $\binom{\tilde f_n}{\tilde g_n}$ for $n \leq M$, we obtain for all $s \in [\tau, \tau_1]$:\\
- if $n = 0$ and $n = 1$, then
$$\left|\theta_n'(s) - \left(1 - \frac n2 \right)\theta_n(s) \right| \leq \frac{C}{s^2},$$
which is the conclusion of part $(i)$ of Proposition \ref{prop:dyn},\\
- if $n = 2$, then 
$$\left|\theta_2'(s) + \frac{2}{s}\theta_2(s) \right| \leq \frac{CA^3}{s^3},$$
which is the conclusion of part $(ii)$ of Proposition \ref{prop:dyn},\\
- if $3 \leq n \leq M$, then 
$$\left|\theta_n'(s) - \left(1 - \frac{n}{2}\right)\theta_n(s) \right| \leq \frac{CA^{n-1}}{s^{\frac{n+1}{2}}},$$
$$\left|\tilde \theta_n'(s) + \left(1 + \frac n2\right)\tilde \theta_n(s) \right| \leq \frac{CA^{n-1}}{s^{\frac{n+1}{2}}},$$
and $n = 0, 1, 2$,
$$\left|\tilde \theta_n'(s) + \left(1 + \frac n2 \right)\tilde \theta_n(s) \right| \leq \frac{C}{s^2}.$$
Integrating these differential equations between $\tau$ and $s$ yields the conclusion of part $(iii)$ of Proposition \ref{prop:dyn}.

\subsubsection{Control of the infinite dimensional part.} \label{sec:negetivepart}
We prove item $(iv)- (v)$ of Proposition \ref{prop:dyn} in this part. We proceed in three parts:\\
\noindent - In the first part, we project system \eqref{eq:LamUp} using the projector $\Pi_{-,M}$. Recall that $\Pi_{-,M}$ is the projector on the subspace of $\Hc$ where the spectrum is less than $\frac{1-M}{2}$.  Unlike the previous part where we used the spectrum of $\Hc + \Mc$, in this step, we use the spectrum of $\Hc$ and consider $\Mc$ as a perturbation. This is enough since for $M > 0$ large enough (see \eqref{eq:Mfixed}), the spectrum of $\Hc + \Mc$ is fully negative.\\
\noindent - In the second part, we collect all the estimates obtained in the first step to write  a system satisfied by $\binom{\Lambda_-}{\Upsilon_-}$, then we use a Gronwall's inequality to get the conclusion of item $(iv)$.\\
\noindent - In the third part, we prove item $(v)$ through a parabolic regularity argument as in \cite{TZpre15} (see also \cite{TZnor15}) applied to the system for $\binom{\Lambda_-}{\Upsilon_-}$. 

\paragraph{Part 1: Projection $\Pi_{-,M}$ of all the terms appearing in \eqref{eq:LamUp}.} In this part, we will find the main contribution to the projection $\Pi_{-,M}$ of the various terms appearing in \eqref{eq:LamUp}.

From the decomposition \eqref{decomoposeLamUp} and the fact that $\Pi_{-,M}\binom{f_n}{g_n} + \Pi_{-,M}\binom{\tilde f_n}{\tilde g_n} = 0$ for all $n \leq M$, we immediately obtain
\begin{equation*}
\Pi_{-,M}\left[\partial_s\binom{\Lambda}{\Upsilon} - (\Hc + \Mc)\binom{\Lambda}{\Upsilon}\right] = \partial_s\binom{\Lambda_-}{\Upsilon_-} - \big(\Hc+ \Mc\big)\binom{\Lambda_-}{\Upsilon_-}.
\end{equation*}

As for the potential term, we have the following estimates:
\begin{lemma}[Estimate of  $\Pi_{-,M}\left(V\binom{\Lambda}{\Upsilon}\right)$]  \label{lemm:Pineg3rdterm}$\;$\\
$(i)\;$ For all $s \geq 1$, we have 
\begin{align*}
\left\|\frac{\Pi_{-,M}(V_1 \Lambda + V_2\Upsilon)}{1 + |y|^{M+1}}\right\|_{L^\infty(\Rb)} &\leq \left(\|V_1\|_{L^\infty(\Rb)} + \frac C s \right)\left\|\frac{\Lambda_-}{1 + |y|^{M+1}}\right\|_{L^\infty(\Rb)}\\
& \quad  + \left(\|V_2\|_{L^\infty(\Rb)} + \frac C s \right)\left\|\frac{\Upsilon_-}{1 + |y|^{M+1}}\right\|_{L^\infty(\Rb)}\\
&\qquad + \sum_{n = 0}^M \frac{C}{s^\frac{M+1 - n}{2}}(|\theta_n(s)| + |\tilde \theta_n(s)|),\\
\left\|\frac{\Pi_{-,M}(V_3 \Lambda + V_4\Upsilon)}{1 + |y|^{M+1}}\right\|_{L^\infty(\Rb)} &\leq \left(\|V_3\|_{L^\infty(\Rb)} + \frac Cs \right)\left\|\frac{\Lambda_-}{1 + |y|^{M+1}}\right\|_{L^\infty(\Rb)}\\
& \quad + \left(\|V_4\|_{L^\infty(\Rb)} + \frac Cs \right)\left\|\frac{\Upsilon_-}{1 + |y|^{M+1}}\right\|_{L^\infty(\Rb)}\\
& \qquad + \sum_{n = 0}^M \frac{C}{s^\frac{M+1 - n}{2}}(|\theta_n(s)| + |\tilde \theta_n(s)|).
\end{align*}
$(ii)\;$ For all $A \geq 1$, there exists $s_8(A) \geq 1$ such that for all $s \geq s_8(A)$, if $\binom{\Lambda(s)}{\Upsilon(s)} \in \Vc_A(s)$, then
\begin{align*}
\left\|\frac{\Pi_{-,M}(V_1 \Lambda+ V_2 \Upsilon)}{1 + |y|^{M+1}}\right\|_{L^\infty(\Rb)} &\leq \|V_1\|_{L^\infty(\Rb)} \left\|\frac{\Lambda_-}{1 + |y|^{M+1}}\right\|_{L^\infty(\Rb)}\\
& \quad + \|V_2\|_{L^\infty(\Rb)} \left\|\frac{\Upsilon_-}{1 + |y|^{M+1}}\right\|_{L^\infty(\Rb)} + \frac{CA^M}{s^\frac{M+2}{2}},\\
\left\|\frac{\Pi_{-,M}(V_3 \Lambda + V_4\Upsilon)}{1 + |y|^{M+1}}\right\|_{L^\infty(\Rb)} &\leq \|V_3\|_{L^\infty(\Rb)} \left\|\frac{\Lambda_-}{1 + |y|^{M+1}}\right\|_{L^\infty(\Rb)}\\
& \quad + \|V_4\|_{L^\infty(\Rb)} \left\|\frac{\Upsilon_-}{1 + |y|^{M+1}}\right\|_{L^\infty(\Rb)} + \frac{CA^M}{s^\frac{M+2}{2}}.
\end{align*}

\end{lemma}
\begin{proof} See Lemma 5.12 in \cite{GNZpre16c} for a similar proof.
\end{proof}

For the nonlinear term, we claim the following:
\begin{lemma}[Estimate of  $\Pi_{-,M}\left(\binom{q}{p}\Lambda \Upsilon\right)$] \label{lemm:Pineg4thterm} Let $\binom{\Lambda(s)}{\Upsilon(s)} \in \Vc_A(s)$. Then for all $A \geq 1$ and $K_0 \geq 1$ introduced in \eqref{def:chi}, there exists $s_9(A,K_0) \geq 1$ such that for all $s \geq s_9(A,K_0)$, we have
\begin{equation*}
\quad \left\|\frac{\Pi_{-,M}\left[\binom{q}{p}\Lambda \Upsilon\right]}{1 + |y|^{M+1}} \right\|_{L^\infty(\Rb)} \leq \frac{CA^{2(M+2)}}{s^\frac{M + 3}{2}}.
\end{equation*}
\end{lemma}
\begin{proof} From part $(i)$ of Proposition \ref{prop:proSt}, we have the estimate
$$|\Lambda(y,s)\Upsilon(y,s)| \leq \frac{CA^{2(M+2)}}{s} s^{-\frac{M+1}{2}}(|y|^{M+1} + 1), \quad \forall y \geq \sqrt s.$$
For $|y| \leq \sqrt s$, we use the decomposition \ref{decomoposeLamUp} and part $(i)$ of Definition \ref{def:St} to write 
\begin{align*}
\Lambda\Upsilon &= \left[\sum_{i = 0}^M \alpha_i y^i + \Lambda_-\right]\left[\sum_{j = 0}^M \beta_j y^j + \Upsilon-\right]\\
& = \sum_{i + j = 0}^{M}\alpha_i \beta_j y^{i+j} + \Oc\left(A^{2(M+1)}s^{-\frac{M + 3}{2}}(|y|^{M+1} + 1)\right), \quad \forall |y| \leq \sqrt{s},
\end{align*}
where we used the fact that $|\alpha_i| + |\beta_i| \leq CA^is^{-\frac{i + 1}{2}}$. Note that for all polynomial functions $f(y)$ of degree $M$, we have $\Pi_{-,M}f(y) = 0$. The conclusion then follows from part $(iv)$ of Lemma \ref{lemm:prokernalL}. This ends the proof of Lemma \ref{lemm:Pineg4thterm}.
\end{proof}

For the error term, we use Lemma \ref{lemm:expandR1R2} to get the following estimates:
\begin{lemma}[Estimate for $\Pi_{-,M}\binom{R_1}{R_2}$.] \label{lemm:Pineg5th} The functions $R_1(y,s)$ and $R_2(y,s)$ defined by \eqref{def:Rys} satisfy 
\begin{equation*}
\quad \left\|\frac{\Pi_{-,M}\big[R_i(y,s)\big]}{1 + |y|^{M+1}} \right\|_{L^\infty(\Rb)} \leq \frac{C}{s^\frac{M+3}{2}}.
\end{equation*}

\end{lemma}
\begin{proof} Applying Lemma \ref{lemm:expandR1R2} with $m = \frac{M+2}{2}$, we write for all $|y| \leq \sqrt s$ and $s \geq 1$,
$$\left|R_i(y,s)- \sum_{k = 1}^{M/2}\frac{1}{s^{k + 1}}R_{i,k}(y) \right| \leq \frac{C(1 + |y|^{M + 2})}{s^\frac{M+4}{2}} \leq  \frac{C(1 + |y|^{M + 1})}{s^\frac{M+3}{2}}.$$
Since $\text{deg}(R_{i,k}) = 2k \leq M$, we have $\Pi_{-,M}R_{i,k} = 0$. The conclusion simply follows by using part $(iv)$ of Lemma \ref{lemm:prokernalL}. This ends the proof of Lemma \ref{lemm:Pineg5th}.
\end{proof}

We now turn to the estimate for the nonlinear gradient term. We claim the following: 
\begin{lemma}[Estimates for $\Pi_{-,M}\binom{G_1}{G_2}$] \label{lemm:Pineg6th} Under the assumption of Lemma \ref{lemm:expG1G2}, we have 
\begin{equation*}
\text{for} \;\; i = 1, 2, \quad \left\|\frac{\Pi_{-,M}\big[G_i(y,s)\big]}{1 + |y|^{M+1}} \right\|_{L^\infty(\Rb)} \leq \frac{C(A)}{s^\frac{M+3}{2}}.
\end{equation*}
\end{lemma}
\begin{proof} We only deal with the $G_1$ term because the estimate for $G_2$ follows similarly. From \eqref{est:G1chi}, \eqref{est:G12chi1} and part $(i)$ of Proposition \ref{prop:proSt}, we see that
$$\|G_1(s)\|_{L^\infty(\Rb)} \leq \frac{C(A)}{s}.$$
This immediately yields the estimate
$$|G_1(y,s)| \leq C(A) s^{-\frac{M+3}{2}}(|y|^{M+1} + 1), \quad \forall |y| \geq \sqrt s.$$
For $|y| \leq \sqrt{s}$, we recall that for all polynomial functions $f(y)$ of degree $M$, we have $\Pi_{-,M}f(y) = 0$. Hence, the conclusion follows once we show that there exists a polynomial function $G_{1,M}$ of degree $M$ in $y$ such that 
\begin{equation}\label{eq:estG1M}
|G_1(y,s) - G_{1,M}(y,s)| \leq C(A)s^{-\frac{M+3}{2}}(|y|^{M+1} + 1), \quad \forall |y| \leq \sqrt{s}.
\end{equation}
In particular, we take
$$G_{1,M} = \Pi_{+,M}\left\{\sum_{j = 1}^{M} \frac{1}{j!} \left[D_j \frac{|\nabla \phi|^2}{\phi^{j + 1}} \Lambda^j + D_{j - 1} \frac{2\nabla \Lambda \cdot \nabla \phi}{\phi^j}\Lambda^{j - 1} + D_{j - 2} \frac{2|\nabla \Lambda|^2}{\phi^{j - 2}}\Lambda^{j - 2} \right]\right\}.$$
Arguing as in the proof of Lemma \ref{lemm:Pro6term}, we deduce that the coefficient of degree $k \geq M+1$ of the polynomial 
$$\sum_{j = 1}^{M} \frac{1}{j!} \left[D_j \frac{|\nabla \phi|^2}{\phi^{j + 1}} \Lambda^j + D_{j - 1} \frac{2\nabla \Lambda \cdot \nabla \phi}{\phi^j}\Lambda^{j - 1} + D_{j - 2} \frac{2|\nabla \Lambda|^2}{\phi^{j - 2}}\Lambda^{j - 2} \right] - G_{1,M}$$  
is controlled by $\frac{CA^{k}}{s^{\frac{k+2}{2}}}$. Hence, for $|y| \leq \sqrt{s}$, 
\begin{align*}
&\left|\sum_{j = 1}^{M} \frac{1}{j!} \left[D_j \frac{|\nabla \phi|^2}{\phi^{j + 1}} \Lambda^j + D_{j - 1} \frac{2\nabla \Lambda \cdot \nabla \phi}{\phi^j}\Lambda^{j - 1} + D_{j - 2} \frac{2|\nabla \Lambda|^2}{\phi^{j - 2}}\Lambda^{j - 2} \right] - G_{1,M}\right|\\
& \qquad \leq \frac{CA^{M+2}}{s^\frac{M+3}2} (|y|^{M+1} + 1).
\end{align*}
According to the expansion \eqref{eq:expG1}, it remains to control $
\frac{|\Lambda|^{M+1}}{s} + \frac{|y|^2}{s^2}|\Lambda|^M + |\Lambda|^{M-1}|\nabla \Lambda|^2$. 
From Proposition \ref{prop:proSt}, we have $|\Lambda(y,s)| \leq \frac{C(A)\ln s}{s^2}(|y|^{M+1} + 1)$ and $|\Lambda(y,s)| + |\nabla \Lambda(y,s)| \leq \frac{C(A)}{\sqrt s}$ for all $y \in \Rb$. This implies that for $|y| \leq \sqrt s$, 
\begin{align*}
& \frac{|\Lambda|^{M+1}}{s} + \frac{|y|^2}{s^2}|\Lambda|^M + |\Lambda|^{M-1}|\nabla \Lambda|^2\\
& \leq C(A)\left(s^{-\frac{ M}{2} - 1} + s^{-\frac{M-1}{2} - 1} + s^{-\frac{M-2}{2} - 1} \right) \frac{\ln s}{s^2}(|y|^{M+1} + 1)\\
& \leq C(A)s^{-\frac{M+3}{2}}(|y|^{M+1} + 1). 
\end{align*}
This concludes the proof of \eqref{eq:estG1M} as well as Lemma \ref{lemm:Pineg6th}.
\end{proof}

\paragraph{Part 2: Proof of item $(iv)$ of Proposition \ref{prop:dyn}.} Applying the projection $\Pi_{-,M}$ to system \eqref{eq:LamUp} and using the various estimates given in the first step, we see that $\Lambda_-$ and $\Upsilon_-$ satisfy the following system:
\begin{align}
\partial_s \Lambda_- = \Lc_1 \Lambda_- + \frac qp \Upsilon_- + H_{1,-}(y,s) \label{eq:Lamneg}\\
\partial_s \Upsilon_- = \Lc_\mu \Upsilon_- + \frac pq \Lambda_-  + H_{2,-}(y,s) \label{eq:Upneg},
\end{align}
where $H_{1,-}$ and $H_{2,-}$ satisfy 
\begin{align*}
\left\|\frac{H_{1,-}(y,s)}{1 + |y|^{M+1}} \right\|_{L^\infty(\Rb)} &\leq \|V_1(s)\|_{L^\infty(\Rb)}\left\|\frac{\Lambda_-}{1 + |y|^{M+1}} \right\|_{L^\infty(\Rb)}\\
& \quad + \|V_2(s)\|_{L^\infty(\Rb)}\left\|\frac{\Upsilon_-}{1 + |y|^{M+1}} \right\|_{L^\infty(\Rb)}  + \frac{CA^M}{s^{\frac{M+2}{2}}} + \frac{C(A)}{s^\frac{M+3}{2}},
\end{align*}
and 
\begin{align*}
\left\|\frac{H_{2,-}(y,s)}{1 + |y|^{M+1}} \right\|_{L^\infty(\Rb)} &\leq \|V_3(s)\|_{L^\infty(\Rb)}\left\|\frac{\Lambda_-}{1 + |y|^{M+1}} \right\|_{L^\infty(\Rb)}\\
& \quad + \|V_4(s)\|_{L^\infty(\Rb)}\left\|\frac{\Upsilon_-}{1 + |y|^{M+1}} \right\|_{L^\infty(\Rb)}  + \frac{CA^M}{s^{\frac{M+2}{2}}} + \frac{C(A)}{s^\frac{M+3}{2}}.
\end{align*}
Using the integral formulation associated to the linear operator $\Lc_\eta$ with $\eta \in \{1, \mu\}$, we write for all $s \in [\tau, \tau_1]$,
\begin{align*}
\Lambda_-(s) &= e^{(s - \tau)\Lc_1}\Lambda_-(\tau) + \int_{\tau}^s e^{(s - s')\Lc_1}\left(\frac qp\Upsilon_-(s') + H_{1,-}(s')\right)ds'\\
\Upsilon_-(s) &= e^{(s - \tau)\Lc_\mu}\Upsilon_-(\tau) + \int_{\tau}^s e^{(s - s')\Lc_\mu}\left(\frac pq \Lambda_-(s')  + H_{2,-}(s') \right)ds'.
\end{align*}
Using part $(iii)$ of Lemma \ref{lemm:prokernalL}, we estimate
\begin{align*}
\left\|\frac{\Lambda_-(s)}{1 + |y|^{M+1}}\right\|_{L^\infty} &\leq e^{-\frac{M+1}{2}(s - \tau)}\left\|\frac{\Lambda_-(\tau)}{1 + |y|^{M+1}}\right\|_{L^\infty}\\
&+ \int_{\tau}^s e^{-\frac{M+1}{2}(s - s')} \left\{\frac qp  \left\|\frac{\Upsilon_-(s')}{1 + |y|^{M+1}}\right\|_{L^\infty} + \left\|\frac{H_{1,-}(y,s)}{1 + |y|^{M+1}} \right\|_{L^\infty}\right\} ds',
\end{align*}
and
\begin{align*}
\left\|\frac{\Upsilon_-(s)}{1 + |y|^{M+1}}\right\|_{L^\infty} &\leq e^{-\frac{M+1}{2}(s - \tau)}\left\|\frac{\Upsilon_-(\tau)}{1 + |y|^{M+1}}\right\|_{L^\infty}\\
&+ \int_{\tau}^s e^{-\frac{M+1}{2}(s - s')}  \left\{\frac pq \left\|\frac{\Lambda_-(s')}{1 + |y|^{M+1}}\right\|_{L^\infty}+ \left\|\frac{H_{2,-}(y,s)}{1 + |y|^{M+1}} \right\|_{L^\infty}\right\}ds'.
\end{align*} 
Introducing $\lambda(s) = \left\|\frac{\Lambda_-(s)}{1 + |y|^{M+1}}\right\|_{L^\infty}+\left\|\frac{\Upsilon_-(s)}{1 + |y|^{M+1}}\right\|_{L^\infty}$, then we have 
\begin{align*}
\lambda(s) &\leq e^{-\frac{M+1}{2}(s -\tau)}\lambda(\tau) +\int_{\tau}^s e^{-\frac{M+1}{2}(s - s')} \left(\frac{p}{q} + \frac{q}{p} + \sum_{i=1}^4\|V_i\|_{L^\infty}\right)\lambda(s')ds'\\
&+C\int_{\tau}^s e^{-\frac{M+1}{2}(s - s')} \left(\frac{C(A)}{{s'}^{\frac{M+3}{2}}} + \frac{A^M}{{s'}^\frac{M+2}{2}} \right)ds'.
\end{align*}
Since we have already fixed $M$ in \eqref{eq:Mfixed}, we then apply Lemma \ref{lemm:Gronwall} to deduce that 
$$e^{\frac{M+1}{2}s} \lambda(s) \leq e^{\frac{M+1}{4}(s - \tau)}e^{\frac{M+1}{2}\tau} \lambda(\tau) + Ce^{\frac{M+1}{2}s}\frac{A^M}{s^\frac{M+2}{2}},$$
which concludes the proof of part $(iv)$ of Proposition \ref{prop:dyn}.\\

\paragraph{Part 3: Proof of item $(v)$ of Proposition \ref{prop:dyn}.} In this part, we use the parabolic regularity of the semigroup associated to the linear operator $\Lc_\eta$ for $\eta \in \{1, \mu\}$ to prove item $(v)$ of Proposition \ref{prop:dyn}. Since the controls of $\nabla \Lambda_-$ and $\nabla \Upsilon_-$ are the same, we only deal with $\nabla \Lambda_-$. By \eqref{eq:Lamneg} and part $(i)$ of Definition \ref{def:St}, we have 
$$\partial_s \Lambda_- = \Lc_1 \Lambda_- + \tilde{H}_{1,-},$$
where 
$$ |\tilde{H}_{1,-}(y,s)| \leq \frac{CA^{M+1}}{s^{\frac{M+2}{2}}}(|y|^{M+1} + 1), \quad \forall y \in \Rb. $$
We then write 
$$ \Lambda_-(s) = e^{(s - s_0)\Lc_1}\Lambda_-(s_0) + \int_{s_0}^s e^{(s - s')\Lc_1}\tilde{H}_{1,-}(s') ds',$$
and 
$$
 |\nabla \Lambda_-(s)| \leq |\nabla e^{(s - s_0)\Lc_1}\Lambda_-(s_0)| + \int_{s_0}^s \left|\nabla e^{(s - s')\Lc_1}\tilde{H}_{1,-}(s')\right| ds'.
$$
We consider two cases:\\
- Case 1: $s \leq s_0 + 1$. We use parts $(v) - (vi)$ of Lemma \ref{lemm:prokernalL} and part $(i)$ of Proposition \ref{prop:uvt0} to estimate 
\begin{align*}
|\nabla \Lambda_-(y,s)| &\leq \frac{C}{s_0^{\frac{M+2}{2}}}(|y|^{M+1} + 1) + \frac{CA^{M+1}}{s^{\frac{M+2}{2}}}(|y|^{M+1} + 1)\int_{s_0}^s \frac{ds'}{\sqrt{1 - e^{-(s - s')}}}\\
&\leq \frac{CA^{M+1}}{s^{\frac{M+2}{2}}}(|y|^{M+1} + 1), \quad \forall y \in \Rb.
\end{align*}
- Case 2: $s > s_0 + 1$. We write for $s > s_0 + 1$, 
\begin{align*}
\nabla \Lambda_-(s) = \nabla e^{\Lc_1} \Lambda_-(s - 1) + \int_{s - 1}^s \nabla \left(e^{(s  -s')\Lc_1} \tilde{H}_{1,-}(s')\right) ds'.
\end{align*}
From part $(i)$ of Definition \ref{def:St}, we have 
$$|\Lambda_-(y,s - 1)| \leq \frac{A^{M+1}}{(s - 1)^\frac{M+2}{2}}(|y|^{M+1} + 1), \quad \forall y \in \Rb,$$
from which and part $(vi)$ of Lemma \ref{lemm:prokernalL}, we estimate
\begin{align*}
|\nabla \Lambda_-(y,s)| &\leq \frac{CA^{M+1}}{(s - 1)^{\frac{M+2}{2}} \sqrt{1 - e^{-1}}}(|y|^{M+1} + 1) + \frac{CA^{M+1}}{s^{\frac{M+2}{2}}}(|y|^{M+1} + 1)\int_{s_0}^s \frac{ds'}{\sqrt{1 - e^{-(s - s')}}}\\
&\leq \frac{CA^{M+1}}{s^{\frac{M+2}{2}}}(|y|^{M+1} + 1), \quad \forall y \in \Rb.
\end{align*}
This concludes the proof of item $(v)$ of Proposition \ref{prop:dyn}.

\subsubsection{Control of the outer part.}\label{sec:outerpart} 
We prove part $(vi)$ of Proposition \ref{prop:dyn} in this subsection. Let us write from \eqref{eq:LamUp} a system satisfied by $\tilde \Lambda_e = (1 - \chi(2y,s))\Lambda$ and $\tilde \Upsilon_e = (1 - \chi(2y,s))\Upsilon$ ($\chi$ is defined by \eqref{def:chi}):
\begin{align*}
\partial_s \tilde \Lambda_e &= \Lc_1 \tilde \Lambda_e - \tilde \Lambda_e + (1 - \chi(2y,s))\big(\tilde{F}_1(y,s) + R_1(y,s) + G_1(y,s) \big)\\
& - \Lambda(s)\left(\partial_s \chi(2y,s) + \Delta \chi(2y,s) + \frac{1}{2}y\cdot \nabla \chi(2y,s) \right) + 2\text{div}(\Lambda \nabla \chi(2y,s)),\\
\partial_s \tilde \Upsilon_e &= \Lc_\mu \tilde \Upsilon_e - \tilde \Upsilon_e + (1 - \chi(2y,s))\big(\tilde{F}_2( y,s) + R_2(y,s) + G_2(y,s) \big)\\
& - \Upsilon(s)\left(\partial_s \chi(2y,s) + \mu\Delta \chi(2y,s) + \frac{1}{2}y\cdot \nabla \chi(2y,s) \right) + 2\mu\text{div}(\Upsilon \nabla \chi(2y,s)),
\end{align*}
where 
$$\frac{1}{q}\tilde{F}_1 = \frac{1}{p}\tilde{F}_2 = \Lambda\Upsilon + \psi \Lambda + \phi \Upsilon.$$
Using the  semigroup representation of $\Lc_\eta$ with $\eta \in \{1, \mu\}$ and parts $(i) - (ii)$ of Lemma \ref{lemm:prokernalL}, we write for all $s \in [\tau, \tau_1]$,
\begin{align*}
\|\tilde \Lambda_e(s)\|_{L^\infty} &\leq e^{-(s - \tau)}\|\tilde \Lambda_e(\tau)\|_{L^\infty}\\
& + \int_{\tau}^s e^{-(s - s')}\left(\left\|(1 - \chi(2y,s'))\tilde{F}_1(s')\right\|_{L^\infty} + \left\|(1 - \chi(2y,s'))R_1(s')\right\|_{L^\infty} \right)ds'\\
& + \int_{\tau}^s e^{-(s - s')}\left\|\Lambda(s')\left(\partial_s \chi(2y,s') + \Delta \chi(2y,s') + \frac{1}{2}y\cdot \nabla \chi(2y,s') \right)  \right\|_{L^\infty}ds'\\
& + \int_{\tau}^s e^{-(s - s')} \frac{C}{\sqrt{1 - e^{-(s - s')}}}\|\Lambda(s')\nabla \chi(2y,s')\|_{L^\infty} ds',
\end{align*}
and 
\begin{align*}
\|\tilde \Upsilon_e(s)\|_{L^\infty} &\leq e^{-(s - \tau)}\|\tilde \Upsilon_e(\tau)\|_{L^\infty}\\
& + \int_{\tau}^s e^{-(s - s')}\left(\left\|(1 - \chi(2y,s'))\tilde{F}_2(s')\right\|_{L^\infty} + \left\|(1 - \chi(2y,s'))R_2(s')\right\|_{L^\infty} \right)ds'\\
& + \int_{\tau}^s e^{-(s - s')}\left\|\Upsilon(s')\left(\partial_s \chi(2y,s') + \mu\Delta \chi(2y,s') + \frac{1}{2}y\cdot \nabla \chi(2y,s') \right)  \right\|_{L^\infty}ds'\\
& + \int_{\tau}^s e^{-(s - s')} \frac{C}{\sqrt{1 - e^{-(s - s')}}}\|\Upsilon(s')\nabla \chi(2y,s')\|_{L^\infty} ds'.
\end{align*}
From the definition \eqref{def:chi} of $\chi$ and  part $(i)$ of Proposition \ref{prop:proSt}, we have 
\begin{align*}
&\left\|\Lambda(s')\left(\partial_s \chi(2y,s') + \Delta \chi(2y,s') + \frac{1}{2}y\cdot \nabla \chi(2y,s') \right)  \right\|_{L^\infty}\\
& + \left\|\Upsilon(s')\left(\partial_s \chi(2y,s') + \mu\Delta \chi(2y,s') + \frac{1}{2}y\cdot \nabla \chi(2y,s') \right)  \right\|_{L^\infty}\\
& \quad \leq C\left(\|\Lambda(s')\|_{L^\infty(|y| \leq K_0\sqrt {s'})} + \|\Upsilon(s')\|_{L^\infty(|y| \leq K_0\sqrt {s'})} \right) \leq \frac{CA^{M+1}}{\sqrt {s'}},
\end{align*}
and 
\begin{align*}
&\|\Lambda(s')\nabla \chi(2y,s')\|_{L^\infty} + \|\Upsilon(s')\nabla \chi(2y,s')\|_{L^\infty}\\
& \leq \frac{C}{K_0\sqrt {s'}}\left(\|\Lambda(s')\|_{L^\infty(|y| \leq K_0\sqrt {s'})} + \|\Upsilon(s')\|_{L^\infty(|y| \leq K_0\sqrt {s'})} \right) \leq \frac{CA^{M+1}}{s'}.
\end{align*}
Note from the proof of Lemma \ref{lemm:expandR1R2} that 
$$\left\|R_1(s')\right\|_{L^\infty} + \left\|R_2(s')\right\|_{L^\infty} \leq \frac{C}{s'}.$$
From the definitions \eqref{def:phipsi} of $\phi$ and $\psi$, we see that $|(1 - \chi(2y, s'))\phi(y,s')| + |(1 - \chi(2y, s'))\psi(y,s')| \leq \frac{1}{4}$ for $K_0$ large enough. By part $(i)$ of Proposition \ref{prop:proSt}, we derive
\begin{align*}
\left\|(1 - \chi(2y,s'))\tilde{F}_1(s')\right\|_{L^\infty} + \left\|(1 - \chi(2y,s'))\tilde{F}_2(s')\right\|_{L^\infty}& \leq  \frac{1}{2} \left(\|\tilde \Lambda_e(s')\|_{L^\infty} + \|\tilde \Upsilon_e(s')\|_{L^\infty} \right),
\end{align*}
for $K_0$ large enough. 
From \eqref{est:G12chi1}, we have 
$$\|(1 - \chi(2y, s'))G_1(s')\|_{L^\infty} + \|(1 - \chi(2y, s'))G_2(s')\|_{L^\infty} \leq \frac{C}{s'}. $$
Let $\lambda(s) = \|\tilde \Lambda_e(s)\|_{L^\infty} + \|\tilde \Upsilon_e(s)\|_{L^\infty}$, then we end up with 
\begin{align*}
\lambda(s) &\leq e^{-(s - \tau)}\lambda(\tau)\\
&+\int_{\tau}^s e^{-(s - s')}\left(\frac 12 \lambda(s') + \frac{CA^{M+1}}{\sqrt{s'}} + \frac{CA^{M+1}}{s' \sqrt{1 - e^{-(s - s')}}}\right)ds'.
\end{align*}
Applying Lemma \ref{lemm:Gronwall} yields
$$\lambda(s) \leq e^{-\frac 12 (s - \tau)}\lambda(\tau) + \frac{CA^{M+1}}{\sqrt{s}}(s - \tau + \sqrt{s - \tau}).$$
Since $\text{supp}(1 - \chi(y,s)) \subset \text{supp}(1 - \chi(2y,s))$, we have $\|\Lambda_e\|_{L^\infty} \leq \|\tilde \Lambda_e\|_{L^\infty}$ and $\|\Upsilon_e\|_{L^\infty} \leq \|\tilde \Upsilon_e\|_{L^\infty}$. This concludes the proof of part $(vi)$ of Proposition \ref{prop:dyn}.

\subsection{A priori estimates in $\Dc_2$ and $\Dc_3$.}
In this section, we estimate directly the solution of system \eqref{PS} through a classical parabolic regularity argument. Note that this part corresponds to Section 4.1 in \cite{GNZjde17} (see also Section 4 in \cite{MZnon97}). Note also that the mentioned papers deal with a single equation, however, it can be naturally extended to system \eqref{PS} without any difficulties. For the sake of completeness, we will sketch the proof. 

We have the following \textit{a priori} estimates in $\Dc_2$. 

\begin{proposition}[\textbf{A priori estimate in $\Dc_2$}]\label{prop:D2}
There exists $K_{0,3} > 0$ such that for all $K_0 \geq K_{0,3}$, $\delta_1 \leq 1$, $\xi_0 \gg 1$ and $C_{0,1}^* > 0$, $C_{0,2}^* > 0$, we have the following property: Assume that $(\tilde{u}, \tilde v)$ is a solution to the system 
\begin{equation}\label{eq:UcD2}
\partial_\tau \tilde{u} = \Delta \tilde{u} + e^{p \tilde{v}}, \quad \partial_\tau \tilde{v} = \Delta \tilde{v} + e^{q \tilde{u}},
\end{equation}
for $\tau \in [\tau_1, \tau_2]$ with $0 \leq \tau_1 \leq \tau_2 \leq 1$. Assume in addition, for all $\tau \in [\tau_1, \tau_2]$,
\begin{itemize}
\item[(i)] for all $|\xi| \leq 2\xi_0$, 
$$|\tilde u(\xi, \tau_1) - \hat u(\tau_1)| + |\tilde v(\xi, \tau_1) - \hat v(\tau_1)|\leq \delta_1, \quad |\nabla \tilde{u}(\xi, \tau_1)| + |\nabla \tilde{v}(\xi, \tau_1)|  \leq \frac{C_{0,1}^*}{\xi_0},$$ where $\hat u(\tau)$ and $\hat v (\tau)$ are given by \eqref{def:solUc},
\item[(ii)] for all $|\xi| \leq \frac 74\xi_0$, $|\nabla \tilde u(\xi, \tau)| + |\nabla \tilde v(\xi, \tau)|  \leq \frac{C_{0,2}^*}{\xi_0}$.
\item[(iii)] for all $|\xi|\leq  \frac 74\xi_0$, $\tilde u(\xi, \tau) \leq \frac{1}{2}\hat u(\tau)$ and $\tilde v(\xi, \tau) \leq \frac{1}{2}\hat v(\tau)$.
\end{itemize}
Then, for $\xi_0 \geq \xi_{0,3}(C_0^*)$, there exists $\epsilon = \epsilon(K_0, C_{0,2}^*, \delta_1, \xi_0)$ such that for all $|\xi| \leq \xi_0$ and $\tau \in [\tau_1, \tau_2]$,
$$|\tilde u(\xi, \tau) - \hat u(\tau)| + |\tilde v(\xi, \tau) - \hat v(\tau)|  \leq \epsilon, \quad |\nabla \tilde u(\xi, \tau)| + |\nabla \tilde v(\xi, \tau)|  \leq \frac{2C_{0,1}^*}{\xi_0},$$
where $\epsilon \to 0$ as $(\delta_1, \xi_0) \to (0, +\infty)$.
\end{proposition}
\begin{proof} We first deal with the gradient estimate. Let $\theta = |\nabla \tilde u|^2 + |\nabla \tilde v|^2$, then we write from \eqref{eq:UcD2}, 
$$\partial_\tau \theta \leq \Delta \theta + C \theta,$$
where we used the fact that $2 \nabla f \cdot\nabla (\Delta f) \leq \Delta (|\nabla f|^2)$ and the boundedness of $e^{p\tilde v}$ and $e^{q\tilde{u}}$. \\
Consider $\varphi_1 \in \mathcal{C}^\infty(\Rb^N)$ such that $\varphi_1 \in [0,1]$, $\varphi_1(\xi) =1$ for $|\xi| \leq \frac{3}{2}\xi_0$ and $\varphi_1(\xi) = 0$ for $|\xi| \geq \frac{7}{4}\xi_0$, $|\nabla \varphi_1(\xi)| \leq \frac{1}{\xi_0}$ and $|\Delta \varphi_1(\xi)| \leq \frac{1}{\xi_0^2}$. Then, $\theta_1 = \varphi_1 \theta$ satisfies
$$\partial_\tau \theta_1 \leq \Delta \theta_1 + C(C_{0,2}^*) \xi_0^{-2} \mathbf{1}_{\{\frac 32 \xi_0 \leq |\xi| \leq 2\xi_0\}} + C\theta_1.$$
Let $\theta_2 = e^{-C\tau}\theta_1$, we write
$$\partial_\tau \theta_2 \leq \Delta \theta_2 + C(C_{0,2}^*)\xi_0^{-2} \mathbf{1}_{\{\frac 32 \xi_0 \leq |\xi| \leq 2\xi_0\}}, \quad 0 \leq \theta_2(\tau_1) \leq \frac{{C_{0,1}^*}^2}{\xi_0^2}.$$
By the maximum principle, we deduce 
$$\forall|\xi | \leq \frac{5}{4} \xi_0, \;\; \tau \in[\tau_1, \tau_2], \quad \theta (\xi, \tau) \leq \frac{{C_{0,1}^*}^2 + C(C_{0,2}^*)^2e^{-C'\xi_0^2}}{\xi_0^2} \leq \frac{2C_{0,1}^*}{\xi_0^2},$$
for $\xi_0 \geq \xi_{0,3}(C_{0,2}^*)$, which yields the conclusion.

We now turn to the estimates on $\tilde{u}$ and $\tilde{v}$. Let us consider $\tilde{u}_1$ and $\tilde{v}_1$ a solution of system \eqref{eq:UcD2} such that for all $|\xi| \leq 2$ and $\tau \in [\tau_1, \tau_2]$:
$$|\tilde{u}_1(\xi, \tau_1) - \hat{u}(\tau_1)| +|\tilde{v}_1(\xi, \tau_1) - \hat{v}(\tau_1)| \leq \delta_1, \quad |\nabla \tilde{u}_1(\xi, \tau)| + |\nabla \tilde{v}_1(\xi, \tau)| \leq \epsilon,$$
where $\hat u$ and $\hat v$ are defined as in \eqref{def:solUc}. Let us show that for all $|\xi| \leq 2$ and $\tau \in [\tau_1, \tau_2]$:  
$$|\tilde{u}_1(\xi, \tau) - \hat{u}(\tau)| +|\tilde{v}_1(\xi, \tau) - \hat{v}(\tau)| \leq C(K_0)\epsilon + \delta_1,$$
where $C(K_0)$ is independent from $\epsilon$.\\
We have for all $\tau \in [\tau_1, \tau_2]$,
$$\tilde{u}_1(0,\tau) = \frac{1}{|B_2(0)|} \int_{|\xi| \leq 2} \tilde{u}_1(\xi, \tau) d\xi  + \tilde{u}_2(\tau), \quad \tilde{v}_1(0,\tau) = \frac{1}{|B_2(0)|} \int_{|\xi| \leq 2} \tilde{v}_1(\xi, \tau) d\xi  + \tilde{v}_2(\tau),$$
and 
$$e^{q\tilde{u}_1(0,\tau)} = \frac{1}{|B_2(0)|} \int_{|\xi| \leq 2} e^{q\tilde{u}_1(\xi, \tau)} d\xi  + \tilde{u}_3(\tau), \quad e^{p\tilde{v}_1(0,\tau)} = \frac{1}{|B_2(0)|} \int_{|\xi| \leq 2} e^{p\tilde{v}_1(\xi, \tau)} d\xi  + \tilde{v}_3(\tau),$$
where $|B_2(0)|$ is the volume of the sphere of radius $2$ in $\Rb^N$, $\|\tilde{u}_i\|_{L^\infty} + \|\tilde{v}_i\|_{L^\infty} \leq C\epsilon$ for $i = 2,3$. \\
For $\epsilon$ small, we consider in the distribution sense,
$$\tilde{U}(\tau) = \frac{1}{|B_2(0)|} \int_{|\xi| \leq 2} \tilde{u}_1(\xi, \tau) d\xi, \quad \tilde{V}(\tau) = \frac{1}{|B_2(0)|} \int_{|\xi| \leq 2} \tilde{v}_1(\xi, \tau) d\xi,$$
then we have from \eqref{eq:UcD2}, 
$$e^{p\tilde{V}} - C\epsilon \leq \frac{d\tilde{U}}{d\tau} \leq e^{p \tilde{V}} + C\epsilon, \quad e^{q\tilde{U}} - C\epsilon \leq \frac{d\tilde{V}}{d\tau} \leq e^{p \tilde{U}} + C\epsilon,$$
and 
$$|\tilde{U}(\tau_1) - \hat u(\tau_1)| + |\tilde{V}(\tau_1) - \hat v(\tau_1)| \leq C\epsilon + \delta_1.$$
We obtain by a classical \textit{a priori} estimates that for all $\tau \in [\tau_1, \tau_2]$, $|\tilde{U}(\tau) - \hat u(\tau)| + |\tilde{V}(\tau) - \hat v(\tau)| \leq C(K_0)\epsilon + \delta_1$ (since $C_1 \leq |\hat u(\tau)| + |\hat v (\tau)| \leq C_1'(K_0)$). Therefore, for all $|\xi| \leq 2$ and $\tau \in[\tau_1, \tau_2]$, we have $|\tilde{u}(\xi,\tau) - \hat u(\tau)| + |\tilde{v}(\xi,\tau) - \hat v(\tau)| \leq C(K_0)\epsilon + \delta_1$. Applying this result to $\tilde{u}'(\xi, \tau) = \tilde{u}(\xi - \bar \xi_0, \tau)$ and $\tilde{v}'(\xi, \tau) = \tilde{v}(\xi - \bar \xi_0, \tau)$ for $|\bar{\xi}_0| \leq \xi_0 - 2$ with $\xi_0 \gg 1$, from the assumption and the gradient estimates proved in the previous step, i.e. $|\nabla \tilde u(\xi, \tau)| + |\nabla \tilde v(\xi, \tau)|  \leq \frac{2C_{0,1}^*}{\xi_0}$, we end-up with 
$$\forall |\xi| \leq \xi_0, \;\; \tau \in [\tau_1, \tau_2], \quad |\tilde{u}(\xi,\tau) - \hat u(\tau)| + |\tilde{v}(\xi,\tau) - \hat v(\tau)| \leq \epsilon,$$
where $\epsilon = \epsilon(\delta_1, \xi_0) \to 0$ as $(\delta_1, \xi_0) \to (0, +\infty)$. This concludes the proof of Proposition \ref{prop:D2}.
\end{proof}

For the \textit{a priori} estimates in $\Dc_3$, we have the following:
\begin{proposition}[\textbf{A priori estimate in $\Dc_3$}]\label{prop:D3}
For all $\epsilon > 0$, $\epsilon_0 > 0$, $\sigma_0 > 0$, there exists $t_{0,4}(\epsilon, \epsilon_0, \sigma_0) < T$ such that for all $t_0 \in [t_{0,4}, T)$, if $(u,v)$ is a solution of \eqref{PS} on $[t_0,t_*]$ for some $t_* \in [t_0, T)$ satisfying 
\begin{itemize}
\item[(i)] for all $|x| \in \left[\frac{\epsilon_0}6, \frac{\epsilon_0}{4}\right]$ and $t \in [t_0, t_*]$,
\begin{equation}
i = 0, 1, \quad |\nabla^i u(x,t)| + |\nabla^i v(x,t)| \leq \sigma_0,
\end{equation}
\item[(ii)] For $|x| \geq \frac{\epsilon_0}{6}$, $u(x,t_0) = \hat u^*(x)$ and $v(x,t_0) = \hat v^*(x)$ where $\hat u^*$ and $\hat v^*$ are defined in \eqref{def:ustar} and \eqref{def:vstar} respectively.
\end{itemize}
Then for all $|x| \in \left[\frac{\epsilon_0}{4}, +\infty\right)$ and $t \in [t_0, t_*]$, 
\begin{equation}
i = 0, 1, \quad |\nabla ^iu(x,t) - \nabla ^i u(x,t_0)| + |\nabla ^i v(x,t) - \nabla ^i v(x,t_0)|  \leq \epsilon.
\end{equation}
\end{proposition}
\begin{proof} The proof follows from a standard parabolic regularity argument. We refer the interested reader to Proposition 4.3 in \cite{GNZjde17} for a similar proof.
\end{proof}

\subsection{Conclusion of the proof of Proposition \ref{prop:redu}.}

In this subsection we complete the proof of Proposition \ref{prop:redu}. We will show that we can choose the parameters $K_0, \delta_0, C_0$ independently from $A$, where $A$ is fixed large enough. Then we choose the parameter $\epsilon_0, \alpha_0, \eta_0, s_0$ in term of $A$ such that all the bounds given in Definition \ref{def:St} are improved, except for the components $\theta_0$ and $\theta_1$. This concludes the proof of part $(i)$ of Proposition \ref{prop:redu}. Part $(ii)$ is just a direct consequence of the dynamics on the components $\theta_0$ and $\theta_1$ given in Proposition \ref{prop:dyn}. \\

\noindent - \textit{Proof of part $(i)$ of Proposition \ref{prop:redu}.} For the proof of the improved bounds in $\Dc_1$, we have the following: for all $s \in [s_0, s_1]$, 
$$\|\Lambda_e(s)\|_{L^\infty(\Rb)} +  \|\Upsilon_e(s)\|_{L^\infty(\Rb)} \leq \frac{A^{M+2}}{2\sqrt{s}},$$
$$ \left\|\frac{\Lambda_-(y,s)}{1 + |y|^{M+1}} \right\|_{L^\infty(\Rb)} +  \left\|\frac{\Upsilon_-(y,s)}{1 + |y|^{M+1}} \right\|_{L^\infty(\Rb)} \leq \frac{A^{M+1}}{2s^{\frac{M+2}{2}}},$$
$$ \left\|\frac{\nabla \Lambda_-(y,s)}{1 + |y|^{M+1}} \right\|_{L^\infty(\Rb)} +  \left\|\frac{\nabla \Upsilon_-(y,s)}{1 + |y|^{M+1}} \right\|_{L^\infty(\Rb)} \leq \frac{A^{M+2}}{2s^{\frac{M+2}{2}}},$$
\begin{equation*}
 |\theta_j(s)|\leq \frac{A^j}{2s^\frac{j+1}{2}},\quad |\tilde{\theta}_j(s)| \leq \frac{A^j}{2s^\frac{j+1}{2}} \;\; \text{for}\;\; 3\leq j\leq M, 
\end{equation*}
$$ |\tilde \theta_i(s)| \leq \frac{A^{2}}{2s^2}\;\; \text{for}\;\; i = 0, 1,2, \quad |\theta_2(s)| < \frac{A^4 \ln s}{s^2}.$$
Since the proof of these estimates uses the same argument as in  Section 5.2.1 of \cite{GNZpre16c} through the dynamics of the solution given in Proposition \ref{prop:dyn}, therefore we omit it here. \\
For the improved control on $\Dc_2$, we use the following result:
\begin{lemma} \label{lemm:44} Under the hypothesis of Proposition \eqref{prop:redu}, we have  for all 
$$|x| \in \left[ \frac{K_0}{4}\sqrt{(T-t_*)|\ln(T-t_*)|},\epsilon_0 \right],$$ 
$(i)$ For all $|\xi| \leq \frac{7}{4}\alpha_0 \sqrt{|\ln \sigma(x)|}$ and $\tau \in \left[\max \left\{0, \frac{t_0 -t(x)}{\sigma(x)} \right\}, \frac{t_* - t(x)}{\sigma(x)} \right]$, 
$$|\nabla_\xi \tilde{u}(x, \xi, \tau)| + |\nabla_\xi \tilde{v}(x, \xi, \tau)| \leq \frac{2 C_0}{\sqrt{|\ln \sigma(x)|}}, \quad \tilde{u}(x, \xi, \tau) \leq \frac{1}{2}\hat u(\tau), \quad \tilde{v}(x, \xi, \tau) \leq \frac{1}{2}\hat v(\tau).$$

\noindent $(ii)$ For all $|\xi| \leq 2\alpha_0 \sqrt{|\ln \sigma(x)|}$ and $\tau = \max \left\{0, \frac{t_0 -t(x)}{\sigma(x)} \right\}$ and for all $\delta_1 \leq 1$, 
$$ |\tilde{u}(x, \xi, \tau) - \hat u(\tau)| + |\tilde{v}(x, \xi, \tau) - \hat v(\tau)| \leq \delta_1, \quad |\nabla_\xi \tilde{u}(x, \xi, \tau)| + |\nabla_\xi \tilde{v}(x, \xi, \tau)| \leq \frac{C_0}{4 \sqrt{|\ln \sigma(x)|}}.  $$
\end{lemma}
\begin{proof} See Lemma 4.4 in \cite{MZnon97} for an analogous proof. 
\end{proof}
From Lemma \ref{lemm:44}, we apply Proposition \ref{prop:D2} with $C_{0,1}^* = \frac{C_0}{4}$, $C_{0,2}^* = 2C_0$, $\xi_0 = \alpha_0 \sqrt{|\ln \sigma(\epsilon_0)|}$ for $\alpha_0 \in (0,1)$ to derive 
$$|\tilde{u}(x, \xi, \tau_*) - \hat u(\tau_*)| + |\tilde{v}(x, \xi, \tau_*) - \hat v(\tau_*)| \leq \frac{\delta_0}{2}, \quad |\nabla_\xi \tilde{u}(x, \xi, \tau_*)| + |\nabla_\xi \tilde{v}(x, \xi, \tau_*)| \leq \frac{C_0}{2\sqrt{|\ln \sigma(x)|}},$$
which concludes the improved control in $\Dc_2$.

For the proof of the improved bounds on $\Dc_3$, we note from the choice of initial data \eqref{def:uvt0} that the hypothesis of Proposition \ref{prop:D3} holds. We then apply Proposition \ref{prop:D3} with $\epsilon = \eta_0/2$ to obtain the estimate for all $t \in [t_0, t_*]$ and $|x| \geq \frac{\epsilon_0}{4}$, 
$$ i =0, 1, \quad |\nabla ^i u(x,t) - \nabla^i u(x,t_0)| + |\nabla ^i v(x,t) - \nabla^i v(x,t_0)| \leq \frac{\eta_0}{2}.$$  
This completes the proof of part $(i)$ of Proposition \ref{prop:redu}.\\

\noindent - \textit{Proof of part $(ii)$ of Proposition \ref{prop:redu}.} This is just a consequence of the dynamics of the components $\theta_0$ and $\theta_1$. Indeed, from part $(i)$ of Proposition \ref{prop:redu}, we know that for $n = 0$ or $1$ and $\omega = \pm 1$, we have $\theta_n(s_1) = \omega \frac{A}{s_1^2}$. From part $(i)$ of Proposition \ref{prop:dyn}, we see that 
$$\omega \theta_n'(s_1) \geq \left(1 - \frac{n}{2}\right)\omega\theta_n(s_1) - \frac{C}{s_1^2} \geq \frac{(1 - n/2)A - C}{s_1^2}.$$
Taking $A$ large enough gives $\omega \theta_n'(s_1) > 0$, which means that $\theta_n$ is traversal outgoing to the bounding curve $s \mapsto \omega As^{-2}$ at $s = s_1$. This completes the proof of Proposition \ref{prop:redu}. \hfill $\square$

\bigskip

\noindent \textbf{Acknowledgments:} The authors would like to thank the anonymous referees for their careful reading and their suggestions to improve the presentation of the paper.  

\appendix

\section{Some technical results used in the proof of Theorem \ref{theo1}.}
The following lemma is an integral version of Gronwall's inequality:
\begin{lemma}[A Gronwall's inequality]\label{lemm:Gronwall} If $\lambda(s)$, $\alpha(s)$ and $\beta(s)$ are continuous defined on $[s_0, s_1]$ such that 
$$\lambda(s) \leq \lambda(s_0) + \int_{s_0}^s\alpha(\tau) \lambda(\tau) d\tau + \int_{s_0}^{s}\beta(\tau)d\tau, \quad s_0 \leq s\leq s_1,$$
then 
$$\lambda(s) \leq \exp\left(\int_{s_0}^s\alpha(\tau) d\tau \right) \left[\lambda(s_0) + \int_{s_0}^s \beta(\tau) \exp \left(-\int_{s_0}^\tau \alpha(\tau')d\tau' \right)d\tau\right].$$  
\end{lemma}
\begin{proof} See Lemma 2.3 in \cite{GKcpam89} for an example of the proof.

\end{proof}

In the following lemma, we recall some linear regularity estimates of the linear operator $\Lc_\eta$ defined in \eqref{def:HLM}:

\begin{lemma}[Properties of the semigroup $e^{\tau \Lc_\eta}$]\label{lemm:prokernalL} The kernel $e^{\tau \Lc_\eta}(y,x)$ of the semigroup $e^{\tau \Lc_\eta}$ is given by
\begin{equation}\label{def:kernelE}
e^{\tau \Lc_\eta}(y,x) =  \frac{1}{\big[4\pi(1 - e^{-\tau})\big]^{N/2}}\exp \left(-\frac{|ye^{-\tau/2} - x|^2}{4\eta (1 - e^{\tau})} \right), \quad \forall \tau > 0,
\end{equation}
and $e^{\tau \Lc_\eta}$ is defined by
\begin{equation}\label{def:semieL}
e^{\tau \Lc_\eta}g(y) = \int_{\RN}e^{\tau \Lc_\eta}(y,x)g(x)dx.
\end{equation}
We have the following estimates:

\noindent $(i)\;$ $\left\|e^{\tau \Lc_\eta}g\right\|_{L^\infty(\RN)} \leq \|g\|_{L^\infty(\RN)}$ for all $g \in L^\infty(\RN)$,\\

\noindent $(ii)\;$ $\left\|e^{\tau \Lc_\eta}\, \text{div} (g)\right\|_{L^\infty(\RN)} \leq \frac{C}{\sqrt{1 - e^{-\tau}}}\|g\|_{L^\infty(\RN)}$ for all $g \in L^\infty(\RN)$,\\

\noindent $(iii)\;$ If $|g(x)| \leq c(1 + |x|^{M+1})$ for all $x \in \RN$, then 
$$\left|e^{\tau \Lc_\eta} \Pi_{-,M}(g(y))\right| \leq Cce^{-\frac{(M+1)\tau}{2}} (1 + |y|^{M+1}), \quad \forall y \in \RN.$$

\noindent $(iv)\;$ For all $k \geq 0$, we have 
\begin{equation*}
\left\|\frac{\Pi_{-,M}(g)}{1 + |y|^{M+k}} \right\|_{L^\infty(\RN)} \leq C\left\|\frac{g}{1 + |y|^{M + k}} \right\|_{L^\infty(\RN)}.
\end{equation*}

\noindent $(v)\;$ If $|\nabla g(x)| \leq D(1 + |x|^m)$ for all $x \in \Rb^N$, then 
$$|\nabla \left(e^{\tau \Lc_\eta} g\right)(y)| \leq CDe^{\frac{\tau}{2}}(1 + |y|^m), \quad \forall y \in \Rb^N.$$

\noindent $(vi)\;$ If $|g(x)| \leq D(1 + |x|^m)$ for all $x \in \Rb^N$, then 
$$|\nabla \left(e^{\tau \Lc_\eta} g\right)(y)| \leq CD\frac{e^{\frac{\tau}{2}}}{\sqrt{1 - e^{-\tau}}}(1 + |y|^m), \quad \forall y \in \Rb^N.$$

\end{lemma}
\begin{proof} The expressions of $e^{\tau \Lc_\eta}(y,x)$ and $e^{\tau \Lc_\eta}$ are given in \cite{BKnon94}, page 554. For item $(i)-(ii)$ and $(v)-(vi)$, see Lemma 4.15 in \cite{TZpre15}. For item $(iii)-(iv)$, see Lemmas A.2 and A.3 in \cite{MZjfa08}.
\end{proof}


\def\cprime{$'$}


\bigskip

\end{document}